\documentclass[twoside, 10pt, a4paper]{amsart}

\usepackage{amssymb} \usepackage{amsfonts} \usepackage{amsmath}
\usepackage{thmtools}
\usepackage{epsfig} 
\usepackage{color}
\usepackage{pinlabel}
\usepackage{mathtools}
\usepackage{mathrsfs}  
\usepackage{amsthm}
\usepackage{natbib}
\usepackage{amscd}
\usepackage[all]{xy}
\usepackage[font=small,labelfont=bf]{caption}

\usepackage{tcolorbox}
\usepackage{listings}
\usepackage{graphicx}
\usepackage{tikz}

\usepackage[margin=30mm]{geometry}
\usepackage{tikz-cd}
\usepackage[shortlabels]{enumitem}
\usepackage{hyperref}
\usepackage{float}
\usepackage{cleveref}

\hypersetup{hypertexnames=false}
\newtheorem{teo}{Theorem}[section]
\newtheorem{lemma}[teo]{Lemma}
\newtheorem{prop}[teo]{Proposition}
\newtheorem{cor}[teo]{Corollary}

\theoremstyle{definition}
\newtheorem{defn}[teo]{Definition}

\newtheorem{example}[teo]{Example}
\newtheorem{assum}{Assumption}
\theoremstyle{remark}
\newtheorem{rem}[teo]{Remark}

\newtheoremstyle{named}{}{}{\itshape}{}{\bfseries}{}{ }{\thmname{#1}  \thmnote{\normalfont~(#3)}\bfseries.}
\theoremstyle{named}
\newtheorem*{namedconjecture}{L-space knot conjecture}

\newtheoremstyle{named1}{}{}{\itshape}{}{\bfseries}{.}{ }{\thmname{#1}}
\theoremstyle{named1}
\newtheorem*{namedconjecture1}{L-space conjecture}

\newtheoremstyle{namedd}{}{}{\itshape}{}{\bfseries}{.}{ }{#1 \thmnote{#3}}
\theoremstyle{namedd}
\newtheorem*{namedtheorem}{Theorem}
\newtheorem*{namedcorollary}{Corollary}

\newtheoremstyle{named}{}{}{}{}{\bfseries}{.}{ }{\thmname{#1}}
\theoremstyle{named}
\newtheorem*{namedassumption}{Standing assumption}

\newcommand{\matR} {\ensuremath {\mathbb{R}}}
\newcommand{\matQ} {\ensuremath {\mathbb{Q}}}
\newcommand{\matZ} {\ensuremath {\mathbb{Z}}}

\newcommand{\matH} {\ensuremath {\mathbb{H}}}

\newcommand{\matD} {\ensuremath {\mathbb{D}}}

\newcommand{\Z}{\ensuremath {\mathbb{Z}}}
\newcommand{\Q}{\ensuremath {\mathbb{Q}}}

\usepackage{newpxtext}

\makeatletter
\newcommand{\ie}{\emph{i.e.}\@ifnextchar.{\!\@gobble}{}}
\newcommand{\eg}{\emph{e.g.}\@ifnextchar.{\!\@gobble}{}}
\newcommand{\etc}{etc\@ifnextchar.{}{.\@}}
\makeatother

\usepackage{mathtools}

\author{Diego Santoro}
\address{Faculty of Mathematics, University of Vienna}
\email{diego.santoro95@gmail.com}

\title{Taut foliations from knot diagrams}

\begin{document}

\begin{abstract} 
We prove that knots in $S^3$ satisfying certain diagrammatic properties are persistently foliar. As a consequence, all their non-trivial surgeries support coorientable taut foliations.
All non-fibered two-bridge knots, as well as many pretzel and Montesinos knots, satisfy these diagrammatic properties. More generally, our result applies to arborescent knots defined by weighted planar trees with more than one vertex, provided that all weights have absolute value greater than one and at least one weight has absolute value greater than two.
We also use this result to provide sufficient conditions for closures of 3–braids, and more generally odd–strand braids, to be persistently foliar.

The ideas involved in the proof extend to links: as an application, we show that for all surgeries $M$ on the Borromean rings one has that $M$ is not an $L$-space if and only if it supports a coorientable taut foliation.
\end{abstract}

\maketitle

\section{Introduction}\label{sec: intro}
A codimension-$1$ foliation of a closed $3$-manifold $M$ is said to be \emph{taut} if every leaf meets a closed transversal, that is, a smooth simple closed curve in $M$ everywhere transverse to the foliation.
Throughout this paper we work with $C^{\infty,0}$ foliations: roughly speaking, these are foliations $\mathcal{F}$ whose leaves are smoothly immersed, with the tangent planes varying continuously so as to define a subbundle $T\mathcal{F}\subset TM$.
We further restrict to \emph{coorientable} foliations, meaning those for which the line bundle $TM/T\mathcal{F}$ is orientable. When $M$ itself is orientable, this is equivalent to requiring that $T\mathcal{F}$ is an orientable plane bundle.

Taut foliations have long been a classical and fruitful tool in the study of $3$-dimensional manifolds \cite{Th, Gabai, Gabai2, Gabai3},  and they endow manifolds that admit them with interesting topological properties. For instance, if a closed orientable $3$-manifold $M\ne S^2\times S^1$ supports a coorientable taut foliation, then $M$ has infinite fundamental group \cite{Novikov}, is irreducible \cite{Rosenberg}, and its universal cover is diffeomorphic to $\mathbb{R}^3$ \cite{Palmeira}. The existence of coorientable taut foliations is of significant current interest also on account of the following conjecture.



\begin{namedconjecture1}\label{L space conjecture} 
For an irreducible oriented rational homology $3$-sphere $M$, the following are
equivalent:
\begin{enumerate}
    \item  $M$ is left-orderable, i.e, $\pi_1(M)$ is left-orderable.
    \item  $M$ is not an $L$-space.
    \item  $M$ supports a cooriented taut foliation.
\end{enumerate}
\end{namedconjecture1} 

Boyer, Gordon, and Watson postulated the equivalence of $(1)$ and $(2)$ in \cite{BGW}. The question of whether $(2)$ and $(3)$ are equivalent arose informally after $(3)\Rightarrow(2)$ was established \cite{OS, Bow, KR2}, and was first stated explicitly in \cite{J}.
This conjecture posits deep connections among geometric, dynamical, Floer homological, and algebraic properties of $3$-manifolds, and substantial evidence has accumulated in its favor. For instance, it has been verified for all graph manifolds, \ie, those whose JSJ decomposition consists only of Seifert fibered pieces \cite{BC,BGW,BNR,CLW,EHN,HRRW,LS}. Furthermore, Li has recently established that $(1)$ implies $(3)$ for manifolds of Heegaard genus two \cite{L1}.
\newline

Every $3$-manifold arises via Dehn surgery on a link in $S^3$ \cite{Lickorishsurgery, Wallace}, and in this paper we investigate the conjecture through such surgery descriptions. This is a standard approach; see, for example, \cite{Roberts3, roberts2, LR, KR1,  krishna1, DRdiamond, DRpersistently, S, Stwobridge, Zhao, krishna2}.


Concerning foliations, the most natural and common way to construct a taut foliation on the $r$–surgery on a knot $K$ is to first build a coorientable taut foliation on the knot exterior (\ie, the complement of an open tubular neighbourhood of $K$) that meets the boundary torus transversely in parallel curves of slope $r$. The leaves of this foliation can then be capped off with the meridional discs of the attached solid torus, yielding a taut foliation of the $r$–surgery on $K$. This motivates the following definition.

\begin{defn}[\cite{DRdiamond}]
A knot $K$ is \emph{persistently foliar} if for each non-meridional slope (not necessarily rational) on $K$ there exists a coorientable taut foliation in the exterior of $K$ intersecting the boundary of the knot exterior transversely in parallel curves of that slope.
\end{defn}

In particular, every non-trivial (\ie, non-meridional) surgery on a persistently foliar knot supports a coorientable taut foliation.  Recall that a knot $K\subset S^3$ is an \emph{L-space knot} if some non-trivial surgery on $K$ produces an $L$-space\footnote{In much of the literature, the term $L$-space knot refers to knots with a \emph{positive} $L$-space surgery; here we follow the convention of  \cite{DRpersistently}.}.
Motivated by the $L$-space conjecture, Delman and Roberts \cite{DRpersistently} proposed the following:

\begin{namedconjecture}[\cite{DRpersistently}]
A knot is persistently foliar if and only if it is not an $L$-space knot and has no reducible surgeries.
\end{namedconjecture}

Many constructions of taut foliations in the literature focus on fibered knots and links, and rely essentially on this property. In contrast, non-fibered knots in $S^3$ are never $L$–space knots \cite{G, Ni}, and the $L$–space conjecture predicts that, aside from reducible surgeries, every non-trivial surgery on such knots admits a taut foliation.

The main theorem of this paper provides sufficient \emph{diagrammatic} conditions for a knot $K$ to be persistently foliar, and thus applies to both fibered and non-fibered knots. More precisely, given a knot diagram, we describe two easily checkable conditions on a pair of graphs derived from the diagram that guarantee $K$ is persistently foliar. We note that the use of diagrammatic properties to construct foliations appears already in \cite{Roberts3}, where Roberts proves that a large class of alternating knots are persistently foliar.

Before stating the theorem, we recall some definitions.

\begin{defn}
A \emph{bigon} in a knot diagram is a complementary region of the knot projection in $S^2$ having two vertices\footnote{corresponding to two crossings of the diagram.} in its boundary. A \emph{twist region} is either a maximal connected chain of bigons arranged in a row — that is, one that is not contained in a longer such chain — or a single crossing adjacent to no bigons.

\end{defn}
\begin{assum}\label{assum 1} We will always suppose that the diagram is alternating in a twist region. Otherwise, the diagram can be replaced by another one with fewer crossings by using Reidemeister $2$ moves.
\end{assum}

Throughout the paper, when drawing diagrams, we will indicate the crossings in a twist region with an integer whose absolute value is the number of crossings, and whose sign indicates the type of crossings: positive for right-handed, negative for left-handed. 

We associate two graphs with weighted edges, $\mathcal{G}_g$ and $\mathcal{G}_r$, to a knot diagram $D$ as follows:

\begin{itemize}
\item \textbf{Step 1.} We replace each twist region in $D$ with a weighted blue arc, where the weight records the number of crossings in the region, and we call $\Gamma$ the resulting graph in $S^2$ (see \Cref{fig: Gamma' intro} -- left).


\item \textbf{Step 2.} We collapse each blue arc of $\Gamma$ to a point, obtaining a graph  $\Gamma' \subset S^2$ whose vertices all have valence $4$, and whose complementary regions correspond naturally to those of $\Gamma$. We take any checkerboard colouring of the complementary regions of $\Gamma'$, and transfer this colouring to $\Gamma$ via the correspondence (see \Cref{fig: Gamma' intro} -- right). For convenience, we colour the regions green and red.
\begin{figure}[h]
    \centering
    \includegraphics[width=0.55\textwidth]{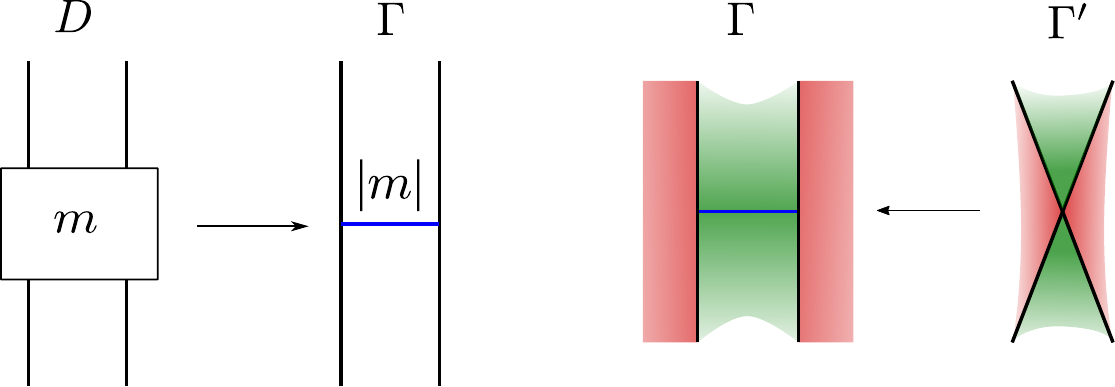}
    \caption{Left: Construction of $\Gamma$ from $D$. Right: Colouring the complementary regions of $\Gamma$ via a checkerboard colouring of the complementary regions of $\Gamma'$.}
    \label{fig: Gamma' intro}
\end{figure}

\item \textbf{Step 3.} The vertices of the graph $\mathcal{G}_g$ correspond to the green-coloured complementary regions of $\Gamma$. Then, each blue arc in $\Gamma$ intersecting some green regions \emph{only} at its endpoints determines a weighted edge between the corresponding vertices, with weight equal to that of the blue arc. The graph $\mathcal{G}_r$ is defined analogously, using the red-coloured regions and the blue arcs that intersect them only at their endpoints.
\end{itemize}

\Cref{fig: K_n} illustrates the construction of $\mathcal{G}_g$ and $\mathcal{G}_r$ in a concrete example. The main result of the paper is the following. We recall that all diagrams in this paper satisfy \Cref{assum 1}.

\begin{namedtheorem}[\ref{thm: main theorem}]
Let $K$ be a knot in $S^3$ with a diagram $D$, and suppose that $D$ has more than one twist region. Assume that:
\begin{itemize}
\item All weights of the graphs $\mathcal{G}_r$ and $\mathcal{G}_g$ are greater than one, and at least one weight is greater than two.
\item The graphs $\mathcal{G}_r$ and $\mathcal{G}_g$ are connected.
\end{itemize}
Then $K$ is persistently foliar. In particular, every non-trivial surgery on $K$ admits a coorientable taut foliation.
\end{namedtheorem}

\begin{figure}[h]
    \centering
    \includegraphics[width=0.8\textwidth]{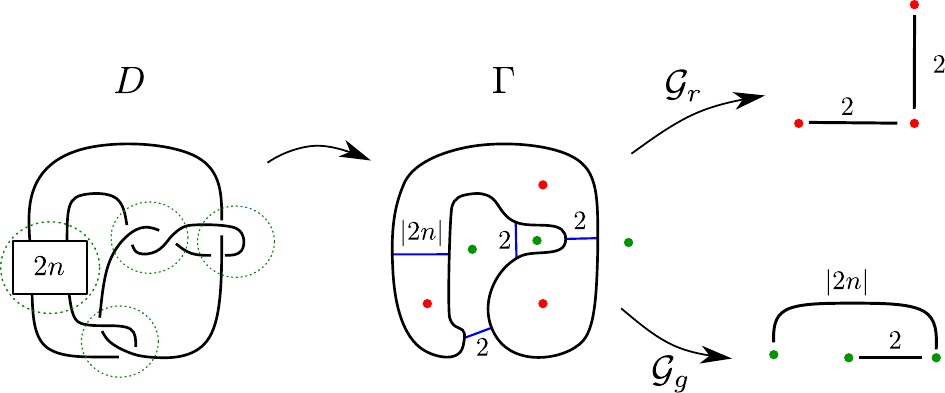}
    \caption{Construction of the graphs $\mathcal{G}_g$ and $\mathcal{G}_r$ in a concrete example. Red and green dots indicate the colours of the complementary regions of $\Gamma$. When $|n|>1$ the knot shown is persistently foliar. When $n=-1$ it is an $L$-space knot.}
    \label{fig: K_n}
\end{figure}

Observe that the first condition in the statement of \Cref{thm: main theorem} is equivalent to requiring that every twist region contains more than one crossing, with at least one twist region containing more than two. We stress that the weights are always positive, as they record the number of crossings rather than their signs.
Moreover, by \Cref{lemma: grafi connessi sse contrattili} we have that in order to verify the second condition in \Cref{thm: main theorem}, it suffices to check that either of the two graphs is a tree.

\begin{rem}
The assumption that the knot diagram $D$ has more than one twist region is not restrictive, since a knot admitting such a diagram must be a $(2,k)$-torus knot. Torus knots always have $L$-space surgeries \cite{Moser}, and therefore cannot be persistently foliar. Moreover, to apply \Cref{thm: main theorem} it is unnecessary to verify that $K$ is not a torus knot; one only needs to check that the diagram $D$ has more than one twist region, which is immediate.    
\end{rem}

\begin{rem}
\Cref{thm: main theorem} can be rephrased in terms of two other well-known graphs associated to a knot diagram, the \emph{Tait graphs}. This reformulation is presented in \Cref{sec: Tait}.
\end{rem}

As the following two examples show, both conditions required in the statement of \Cref{thm: main theorem} are necessary. 

\begin{example}
For a given integer $n\ne 0$, consider the knot $K_n$ defined by the diagram in \Cref{fig: K_n}. The graphs $\mathcal{G}_r$ and $\mathcal{G}_g$ associated to this diagram are both connected, and all twist regions (encircled with green dashed circles) have more than one crossing. When $|n|>1$, there is at least one twist region with more than two crossings, and, by \Cref{thm: main theorem}, $K_n$ is persistently foliar. When $|n|=1$, the first condition required in \Cref{thm: main theorem} is not satisfied; in fact, one can verify that $K_{-1}$ is the torus knot $T(2,5)$, and hence is not persistently foliar.

\end{example}

\begin{example}
In the diagram $D$ of the pretzel knot $P(-2,3,7)$ depicted in \Cref{fig: Fintushel-Stern}, all twist regions have more than one crossing, and there is at least one twist region with more than two crossings. However, one of the graphs associated to $D$ is not connected. It was shown by Fintushel and Stern \cite{Fint-Ster} that $P(-2,3,7)$ has positive lens space surgeries; hence is an $L$-space knot and cannot be persistently foliar.
\end{example}

\begin{figure}[h]
    \centering
    \includegraphics[width=0.8\textwidth]{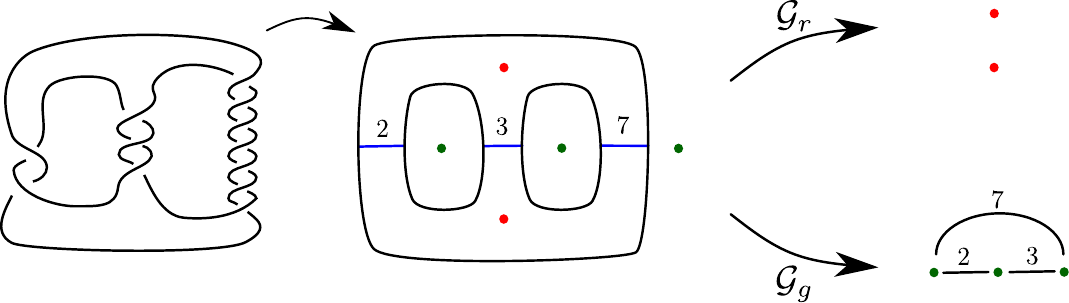}
    \caption{The graph $\mathcal{G}_r$ associated to the diagram in figure is not connected. The pretzel knot $P(-2,3,7)$ is an $L$-space knot and hence cannot be persistently foliar.}
    \label{fig: Fintushel-Stern}
\end{figure}

\subsection{Foliations on surgeries on the Borromean rings.}
The statement of \Cref{thm: main theorem} concerns surgeries on knots. Nevertheless, the ideas used in its proof can also be applied to construct foliations on surgeries on links with multiple components. For example, we will be able to prove the following:

\begin{namedtheorem} [\ref{thm: Borromean}]
    Let $\mathcal{B}$ be the Borromean rings and let $M$ be a Dehn surgery on $\mathcal{B}$. Then $M$ supports a coorientable taut foliation if and only if $M$ is not an $L$-space. More precisely, if $r_1, r_2, r_3$ are rational numbers, then:
    \begin{itemize}
    \item The $(r_1,r_2,r_3)$-surgery on $\mathcal{B}$ is an $L$-space if and only if $(r_1,r_2,r_3)\in [1, \infty)^3 \cup (-\infty, -1]^3$.
    \item The $(r_1,r_2,r_3)$-surgery on $\mathcal{B}$ supports a coorientable taut foliation otherwise.
    \end{itemize}
\end{namedtheorem}

Note that when at least one surgery coefficient is $\infty$, the only rational homology spheres obtained by surgery on $\mathcal{B}$ are lens spaces or connected sums of lens spaces. A pictorial illustration of the statement of \Cref{thm: Borromean} is shown on the right side of \Cref{fig: B}.

\begin{figure}[h]
    \centering
    \includegraphics[width=0.7\textwidth]{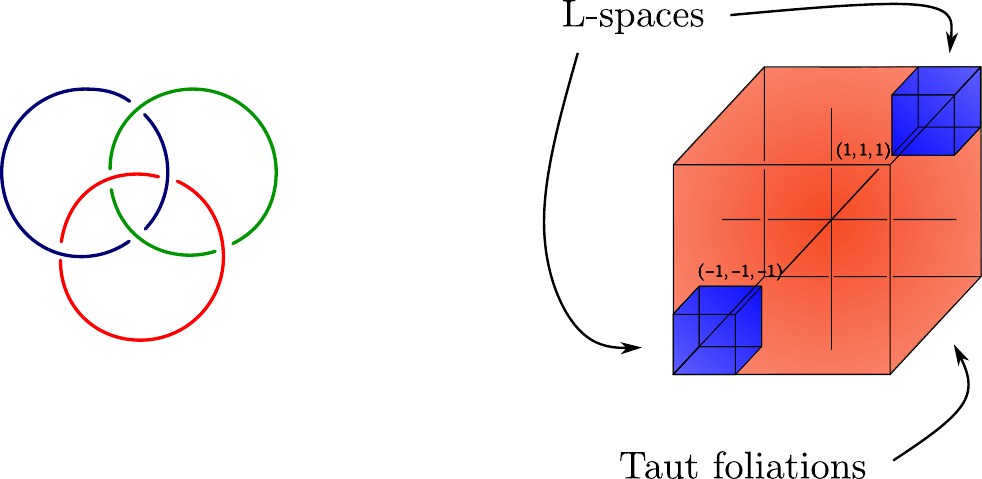}
    \caption{The Borromean rings, together with the set of surgeries on it that admit a coorientable taut foliation (in red) and those that yield $L$-spaces (in blue).}
    \label{fig: B}
\end{figure}

To the best of the author's knowledge, \Cref{thm: Borromean} provides the first instance of the equivalence between conditions $(2)$ and $(3)$ of the $L$-space conjecture for all surgeries on a hyperbolic link with more than two components. We remark that the Borromean rings is a fibered link, but we will not make use of this property.

The proof of \Cref{thm: Borromean} is given in \Cref{sec: Borromean rings and Bing doubles}, where we also present further applications to \emph{Bing doubles} of knots and links (see \Cref{thm: Bing doubles}).

\subsection{Applications to arborescent knots and some braid closures.}
\Cref{sec: applications} is devoted to some applications of \Cref{thm: main theorem} to arborescent knots and braid closures.
Given a finite tree $T$ embedded in the plane, whose vertices are endowed with an integer weight and the choice of an angular sector, one can construct a link in $S^3$ by plumbing twisted Hopf bands according to the pattern given by the tree. Links arising in this way are called \emph{arborescent links}, they were defined by Conway \cite{Conway70}\footnote{Conway called them algebraic links, but the term algebraic link usually refers to link of plane curves singularities.}, and include many well-known families of links, such as two-bridge links, pretzel links, and Montesinos links.

It was proven in \cite{Wu} that non-torus arborescent knots have no reducible surgeries, and in \cite{LM, BM, LMZ} that the only arborescent knots admitting non-trivial $L$-space surgeries are, up to mirroring, the pretzel knots $P(-2, 3, q)$ and the torus knots $T(2, q)$, with odd $q\geq 1$.  Therefore, the $L$-space conjecture predicts that all non-trivial surgeries on other arborescent knots admit coorientable taut foliations.

\Cref{thm: main theorem} can be applied to show that this prediction holds for many arborescent knots.
\begin{namedtheorem}[\ref{thm: arborescent}]
Let $K$ be an arborescent knot associated to a weighted planar tree $T$ with more than one vertex. If all vertex weights have absolute value greater than one, and at least one weight has absolute value greater than two, then $K$ is persistently foliar.
\end{namedtheorem}

We conclude the paper by describing some applications to knots via their presentations as braid closures. Let $B_n$ denote the braid group on $n$ strands, with Artin generators $\sigma_1, \dots, \sigma_{n-1}$.

\begin{namedtheorem}[\ref{thm: braids}]Let $n\geq 3$ be an odd positive integer, and let $K=\hat{\beta}$, where 
\[
\beta = \sigma_{i_1}^{a_1} \sigma_{i_2}^{a_2} \cdots \sigma_{i_k}^{a_k} \in B_n
\] 
is a reduced word, each exponent satisfies $|a_j| \ge 2$, and for at least one index $j_0$, we have $|a_{j_0}| > 2$. Suppose further that the following conditions hold: 
\begin{itemize}
\item For every odd index $h\geq 1$, between any two consecutive occurrences of $\sigma_h$, there is at least one occurrence of $\sigma_{h+1}$. 
\item For every even index $h\geq 2$, between any two consecutive occurrences of $\sigma_h$, there is at least one occurrence of $\sigma_{h-1}$. 
\end{itemize}
Then $K$ is persistently foliar.
\end{namedtheorem}

The terms \emph{reduced word} and \emph{occurrence} are defined in \Cref{sec: applications}. As a concise corollary of \Cref{thm: braids}, we highlight the following.

\begin{namedcorollary}[\ref{cor: 3-braids}]
Let $K$ be the closure of a braid 
\[
\beta = \sigma_1^{a_1}\sigma_2^{a_2}\sigma_1^{a_3}\cdots \sigma_2^{a_k} \in B_3,
\] 
with $|a_i|\ge 2$ for all $i$, and $|a_{i_0}|>2$ for some index $i_0$. Then $K$ is persistently foliar.
\end{namedcorollary}

\textbf{Outline of main proof and structure of paper.}
A brief outline of the proof of \Cref{thm: main theorem} is as follows. Given a diagram of a knot  $K$ in  $S^3$ with $n$  twist regions, we associate to it an $(n+1)$-component link $ L = K_0 \sqcup \cdots \sqcup K_n $ in  $S^3$ , with the property that $K$ can be recovered by performing Dehn surgery -- with appropriate coefficients  $\frac{1}{k_1}, \dots, \frac{1}{k_n}$ -- on the components $K_1, \dots, K_n$ of $L$. We then define a $2$-complex in the exterior of $L$ and describe how to smooth it into a branched surface. The conditions on the diagram of $K$ stated in \Cref{thm: main theorem} are used to show that this branched surface gives rise to coorientable taut foliations in the exterior of $L$, which intersect the boundary component associated to  $K_0$ transversely in a foliation of any fixed non-meridional slope, and that associated to $K_i$  in a foliation of slope $\frac{1}{k_i}$, for $i = 1, \dots, n$. Capping off these foliations with meridional discs during Dehn surgery yields the desired taut foliations in the exterior of $K$ .

The necessary background on branched surfaces and sutured manifolds is provided in \Cref{sec: Background}. In this section we also recall \Cref{boundary train tracks} by Li and \Cref{cusps implies persistently foliar} by Delman and Roberts, both of which play a key role in the proof of the main theorem.
In \Cref{sec: Borromean rings and Bing doubles}, we construct foliations on surgeries of two three-component links: the Borromean rings and the link $L8n5$. This section contains the proof of \Cref{thm: Borromean}, and also discusses applications to Bing doubles of knots and links. Moreover, it serves to illustrate concrete examples of the more general construction needed
for the proof of \Cref{thm: main theorem}, that is presented in \Cref{sec: main theorem}.
In  \Cref{sec: Tait} we rephrase \Cref{thm: main theorem} in terms of the Tait graphs associated to a knot diagram, and finally, \Cref{sec: applications} contains applications to arborescent knots and certain braid closures.
\newline 

\textbf{Acknowledgments.} 
During the preparation of this paper, I have greatly benefited from discussions with many people; in particular, I would like to thank Ken Baker, Gemma Di Petrillo, Alessio Di Prisa, Tao Li, Paolo Lisca, Bruno Martelli, Alice Merz, and Eric Stenhede. I am also grateful to Charles Delman and Rachel Roberts for carefully explaining the details of the proof of \Cref{cusps implies persistently foliar}.  I would also like to thank the anonymous referee for their careful reading and valuable suggestions.
The author is supported by the FWF project P 34318 
“Cut and Paste Methods in Low Dimensional Topology”.

\section{Background on branched surfaces and sutured manifolds}\label{sec: Background}
In this section we recall background material and state two results that will play a central role in this paper \Cref{boundary train tracks} of Li \cite{L2}, and \Cref{cusps implies persistently foliar} of Delman and Roberts \cite{DRdiamond}.
We begin with branched surfaces, our main tool for constructing foliations, and review the notions of \emph{laminations}, \emph{sink discs}, and \emph{laminar} branched surfaces, which are needed for the  statement of \Cref{boundary train tracks}.
In the second part, we recall elements of Gabai’s theory of sutured manifolds \cite{Gabai}, which will be essential for stating \Cref{cusps implies persistently foliar} and for several arguments in the later sections.
For further background on branched surfaces, see \cite{FO,O}, and for sutured manifolds, see \cite{Gabai}.
We conclude the section with a brief discussion on train tracks.

\subsection{Branched surfaces}
\begin{defn}
A \emph{branched surface with boundary} in a $3$-manifold $M$ is a compact subset $B\subset M$ locally diffeomorphic to one of the models shown of \Cref{branched surface} $a)$ (in $\matR^3$) or \Cref{branched surface} $b)$ (in the closed half-space), where $\partial B:= B\cap \partial M$ is represented with a bold line.
\end{defn}

\begin{figure}[]
    \centering
    \includegraphics[width=0.8\textwidth]{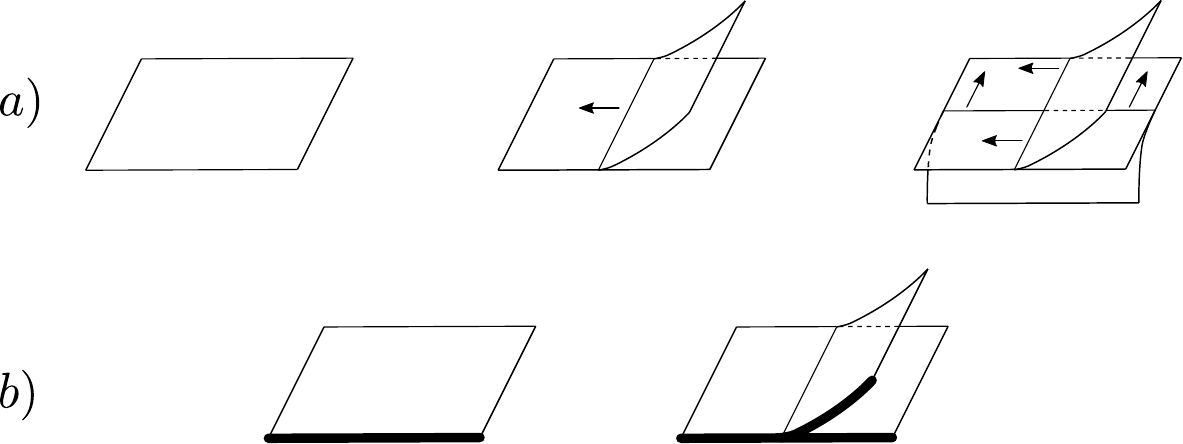}
    \caption{Local models for a branched surface, with cusp directions.}
    \label{branched surface}
\end{figure}

Branched surfaces generalise train tracks from surfaces to $3$-manifolds; when non-empty, $\partial B$ defines a train track in $\partial M$, \ie, an embedded graph in $\partial M$ with well-defined tangencies at the vertices. For more details and the basics on train tracks, we refer the reader to \cite{Martelli}.

One can identify two distinguished subsets in a branched surface $B$: the \emph{branch locus} and the set of \emph{triple points}. The \emph{branch locus} consists of points where $B$ is not locally homeomorphic to a surface, and the set of \emph{triple points} of $B$ consists of points where the branch locus is not locally homeomorphic to an arc. For example, the rightmost model in \Cref{branched surface} $a)$ contains a triple point.

The complement of the branch locus in $B$ is a union of connected surfaces. The abstract closures of these surfaces -- taken under any path metric on $M$ -- are called the \emph{sectors} of $B$. Similarly, the complement of the set of triple points within the branch locus is a union of $1$-dimensional connected manifolds. To each of these manifolds, one can associate an arrow in $B$ indicating the direction of the smoothing, as illustrated in \Cref{branched surface}. These arrows are called \emph{cusp directions}.

\begin{figure}[h]
    \centering
    \includegraphics[width=0.6\textwidth]{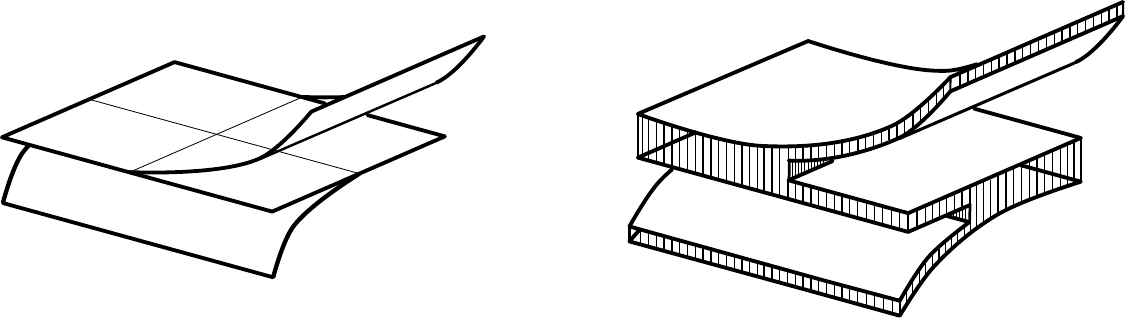}
    \caption{Regular neighbourhood of a branched surface.}
    \label{regular neighbourhood}
\end{figure}

We denote by $N_B$ a fibered regular neighbourhood of $B$, well-defined up to isotopy, whose local structure is depicted in \Cref{regular neighbourhood}.
The boundary of $N_B$ naturally decomposes into three compact subsurfaces: the \emph{horizontal boundary} $\partial_h N_B$, the \emph{vertical boundary} $\partial_v N_B$, and $N_B\cap \partial M$. The horizontal boundary is transverse to the interval fibers of $N_B$, while the vertical boundary intersects them, if at all, in one or two proper closed subintervals contained in their interiors. Collapsing each interval fiber of $N_B$ to a point produces a branched surface in $M$ isotopic to $B$, with the image of $\partial_v N_B$ coinciding with the branch locus of this branched surface. By slight abuse of notation, we identify this branched surface with $B$, and consider $N_B$ as a $[0,1]$-bundle over $B$, with projection map $\pi: N_B\rightarrow B$.

A branched surface can be obtained from another through the following operation, called \emph{splitting}.

\begin{defn}
Let $B_1$ and $B_2$ be branched surfaces in $M$, and let $F$ be a compact orientable surface. We say that $B_2$ is obtained by \emph{splitting} $B_1$ if 
$$
N_{B_1}=N_{B_2}\cup J, \text{ where } J=F\times [0,1],
$$
with the following properties:

\begin{itemize}
\item $F\times \{0,1\}\subset \partial_h N_{B_2}$.
\item $(\partial F \times [0,1])\cap \partial N_{B_2}\subset \partial_v N_{B_2}$, and the $[0,1]$-fibers of $\partial F\times [0,1]$ coincide with the $[0,1]$-fibers of $\partial_v N_{B_2}$.
\end{itemize}
\end{defn}

\Cref{splitting} shows examples of splitting in the case of train tracks. Branched surfaces provide a useful tool to construct \emph{laminations} on $3$-manifolds. 

\begin{figure}[]
    \centering
    \includegraphics[width=0.55\textwidth]{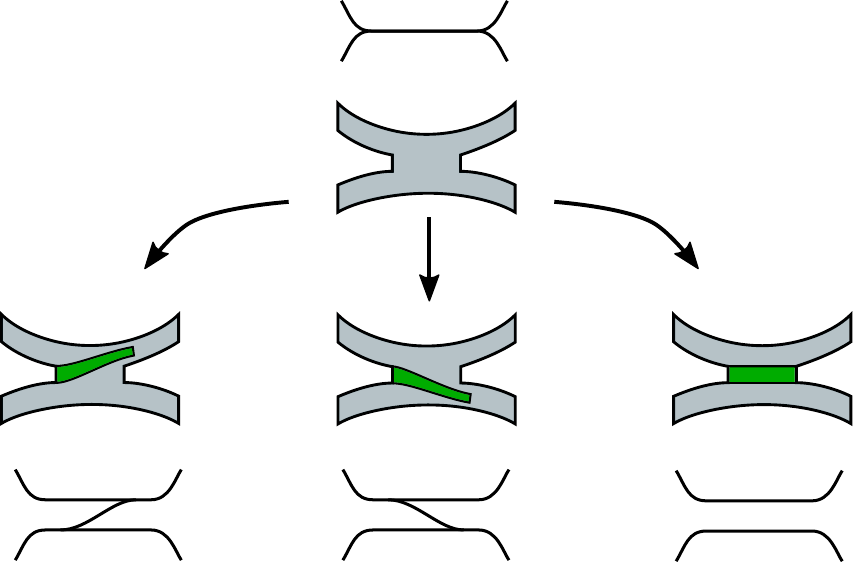}
    \caption{Three possible splittings of a train track. The green-coloured region is the  $[0,1]$-bundle $J$.}
    \label{splitting}
\end{figure}

\begin{defn}\label{lamination abstr}
A lamination $\mathcal{L}$ in $M$ is a decomposition of a closed subset of $M$ into a union of injectively immersed connected surfaces, called \emph{leaves} of $\mathcal{L}$, such that $(M, \mathcal{L})$ is locally homeomorphic to $(\matR^3, \matR^2\times C)$ or to $(\matH^3, \matH^2\times C)$. Here, $C$ is a closed subset of $\matR$ depending on the chart, $\matH^2$ is the closed half-plane, and $\matH^3=\matH^2\times \mathbb{R}$ is the closed half-space.
\end{defn}

\begin{defn}\label{lamination}
Let $B$ be a branched surface in a $3$-manifold $M$. A lamination $\mathcal{L}$ is \emph{carried by B} if $\mathcal{L}\subset N_B$ and $\mathcal{L}$ intersects the fibers of $N_B$ transversely. We say that $\mathcal{L}$ is \emph{fully carried} by $B$ if it is carried by $B$ and intersects every fiber of $N_B$. If $F$ is a surface carried by $B$, we say that $F$ \emph{passes through} a sector of $B$ if $F$ intersects some fibers of $N_B$ lying over its interior.
\end{defn}

We will assume that the mappings of the leaves into $M$
are smooth immersions. Analogous definitions apply to train tracks on surfaces, and, for later use, we introduce the following notion.
\begin{defn}
Let $\tau$ be a train track on a torus $T$. A rational slope $s$ on $T$ is \emph{realised} by $\tau$ if $\tau$ fully carries a union of finitely many curves of slope $s$.
\end{defn}

In \cite{L}, Li introduces the notion of \emph{sink disc}.

\begin{defn}
Let $B$ be a branched surface in $M$ and let $S$ be a sector in $B$.
We say that $S$ is a \emph{sink disc} if $S$ is homeomorphic to a disc, $S\cap \partial M=\emptyset$, and the cusp direction of any smooth curve or arc in its boundary points into $S$.
We say that $S$ is a \emph{half sink disc} if $S$ is a disc, $S  \cap \partial M\neq \emptyset$, and the cusp direction of any smooth arc in $\partial S\setminus \partial M$ points into $S$.
\end{defn}

See \Cref{discs} for some examples of sink discs and half sink discs. Notice that if $S$ is a half sink disc the intersection $\partial S \cap \partial M$, represented with  bold lines, can be disconnected.

\begin{figure}[]
    \centering
    \includegraphics[width=0.5\textwidth]{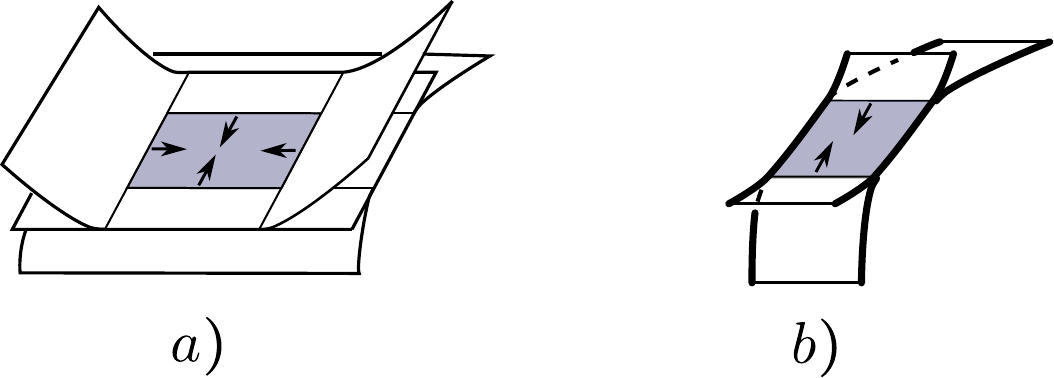}
    \caption{Examples of $a)$ sink discs and $b)$ half sink discs.}
    \label{discs}
\end{figure}

If $B$ contains a sink disc or a half sink disc there is a simple way to eliminate it: it is enough to blow an air bubble in its interior\footnote{More precisely, we split $B$ along a small disc contained in the sink disc sector.}, as in \Cref{air bubble}, so to obtain a new branched surface $B'$. However there is really no difference between $B$ and $B'$: in fact it is not difficult to see that $B$ carries a lamination if and only if $B'$ carries a lamination. 

We do not want to artificially eliminate sink discs with this procedure, and so we recall the notion of \emph{trivial bubble}.
A connected component of $M\setminus {\rm int}(N_B)$ is a $\matD^2\times [0,1]$ \emph{region} if it is homeomorphic to a ball and its boundary can be subdivided into an annular region, corresponding to a component of $\partial_v N_B$, and two $\matD^2$ regions corresponding to components of $\partial_h N_B$. A $\matD^2 \times [0,1]$ region is \emph{trivial} if the projection map $\pi: N_B\to B$ is injective on ${\rm int}(\matD^2)\times \{0,1\}$. In this case, the image of $\matD^2\times \{0,1\}$ via the collapsing map is called a \emph{trivial bubble} in $B$. 
When $M$ and $B$ have boundary these definitions can be generalised in a natural way to the relative case, see \cite{L2}. Trivial bubbles and trivial $\matD^2\times [0,1]$ regions are created when we eliminate sink discs as in \Cref{air bubble}.

\begin{figure}[h]
    \centering
    \includegraphics[width=0.4\textwidth]{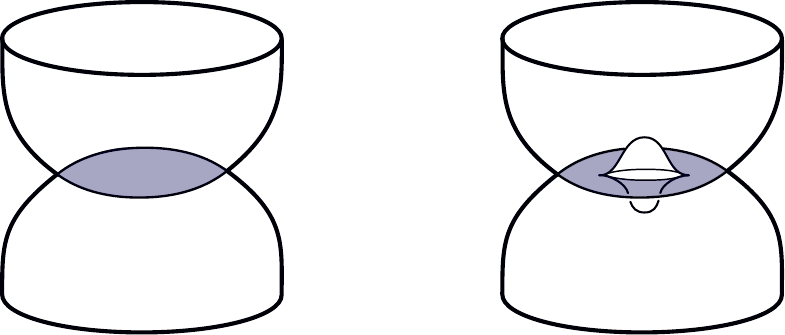}
    \caption{The shaded sector on the left is a sink disc. In the branched surface on the right, this sink disc has been removed, and a trivial bubble has been created.}
    \label{air bubble}
\end{figure}

In \cite{L}, Li introduces the definition of laminar branched surface, and in \cite{L2} he generalises this definition to branched surfaces with boundary as follows.

\begin{defn}[\cite{L,L2}]\label{def: laminar}
Let $B$ be a branched surface in a $3$-manifold $M$. We say that $B$ is \emph{laminar} if $B$ has no trivial bubbles and the following hold:

\begin{enumerate}
    \item $\partial_h N_B$ is incompressible and $\partial$-incompressible in $M\setminus {\rm int}(N_B)$, and no component of $\partial_h N_B$ is a sphere or a properly embedded disc in $M$.
    \item There is no monogon in $M\setminus {\rm int}(N_B)$, \ie, no disc $D\subset M\setminus {\rm int}(N_B)$ such that 
    \[
    \partial D=D\cap N_B=\alpha\cup \beta,
    \]
    where $\alpha$ is in an interval fiber of $\partial_v N_B$, and $\beta$ is an arc in $\partial_h N_B$.
    \item $M\setminus {\rm int}(N_B)$ is irreducible, and $\partial M\setminus {\rm int}(N_B)$ is incompressible in $M\setminus {\rm int}(N_B)$.
    \item $B$ contains no Reeb branched surfaces (see \cite{GO} for the definition).
    \item $B$ has no sink discs or half sink discs.
\end{enumerate}
\end{defn}

The requirement of $\partial$-incompressibility in $(1)$ means the following. Suppose that $D$ is a disc in $M\setminus {\rm int}(N_B)$ with ${\rm int}(D)\subset M\setminus N_B$ and $\partial D=\alpha\cup \beta$, where $\alpha$ is an arc in $\partial_h N_B$ and $\beta$ is an arc in $\partial M$. Then $\partial$-incompressibility requires that there exists a disc $D'\subset \partial_h N_B$ with $\partial D'=\alpha\cup \beta'$, where $\beta'=\partial D'\cap \partial M$.
\newline

The key property of laminar branched surfaces is that they fully carry essential laminations\footnote{For the definition of essential laminations, see \cite{GO}.}  \cite{L}, and when the ambient manifold $M$ has torus boundary, they can also provide essential laminations in fillings of $M$, as the theorem below shows.

We fix the following notation. Let $M$ be a manifold whose boundary is union of tori $T_1, \dots, T_k$, and let $h\leq k$ be a positive integer. Given $(s_1,\dots, s_h)$ a rational \emph{multislope}, \ie, the choice of a rational slope $s_i$ on $T_i$ for $i\leq h$, we denote by $M(s_1,\dots, s_h)$ the manifold obtained by filling the boundary component $T_i$ along the slope $s_i$, for $i=1, \dots, h$. When $h < k$, this manifold has non-empty boundary.

\begin{teo}[{\cite[Theorem~2.5]{L2}}]\label{boundary train tracks}
Let $M$ be a compact $3$-manifold whose boundary is union of tori $T_1,\dots, T_k$. Let $B$ be a laminar branched surface in $M$ with $B\cap \partial M\subset T_1\cup \cdots \cup T_h$, for some $h\leq k$, and assume that, for each $i=1, \dots, h$, the complement $T_i \setminus \partial B$ is a union of bigons. Let $(s_1,\dots, s_h)$ be any rational multislope realised by the train track $\partial B$, and suppose that $B$ does not carry a torus that bounds a solid torus in $M(s_1,\dots,s_h)$. Then, there exists an essential lamination $\mathcal{L}$ in $M$ fully carried by $B$ that intersects $T_i$ in parallel simple closed curves of slope $s_i$, for $i=1, \dots, h$. Moreover, this lamination extends to an essential lamination of the filled manifold $M(s_1,\dots, s_h)$.
\end{teo}
\begin{rem}
In \cite{L2}, the theorem is stated for $M$
with connected boundary, but, as noted in \cite{KR1}, the same proof works for manifolds with multiple boundary components. Also, while \cite{L2} assumes that $M$ is irreducible and has incompressible boundary, the proof remains valid without these hypotheses. Indeed, if $M$ contains a laminar branched surface, then $M$
is necessarily irreducible and has incompressible boundary \cite[Theorem~1]{GO}.
\end{rem}

\begin{rem}
The conclusion of \Cref{boundary train tracks} also differs slightly from that in \cite{L2}: the additional details are drawn from the proof itself. The proof proceeds by splitting the branched surface $B$ in a neighbourhood of $\partial M$, so that it intersects each $T_i$ in parallel simple closed curves of slopes $s_i$, for $i=1,\dots h$. Gluing the meridional discs of the solid tori to $B$ during the Dehn filling produces a branched surface $B(s_1,\dots, s_h)$ in $M(s_1,\dots, s_h)$, which is laminar and fully carries an essential lamination by \cite[Theorem~1]{L}. This essential lamination is obtained by extending an essential lamination in $M$ that intersects each $T_i$ in parallel simple closed curves of slopes $s_i$, for $i=1,\dots, h$.
\end{rem}

\subsection{Sutured manifolds}
We now recall some basic definitions from the theory of \emph{sutured manifolds}, introduced by Gabai in \cite{Gabai}. This framework will be useful for stating and proving several results in \Cref{sec: main theorem}.
\begin{defn}
A pair $(M, \gamma)$ is a \emph{sutured manifold} if $M$ is a compact oriented $3$-manifold and $\gamma = A(\gamma) \cup T(\gamma)$ is a subsurface of $\partial M$, with $A(\gamma) \cap T(\gamma) = \emptyset$. Here, $A(\gamma)$ is a disjoint union of annuli and $T(\gamma)$ is a union of tori. Each annulus in $A(\gamma)$ contains a homologically non-trivial oriented simple closed curve, called a \emph{suture}.

Furthermore, every component of $R(\gamma) = \partial M \setminus \operatorname{int}(\gamma)$ is oriented so that the orientations are coherent with the sutures: if $\delta$ is a component of $\partial R(\gamma)$ given the boundary orientation, then $\delta$ represents the same homology class in $H_1(\gamma, \mathbb{Z})$ as some suture.
\end{defn}

We define $R_+(\gamma)$ as the union of those components of $R(\gamma)$ whose normal vectors point into $M$, and $R_-(\gamma)$ as the union those components of $R(\gamma)$ whose normal vectors point out of $M$. This convention is the opposite of that in \cite{Gabai}, but it will more natural for our setting.
\Cref{fig: sutured example}~--~$a)$ and \Cref{fig: sutured example}~--~$b)$ show two examples of sutured manifolds, while \Cref{fig: sutured example}~--~$c)$ illustrates a non-example.

\begin{figure}[h]
    \centering
    \includegraphics[width=1\textwidth]{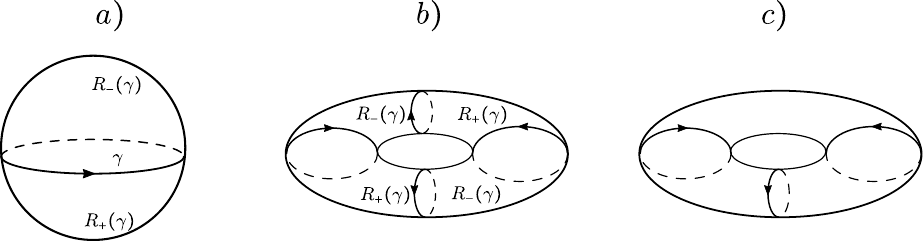}
    \caption{$a)$ and $b)$ are examples of sutured manifolds, but $c)$ is not. In the picture, $\gamma$ is obtained by considering a neighbourhood of the sutures.}
    \label{fig: sutured example}
    \end{figure}

\begin{defn}
A sutured manifold $(M, \gamma)$ is \emph{taut} if $M$ is irreducible and $R(\gamma)$ is both incompressible and Thurston norm minimising in $H_2(M,\gamma)$.
\end{defn}

Examples of taut sutured manifolds are given by the \emph{product sutured manifolds}.

\begin{defn}
Let $S$ be a compact oriented surface with (possibly empty) boundary. A \emph{product sutured manifold} is a sutured manifold $(M, \gamma)$, with $M=S\times [0,1]$ and $\gamma=\partial S\times [0,1]$, where the sutures $\partial S$ are given the boundary orientation.
\end{defn}
For example, \Cref{fig: sutured example}~$a)$ is a product sutured ball.
We record here the following simple observation, that will be useful in the next sections.

\begin{rem}\label{rem: product ball}
Suppose that $(M, \gamma)$ is a sutured manifold, where $M$ is a $3$-ball and $R_{+}(\gamma)$ and $R_-(\gamma)$ are both connected and non-empty. Then, $(M, \gamma)$ is a product sutured ball.
\end{rem}

Sutured manifolds arise naturally in presence of branched surfaces. In fact, if $B$ is a cooriented branched surface in a closed, oriented $3$-manifold $M$, then the pair ($M\setminus {\rm int} (N_B), \partial_v N_B)$ is a sutured manifold. In this case, $R_{+}(\gamma)$ (resp. $R_-(\gamma)$) consists of the components of $\partial_h N_B$ whose coorientation points out of (resp. into) $N_B$. See \Cref{fig: sutured branched surface}.

\begin{figure}[h]
    \centering
    \includegraphics[width=0.6\textwidth]{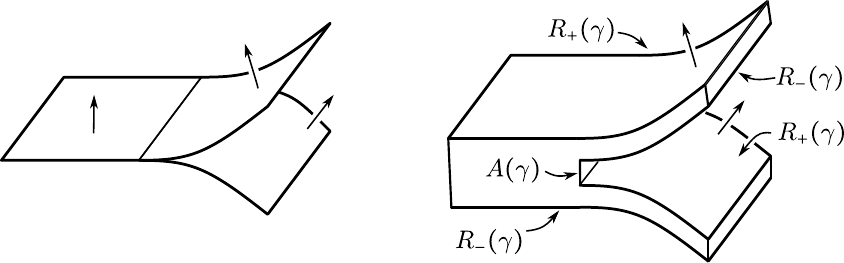}
    \caption{Structure of the exterior of a branched surface as a sutured manifold.}
    \label{fig: sutured branched surface}
\end{figure}

When $M$ has boundary consisting of a union of tori, and the complementary regions of $\partial B$ in $\partial M$ are bigons, annuli and tori, we define ($M\setminus {\rm int} (N_B), \gamma)$ as a sutured manifold, where in this case $\gamma=\partial_v N_B \cup (\partial M\setminus {\rm int} (N_B))$.

We conclude this subsection by stating a result of Delman and Roberts \cite{DRdiamond} that we will use in the next section. Before doing so, we introduce two definitions.

\begin{defn}[\cite{DRdiamond}]
A cooriented branched surface $B$ is \emph{taut} if the exterior of $B$ is a taut sutured manifold and if through every sector of $B$ there is a closed oriented curve that is
positively transverse to $B$.
\end{defn}

If $\mathcal{L}$ is a lamination in $M$ carried by a branched surface $B$, we will suppose that $\partial_h N_B$ is contained in $\mathcal{L}$. Moreover, we denote by $M_{|\mathcal{L}}$ the metric completion of $M\setminus \mathcal{L}$ under any path metric on  $M\setminus \mathcal{L}$.

\begin{defn}[\cite{DRdiamond}]
Suppose that $B$ fully carries a lamination $\mathcal{L}$. A component $A$ of $\partial_v N_B$ satisfies the \emph{noncompact extension property} relative to $(B, \mathcal{L})$ if there exists a proper embedding of $[0,1]\times [0, \infty)$  in $M_{|\mathcal{L}} \cap N_B$ satisfying the following conditions:
\begin{itemize}
    \item the image of $[0,1] \times \{0\}$ is contained in a fiber of $N_B$ and contains a fiber of $A$;
    \item the image of $\{0,1\}\times [0,\infty)$ is contained in leaves of $\mathcal{L}$.
\end{itemize}  
\end{defn}

\begin{teo}[{\cite[Proposition~3.10]{DRdiamond}}]\label{cusps implies persistently foliar}
Let $\partial_0 M$ denote a torus boundary component of $M$, and suppose $B$ is a taut,
cooriented branched surface in $M$ that is disjoint from $\partial_0 M$ and that fully carries a lamination $\mathcal{L}$. Let $(Y, \partial_v Y)$ denote the complementary region of $N_B$ that contains $\partial_0 M$. Suppose that $(Y, \partial_v Y)$ is homeomorphic to
$$
(\partial_0 M\times [0, 1], \mathcal{A}_1 \cup \dots \cup \mathcal{A}_{2l}),
$$where \(\partial_i M\times\{0\}=\partial_i M\) and \(\mathcal{A}_1,\dots,\mathcal{A}_{2n_i}\) are pairwise disjoint essential annuli in \(\partial_i M\times\{1\}\) satisfying the noncompact extension property relative to \((B,\mathcal{L})\). Then, for any multislope \((s_1,\dots,s_m)\) on \(\partial_1 M,\dots,\partial_m M\) such that no \(s_i\) is isotopic in \(Y_i\) to the common core of the annuli \(\mathcal{A}_k\), there exists a cooriented taut foliation that extends \(\mathcal{L}\) and intersects \(\partial_1 M,\dots,\partial_m M\) transversely in a linear foliation of slope \((s_1,\dots,s_m)\).
\end{teo}

\subsection{Train tracks and rational slopes}\label{subsec: train tracks and rational slopes}
We conclude the section with a brief discussion on train tracks on the torus, describing an explicit method to compute the slopes they realise.

A slope on a torus $T$ is an element of the projective space of $H_1(T, \mathbb{R})$. We fix an orientation on $T$, and a meridian-longitude basis, \ie, two oriented essential simple closed curves, $\mu$ and $\lambda$, with $\mu\cdot \lambda=1$, where $\cdot$ denotes the algebraic intersection.
With this chosen basis, the set of slopes on $T$ can be identified with $\overline{\mathbb{R}}=\mathbb{R}\cup \{\infty\}$. Rational slopes on $T$ correspond to elements $\pm (p\mu+q\lambda)\in H_1(T, \Z)$, for coprime integers $p,q$. Here, we set $0$ and $1$ to be coprime, and, with some abuse of notation, we write the homology class of a closed curve $\gamma\subset T$ simply as $\gamma$, omitting the brackets. It is a well-known fact that rational slopes correspond to isotopy classes of essential simple closed curves on $T$. 

Given an oriented train track $\tau$ in $T$, the slopes realised by it can be computed by assigning \emph{weight systems}.

\begin{defn}\label{def: weigth system}
A \emph{weight system} on an oriented train track $\tau\subset T$ is an assignment of a positive rational number, called a \emph{weight}, to each branch of $\tau$, such that at every branch point the sum of the incoming weights equals the sum of the outgoing ones.
\end{defn}

\begin{figure}[]
    \centering
    \includegraphics[width=1\textwidth]{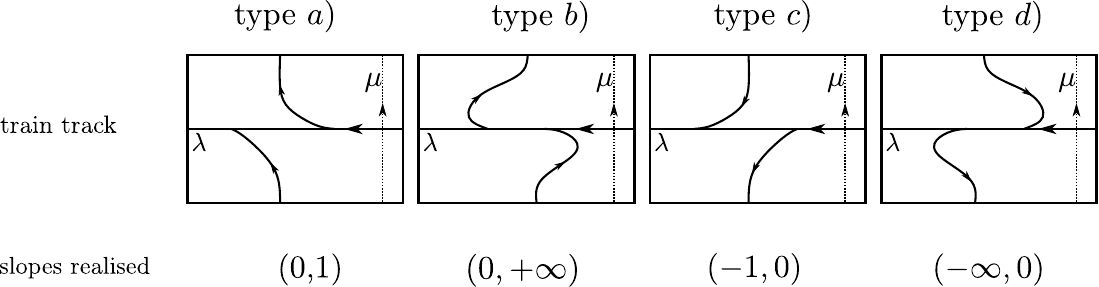}
    \caption{Examples of oriented train tracks on $T$, and the corresponding intervals of rational slopes they realise.}
    \label{fig:train track table}
\end{figure}
A weight system $w$ defines an element of $H_1(T,\mathbb{Q})$, given by the class of $\sum_i w_i c_i$, where the $c_i$ are the oriented branches of $\tau$ and $w_i$ their weights. The \emph{slope} of $w$ is the projective class of this element, generalising the notion of slope for simple closed curves. Concretely, the slope can be computed as $\frac{w_{\mu}}{w_{\lambda}}$, where $w_\mu$ and $w_\lambda$ denote the weighted intersections of $\tau$ with $\lambda$ and $\mu$, respectively.
This notion is relevant for us because of the following well-known fact.
\begin{lemma}
Let $\tau$ be an oriented train track on $T$. A rational slope s on $T$ is realised by $\tau$ if and only if there exists a weight system $w$ on $\tau$ with slope $s$.\qed
\end{lemma}

\Cref{fig:train track table} shows four examples of train tracks on a torus, together with the sets of rational slopes they realise. These sets can be explicitly computed by assigning weight systems and determining their slopes.

In the course of the paper, we will encounter train tracks of the four types described in \Cref{fig:train track table}, or those that we call of \emph{type $t_1+t_2$}, where $t_1$ and $t_2$ are $a), b), c)$ or $d)$. These are obtained by stacking a train track of type $t_1$ and one of type $t_2$ horizontally, as in \Cref{fig:type a+c}.

\begin{figure}[h]
    \centering
    \includegraphics[width=0.4\textwidth]{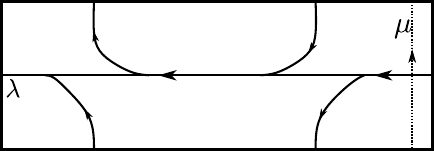}
    \caption{A train track of type $a)+c)$.}
    \label{fig:type a+c}
\end{figure}

The next lemma -- which follows by a direct computation with weight systems -- will be useful when computing the sets of slopes realised by our train tracks.

\begin{lemma}\label{lemma: unlinked}
Let $\tau$ be an oriented train track of type $t_1+t_2$. Then, a rational slope $r$ is realised by $\tau$ if and only if $r=r_1+ r_2$, where $r_i$ is a rational slope realised by the train track of type $t_i$.\qed
\end{lemma}

\section{Two examples}\label{sec: Borromean rings and Bing doubles}
In this section we construct foliations on the surgeries of two three-component links: the Borromean rings and the link $L8n5$. The purpose of this section is twofold: although the results it contain may at first appear unrelated to the main theorem of the paper, the argument presented in this section contain all the key ideas involved in its proof, and serve as illustrative examples to guide the reader in the following sections. In addition, some of the results presented here are of independent interest, the main one being the following.

\begin{teo}\label{thm: Borromean}
    Let $\mathcal{B}$ be the Borromean rings. Then, every Dehn surgery on $\mathcal{B}$ supports a coorientable taut foliation if and only if it is not an $L$-space. More precisely, if $r_1, r_2, r_3$ are rational numbers, then:
    \begin{itemize}
    \item The $(r_1,r_2,r_3)$-surgery on $\mathcal{B}$ is an $L$-space if and only if $(r_1,r_2,r_3)\in [1, \infty)^3 \cup (-\infty, -1]^3$;
    \item The $(r_1,r_2,r_3)$-surgery on $\mathcal{B}$ supports a coorientable taut foliation otherwise.
    \end{itemize}
\end{teo}
Recall that the Borromean rings is the three-components hyperbolic link depicted in \Cref{fig: B}-left. Also recall that it is an amphichiral link, and that any permutation of its components can be realised by an isotopy of the link.
At the end of the section we show how this theorem, together with its proof, can be applied to Bing doubles of knots and links.
\newline

\subsection{Borromean rings.}
We start by fixing some notation that we will use during the section. Let $\nu K$ be a closed tubular neighbourhood of a knot $K$ in $S^3$, oriented with the orientation induced by $S^3$. We orient $\partial \nu K$ as the boundary of $\nu K$, and for simplicity we refer to slopes on $\partial \nu K$ as \emph{slopes on $K$}. Fix an orientation on $K$, and consider the canonical meridian-longitude basis $\mu,\lambda$, where $\lambda$ is the canonical longitude of $K$, oriented consistently with $K$, and $\mu$ is the meridian of $K$, oriented so that $\mu$ has linking number $+1$ with $K$, and use it to identify slopes on $K$ with elements in $\overline{\mathbb{R}}$. Observe that reversing the orientation of $K$ changes the signs of both $\mu$ and $\lambda$, so  the identification between slopes on $K$ and $\overline{\mathbb{R}}$ is independent of the choice of orientation of $K$. If $L\subset S^3$ is a link with multiple components, we consider on each component the canonical meridian-longitude basis obtained as described above. 
\newline

Our first step in constructing foliations on surgeries of the Borromean rings $\mathcal{B}$ is to consider a different diagram of $\mathcal{B}$, depicted in \Cref{fig: borr discs}. We denote its components by $K_0, K_1,$ and $K_2$, where $K_0$ lies in the projection sphere, while $K_1$ and $K_2$ intersect the sphere transversely in two points, and bound two twice-punctured discs $D_1$ and $D_2$ (see \Cref{fig: borr discs} -- right). We denote the exterior of $\mathcal{B}$ by $M$, and let $\partial_i M$ be the boundary component corresponding to $K_i$, for $i = 0, 1, 2$.

\begin{figure}[h]
    \centering
    \includegraphics[width=0.65\textwidth]{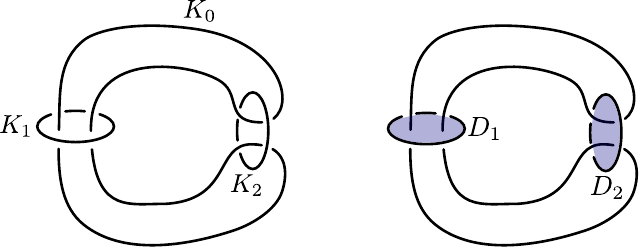}
    \caption{Left: a diagram of the Borromean rings. Right: the twice-punctured discs, $D_1$ and $D_2$, bounded by $K_1$ and $K_2$ respectively.}
    \label{fig: borr discs}
\end{figure}

The construction proceeds as follows. We start by defining, in \Cref{subsub example: Sigma}, a $2$-complex $\Sigma$ in $M$, and assign coorientations to the surfaces composing it. After having slightly perturbed $\Sigma$, we use the coorientations to smooth it to a branched surface $\overline{B}$ in $M$ with two \emph{meridional cusps} (as in  \Cref{def: cusps}) along a torus parallel to $\partial_0 M$, and that intersects the boundary components $\partial_1 M$ and $\partial_2 M$ in suitable train tracks. This construction is carried out in \Cref{subsub example: Bbar}. Finally, in \Cref{subsub example: B}, we use a branched surface $B$, obtained by a simple modification of $\overline{B}$, to provide the desired foliations.

\subsubsection{The $2$-complex $\Sigma$.}\label{subsub example: Sigma}
We begin by fixing a torus $T$ parallel to $\partial_0 M$. Specifically, we set $T=\partial_0 M\times \{1\}$, where $\nu\partial_0 M=\partial_0 M \times [0,1]$ is a fixed collar of $\partial_0 M$ with $\partial_0 M= \partial_0 M\times \{0\}$.
The $2$-complex $\Sigma$ is defined as the union of:

\begin{itemize}
\item The torus $T$.
\item The twice-punctured discs bounded in $M\setminus \big{(}{\partial_0 M \times [0,1)}\big{)}$ by $K_1$ and $K_2$. With a slight abuse of notation, we continue to denote this discs by $D_1$ and $D_2$.
\item The closures of the connected components of $S^2\setminus (D_1\cup D_2 \cup \nu\partial_0 M)$. We denote this set by $\mathcal{S}$.
\end{itemize}

\begin{figure}[]
    \centering
    \includegraphics[width=0.35\textwidth]{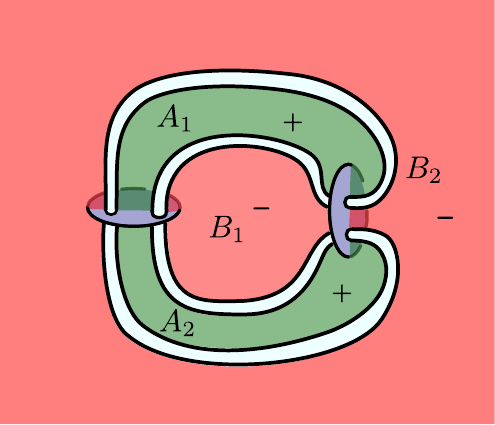}
    \caption{The figure illustrates the coorientations of the surfaces in $\Sigma$ that lie in the projection sphere $S^2$. The viewpoint is from the ball $X_1$ in the complement of $\Sigma$.}
    \label{fig: sec 3 coorientations in S^2}
\end{figure}

\begin{lemma}\label{lemma sec3: exterior of Sigma}
The complement of a regular neighbourhood of $\Sigma$ in $M$ is homeomorphic to the disjoint union of two balls, $X_1$ and $X_2$, and a product $T\times [0,1]$.
\end{lemma}
\begin{proof}
The complement of a regular neighbourhood of $\Sigma$ can be obtained by successively cutting $M$ along the surfaces composing $\Sigma$. We begin by cutting along the torus $T$. Since $T$ is parallel to the boundary component $\partial_0 M$, this produces two connected components: one is a product $T \times [0,1]$, and the other can be identified with $M$ itself. Next, we cut the latter component along the surfaces in $\mathcal{S}$. This is equivalent to cutting $S^3$ along the projection sphere, obtaining two balls, and then removing from each ball the tubular neighbourhoods of two properly embedded boundary-parallel arcs (one coming from $K_1$ and the other from $K_2$). Finally, in each of the two resulting components, the discs $D_1$ and $D_2$ become boundary-compressing discs for the boundary-parallel arcs. Cutting along them therefore yields two balls, denoted $X_1$ and $X_2$.
\end{proof}

We now assign coorientations to the surfaces composing $\Sigma$, that we will use later to guide the smoothing of $\Sigma$ into a cooriented branched surface $\overline{B}$.

The set $\mathcal{S}$ has four elements: two of them, denoted $B_1$ and $B_2$ intersect $\partial_1 M$, and the other two, denoted $A_1$ and $A_2$, intersect $\partial_2 M$. We assign coorientations so that the positive side of $A_1$ and $A_2$ lies in one of the balls in the complement of $\Sigma$, say $X_1$, while the positive side of $B_1$ and $B_2$ lies in $X_2$. This convention is illustrated in \Cref{fig: sec 3 coorientations in S^2}, where green (resp. red) indicates the positive (resp. negative) side of a component.

The torus $T$ intersects $D_1$ and $D_2$ in four circles, each isotopic to a meridian of the component $K_0$ of $\mathcal{B}$. We select two of these intersection circles, $c_1 \in T \cap D_1$ and $c_2 \in T \cap D_2$, and assign coorientations to $D_1$, $D_2$, and $T$ so that, after smoothing the $2$-complex $\Sigma$ according to these coorientations, two \emph{cusps} are created on $T$: one along $c_1$ and one along $c_2$.

\begin{defn}\label{def: cusps}
A \emph{cusp} on $T \subset \Sigma$ is a local smoothing of $\Sigma$ in a neighbourhood of a simple closed curve $c \subset T$ such that the cusp direction along $c$ points into $\Sigma \setminus T$.
\end{defn}

\begin{figure}[]
\centering
\includegraphics[width=0.7\textwidth]{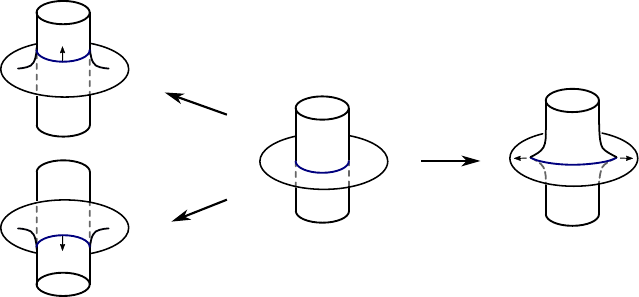}
\caption{The three possible smoothings of $\Sigma$ near a curve of intersection with a disc. The smoothing on the right presents a cusp on $T$, while the two on the left do not. Observe that in the smoothings on the left, the cusp directions point in the torus $T$.}
\label{fig: cusp}
\end{figure}

The right side of \Cref{fig: cusp} shows a cusp. We create two cusps by assigning suitable different coorientations to the two annuli obtained by cutting $T$ along $c_1$ and $c_2$. Our choice of coorientations is illustrated in \Cref{fig: coorientations of torus and discs}.

\begin{figure}[]
    \centering
    \includegraphics[width=0.3\textwidth]{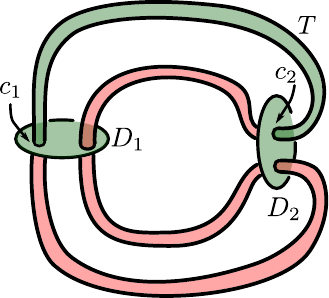}
    \caption{The figure shows how to coorient the discs $D_1$ and $D_2$, and the torus $T$. The curves $c_1$ and $c_2$ cut $T$ into two annuli. After the smoothing induced by the coorientations described in figure, we get two cusps along $c_1$ and $c_2$.}
    \label{fig: coorientations of torus and discs}
\end{figure}

\subsubsection{Desingularisation of the intersections between $\mathcal{S}$ and the discs $D_1$ and $D_2$.}\label{subsubsec: desingularisation}

We have fixed coorientations to all the surfaces composing $\Sigma$. Before smoothing, however, we have to choose a desingularisation of the intersections between the sectors in $\mathcal{S}$ and the discs $D_1$ and $D_2$. In fact, each disc intersects the elements in $\mathcal{S}$ along three arcs that do not respect the local model for branched surfaces. We solve this by sliding the surfaces in $\mathcal{S}$ along the discs in a neighbourhood of each intersection arc, thereby creating two arcs of double points\footnote{In doing so, we also create two triple points, arising from the intersections of the torus $T$, the discs, and the sectors in $S^2$.}. 

For each intersection arc, we consider two possible ways to perform this sliding, and different choices yield different branched surfaces when smoothing. In particular, on $D_1$ there are two arcs of intersection meeting the boundary component $\partial_1 M$. The way in which we modify $\Sigma$ along these arcs is important: before the modification, $\Sigma \cap \partial_1 M$ consists of two meridians and one longitude of $K_1$; after the modification, it becomes a $1$-complex in $\partial_1 M$ that, once smoothed, produces the boundary train track of the branched surface. Different choices of sliding lead to different train tracks, which realise different sets of slopes. The same considerations apply to $D_2$. The two possible choices of sliding, along with their effect on the boundary, are illustrated in \Cref{fig: possible slidings}.

\begin{figure}[]
    \centering
    \includegraphics[width=0.65\textwidth]{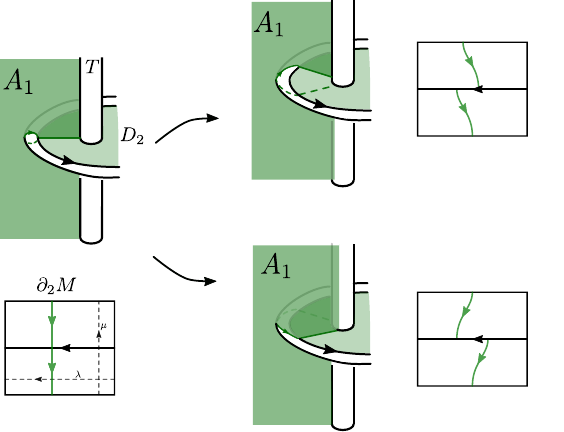}
    \caption{The figure shows a portion of the intersection between $D_2$ and an $A_1$. We illustrate the two possible ways of sliding $A_1$ along $D_2$ and their effect on the intersection with $\partial_2 M$, that we have also endowed with a meridian-longitude basis. The intersection with the boundary of $M$ inherits a coorientation (and thus an orientation) from the coorientation previously fixed on the sectors of $\Sigma$.}
    \label{fig: possible slidings}
\end{figure}

\subsubsection{The branched surface $\overline{B}$.}\label{subsub example: Bbar}
 Finally, we choose a desingularisation of the intersections between the elements in $\mathcal{S}$ and the discs $D_1$ and $D_2$, and smooth the resulting complex according to the fixed coorientations. This produces a cooriented branched surface $\overline{B}$ in $M$, shown in \Cref{fig: final example}, whose main properties are described in the following proposition.

\begin{prop}\label{prop: sec 3 from sigma to B}
The cooriented branched surface $\overline{B}$ has the following properties:
\begin{itemize}
\item $\overline{B}$ has two cusps on $T$.
\item $\overline{B}$ contains no sink discs or half sink discs.
\item $\partial \overline{B} \cap \partial_1 M$ realises every rational slope in  $(-\infty, 1)$. 
\item $\partial \overline{B} \cap \partial_2 M$ realises every rational slope in  $(-1, \infty)$. 
\end{itemize}
\end{prop}

\begin{figure}[]
    \centering
    \includegraphics[width=1\textwidth]{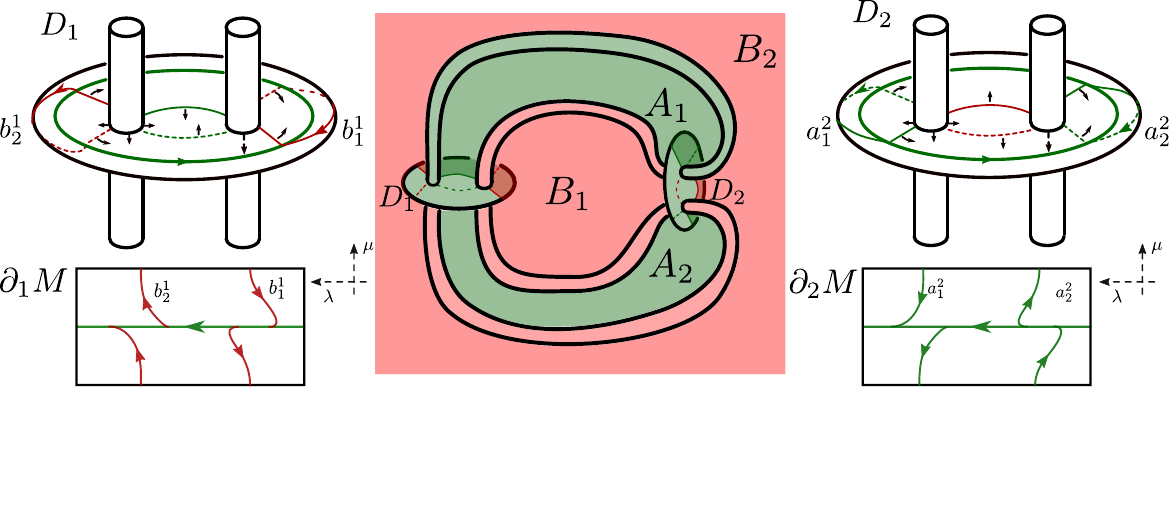}
    \caption{The $2$-complex $\Sigma$ after the desingularisation. By smoothing according to the coorientations we obtain $\overline{B}$. We also illustrate the branch locus of $\overline{B}$ on the discs $D_1$ and $D_2$. Each disc contains five half-disc sectors of $\overline{B}$: four of them are not half sink discs because they have an arc in their boundary with cusp direction points in $T$; the other one is not half sink disc because of our choice of desingularisation. The picture also shows the boundary train tracks of $\overline{B}$.}
    \label{fig: final example}
\end{figure}

\begin{proof}
The presence of the two cusps, one along $c_1$ and one along $c_2$, follows from our choice of coorientations on the discs and on the torus $T$.
We now check that $\overline{B}$ contains no sink discs or half sink discs. A sector of $\overline{B}$ is either an element of $\mathcal{S}$, or lies in one of the discs $D_1$ or $D_2$, or is contained in $T$. The branch locus of $\overline{B}$ consists of three types of curves:

\begin{itemize}
\item \emph{Intersections between sectors in $\mathcal{S}$ and $T$:} the cusp directions along these curves point into $T$.
\item \emph{Intersections between sectors in $\mathcal{S}$ and the discs $D_i$:} the cusp directions along these curves point into the discs $D_i$.
\item \emph{Intersections between the discs $D_i$ and $T$:} these are closed curves. Along $c_1$ and $c_2$, the cusp directions point into the discs; along the other intersections, they point into $T$.
\end{itemize}

The sectors of $\overline{B}$ contained in $\mathcal{S}$ are not half sink discs, since they all intersect $T$ and the cusp directions along these intersections point outward from them. Each disc $D_i$ contains five sectors, as shown in \Cref{fig: final example}. None of these are half sink discs, either because they have a boundary arc contained in a circle of intersection of $T$ and $D_i$ that is not a cusp, or because our choice of desingularisation creates cusp directions pointing outward.
Finally, each sector in $T$ has two boundary arcs of intersection with elements in $\mathcal{S}$. One of these arcs has a cusp direction pointing into the sector and the other has it pointing outward. This is because the intersecting sectors $S^2$ have opposite coorientations by construction (see \Cref{fig: example sectors in torus}). Hence, there are no half sink discs on $T$.

We now turn our attention to the boundary train tracks of $\overline{B}$. We orient them  using the coorientation induced by that of $\overline{B}$. Our choice of desingularisation yields the train tracks depicted in \Cref{fig: final example}. With reference to the notation in \Cref{fig:train track table}:
\begin{itemize}
\item  The train track $\partial \overline{B}\cap \partial_1 M$ is a train track of type $a) + d)$, and hence, by \Cref{lemma: unlinked}, realises every rational slope in $(-\infty, 1)$.
\item The train track $\partial \overline{B}\cap \partial_2 M$ is of type $b)+c)$, and hence realises every rational slope in $(-1, \infty)$.
\end{itemize}
This concludes the proof.
\end{proof}

\begin{figure}[h]
    \centering
    \includegraphics[width=0.18\textwidth]{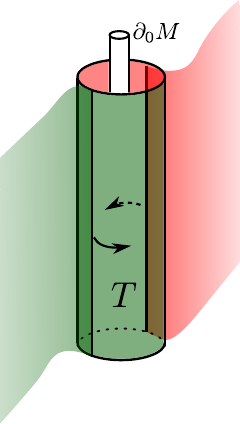}
    \caption{Cusps directions along the intersection between $T$ and the sectors in $S^2$. If the coorientation on the torus is reversed, the cusp directions are reversed too. Thus, in both cases, no sectors in $T$ are half sink discs.}
    \label{fig: example sectors in torus}
\end{figure}
We now describe the exterior of $\overline{B}$ as a sutured manifold. Topologically, the exterior of $\overline{B}$ is homeomorphic to the exterior of $\Sigma$, and hence decomposes into two balls, $X_1$ and $X_2$, together with a component $Y$ homeomorphic to $T\times [0,1]$ containing the boundary component $\partial_0 M$. For the sutured manifold $(Y, \gamma_Y)$, recall that by definition we have

$$
A(\gamma_Y)=Y\cap \partial_v N_{\overline{B}},
$$
and we denote this set by $\partial_v Y$. The presence of the two cusps on $T$ implies that, as a sutured manifold,
$$ (Y,\partial_v Y)=(\nu \partial_0 M, \mathcal{A}_1\cup \mathcal{A}_2), $$
where $\mathcal{A}_1\cup \mathcal{A}_2\subset \partial_0 M \times \{1\}$ are two annuli of slope equal to a meridian of $K_0$. 

To study $X_1$ and $X_2$, we introduce some notation. 
We define $\mathcal{R}^1_+$ as the set of sectors in $\overline{B}$ whose positive side is contained in $X_1$, and $\mathcal{R}^1_-$ as the set of sectors in $\overline{B}$ whose negative side is contained in $X_1$. Moreover, we denote $\mathcal{S}^1_+$ (resp. $\mathcal{S}^1_-$) the set of sectors in $\mathcal{R}^1_+$ (resp. $\mathcal{R}^1_-$) that are in $\mathcal{S}$. Analogously, define $\mathcal{R}^2_{\pm}$ and $\mathcal{S}^2_{\pm}$ using $X_2$. Observe that, by definition, we have that $\mathcal{S}^1_{+}=\mathcal{S}^2_{-}=\{A_1, A_2\}$, and that $\mathcal{S}^1_{-}=\mathcal{S}^2_{+}= \{B_1, B_2\}$.

\begin{lemma}\label{lemma: sec 3 exterior of B}
The sutured manifold $M\setminus {\rm int}(N_{\overline{B}})$ is equal to $X_1\sqcup X_2 \sqcup Y$, where $X_1$ and $X_2$ are product sutured balls and $(Y,\partial_v Y)=(\nu \partial_0 M, \mathcal{A}_1\cup \mathcal{A}_2)$, with $\mathcal{A}_1\cup \mathcal{A}_2\subset \partial_0 M \times \{1\}$ two annuli whose slope is a meridian of $K_0$.
\end{lemma}

\begin{proof}
We already observed that $(Y,\partial_v Y)$ has the desired property, and that $X_1$ and $X_2$ are topologically two balls.
Regarding the sutured manifold $(X_1,\gamma_{X_1})$, we denote 
$$
R(\gamma_{X_1})=\partial_h N_{\overline{B}}\cap X_1=R^1_{+}\sqcup R^1_{-},
$$ where $R^1_{+}$ is the subsurface of $\partial_h N_{\overline{B}}\cap X_1$ along which the coorientation points into $X_1$, and $R^1_{-}$ is the subsurface along which the coorientation points out of $X_1$.
By virtue of \Cref{rem: product ball}, if we prove that $R^1_{+}$ and $R^1_{-}$ are connected, then $(X_1,\gamma_{X_1})$ is a product sutured ball. 
Topologically, the set $R^1_+$ is the union of all the positive sides of sectors in $\mathcal{R}^1_+$. 
Let $F$ be one of such sectors and notice that:
\begin{itemize}
\item If $F$ is contained in $D_1$ or $D_2$, then its positive side can be connected by a path in $R_+$ to the positive side of some element in $\mathcal{S}^1_+$. In fact, the positive side of each of these discs intersects three sectors in $\mathcal{S}$. By our convention on the coorientations of the sectors in $\mathcal{S}$, at least one of them is in $\mathcal{S}^1_+$, and its positive side can be connected to the one of $F$ by an arc passing through the positive side of the disc. 
\item If $F$ is contained in $T$, then its positive side intersects the positive side of some element in $\mathcal{S}^1_+$. In fact, each sector in $T$ intersects two sectors in $\mathcal{S}$, and exactly one of these is in $\mathcal{S}^1_+$.
\end{itemize}
This implies that, in order to prove that $R_+$ is connected, it is enough to prove that the elements in $\mathcal{S}^1_+$, \ie, $A_1$ and $A_2$, can be connected by some arc lying in $R_+$. It is very easy to check that this is true. In fact, both $A_1$ and $A_2$ intersect $\partial_2 M$, and they can be connected with an arc passing through the positive side of $D_2$ (see \Cref{fig: S_+ is connected}). Analogously, $R_-$ is the union of all the negative sides of sectors in $\mathcal{R}^1_-$. With a similar argument, using now that $B_1$ and $B_2$ intersect $\partial_1 M$, and hence can be connected via an arc lying on the negative side of $D_1$, one proves that $R_-$ is connected. So we have that $R_+$ and $R_-$ are connected, and, by \Cref{rem: product ball}, $(X_1, \gamma_{X_1})$ is a product sutured ball. With the same considerations, one gets that $(X_2, \gamma_{X_2})$ is a product sutured ball. 
\end{proof}
\begin{figure}[h]
    \centering
    \includegraphics[width=0.25\textwidth]{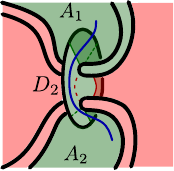}
    \caption{The picture illustrates how the disc $D_2$ allows one to connect the positive sides of $A_1$ and $A_2$ by an arc, shown in blue, contained in $R_+$.}
    \label{fig: S_+ is connected}
\end{figure}

The proof of \Cref{lemma: sec 3 exterior of B} naturally suggests introducing certain graphs, which will also be useful in the more general setting of \Cref{sec: main theorem}.
We define $\mathcal{G}^1_+$ and $\mathcal{G}^1_-$ in the following way.
\begin{itemize}
\item The vertices of $\mathcal{G}^1_{\pm}$ are the elements in $\mathcal{S}^1_{\pm}$.
\item For every intersection of two elements of $\mathcal{S}^1_{\pm}$ with the same connected component of $\partial_1 M \sqcup \partial_2 M$, we add an edge between the corresponding vertices of $\mathcal{G}^1_{\pm}$.
\end{itemize}
Similarly, using $\mathcal{S}^2_{\pm}$, we define $\mathcal{G}^2_+$ and $\mathcal{G}^2_-$. With this notation, we can restate the argument presented in the proof of \Cref{lemma: sec 3 exterior of B} as follows:

\begin{lemma}
For $i=1,2$, If both $\mathcal{G}^i_+$ and $\mathcal{G}^i_-$ are connected, then $(X_i, \gamma_{X_i})$ is a product sutured ball. \qed
\end{lemma}

\subsubsection{The branched surface $B$.}\label{subsub example: B}
Finally, we are ready to define the branched surface $B$ that will be used to construct the desired taut foliations. It is obtained from $\overline{B}$ in a very simple way, namely by setting
$$
B=\overline{B}\setminus{A_1}.
$$
Notice that since all cusp directions along the boundary of $A_1$ point outward from it, $B$ is a well-defined branched surface, \ie, satisfies the local models of \Cref{branched surface}. Also observe that removing $A_1$ modifies the boundary train track; in fact, $\partial B\cap \partial_2 M$ is obtained by removing the branch $a^2_1$, depicted in \Cref{fig: final example}, from $\partial \overline{B}\cap \partial_2 M$. The train track on $\partial_1 M$ is left unchanged, since $A_1$ does not intersect it.

\begin{prop}\label{prop: sec3 from B to B'} 
The branched surface $B$ satisfies the following properties:

\begin{itemize}

\item $B$ has two cusps on $T$.
\item $B$ contains no sink discs or half sink discs.
\item $B \cap \partial_1 M$ realises every rational slope in  $(-\infty, 1)$. 
\item $B \cap \partial_2 M$ realises every rational slope in  $(0, \infty)$. 

\end{itemize}
Moreover, the exterior $M\setminus {\rm int}(N_{B})$, as a sutured manifold, is homeomorphic to $X\cup Y$, where:

\begin{itemize}
\item $X$ is a product sutured ball. 
\item $(Y,\partial_v Y)=(\nu \partial_0 M, \mathcal{A}_1\cup \mathcal{A}_2)$, where $\mathcal{A}_1\cup \mathcal{A}_2\subset  \partial_0 M \times \{1\}$ are two annuli whose slope is a meridian of $K_0$.
\end{itemize}
\end{prop}

\begin{proof}
The first properties are straightforward and they hold for exactly the same reasons presented in the proof of \Cref{prop: sec 3 from sigma to B}. Notice that the train track $\partial B\cap \partial_2 M$ is now of type $b)$, and hence it realises every rational slope in $(0,\infty)$ (see \Cref{fig:train track table}). 
Hence we are left to prove the statement regarding the structure of $M\setminus {\rm int}(N_{B})$ as a sutured manifold.  Removing a sector in $A_1$ from $\overline{B}$ has the effect of gluing the two balls $X_1$ and $X_2$ in the exterior of $\overline{B}$ along a subdisc of their boundary. As a consequence, we have that $M\setminus {\rm int}(N_{B})= X\sqcup Y$ where $X$ is topologically a ball and $(Y, \partial_v Y)$ is homeomorphic to $(\nu \partial_0 M, A_1\cup A_2)$, where $A_1\cup A_2\subset \partial_0 M \times \{1\}$ are two annuli produced by the two meridional cusps on $T$.

Similar to the proof of \Cref{lemma: sec 3 exterior of B}, we denote 
$$
\partial_h N_{B}\cap X=R_{+}\sqcup R_{-},
$$ where $R_{+}$ is the subsurface of $\partial_h N_{B}\cap X$ along which the coorientation points into $X$ and $R_{-}$ is the subsurface along which the coorientation points out of $X$. If we show that these two sets are connected then the statement will follow from \Cref{rem: product ball}. 
We focus our attention on $R_+$. This is homeomorphic to the union of all the positive sides of the sectors of $B$ whose positive side faces $X$. We denote by $\mathcal{S}_+$ the set whose elements are the sectors in $\mathcal{S}$ whose positive sides is contained in $X$. Notice that 
$$
\mathcal{S}_+=\mathcal{S}\setminus\{A_1\}=\{A_2,B_1,B_2\}.
$$
With reasoning analogous to that in the proof of \Cref{lemma: sec 3 exterior of B}, one shows that to prove that $R_+$ is connected, it is sufficient to show that the positive sides of any two elements in $\mathcal{S}_+$ can be connected by a path in $R_+$. This is the case for $B_1$ and $B_2$, since they both intersect $\partial_1 M$, and we can find an arc in $R_+$ that passes through $D_1$ with endpoints in $B_1$ and $B_2$. A similar argument applies to $A_2$ and $B_2$: they both intersect $D_2$ and, since the sector $A_1$ has been removed, we can use $D_2$ to find the desired arc lying in $R_+$ connecting $A_2$ and $B_2$, as shown in \Cref{fig: S is connected}. 

\begin{figure}[h]
    \centering
    \includegraphics[width=0.25\textwidth]{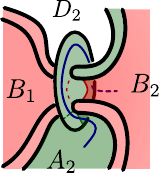}
    \caption{The removal of the sector $A_1$ allows us to connect the positive sides of $A_2$ and $B_2$ with an arc, blue in the picture, lying in $R_+$.}
    \label{fig: S is connected}
\end{figure}
By transitivity, also $B_1$ and $A_2$ can be connected in $R_+$. Hence $R_+$ is connected and, with analogous arguments, one shows that also $R_-$ is connected. This implies that $X$ is a product sutured ball, and concludes the proof.
\end{proof}

By studying the branched surface $B$ further, one shows that $B$ is laminar and proves the following theorem, which relies on a combination of \Cref{boundary train tracks} and \Cref{cusps implies persistently foliar}. We will present these arguments —- in a slightly more general setting —- in \Cref{sec: main theorem}. We believe that presenting them specifically in this example does not provide much additional intuition, though it may still be helpful to keep this example in mind while reading. To avoid repetition, we limit ourselves here to stating only what is necessary to proceed and refer the reader to \Cref{sec: main theorem} for the full details. 

\begin{teo}\label{thm: sec 3 general main thm borr}
Let $s$ be any non-meridional slope on $\partial_0 M$ and let $r_1$ and $r_2$ be rational slopes on $\partial_1 M$ and $\partial_2 M$ respectively. Suppose that the multislope $(r_1, r_2)$ is realised by the boundary train track $\partial B$, and that $r_i$ is not the longitude defined by the disc $D_i$, for $i=1, 2$. Then, there exists a coorientable taut foliation in $M$ intersecting the boundary transversely in a linear foliation of multislope $(s,r_1, r_2)$. In particular, when $s$ is rational, the manifold $M(s, r_1, r_2)$ obtained by filling $M$ along the multislope $(s, r_1, r_2)$ supports a coorientable taut foliation.
\end{teo}

\begin{proof}
See \Cref{thm: general main thm}, which relies on the preliminary results \Cref{lemma: no closed surfaces and annuli}, \Cref{prop: B' is laminar}, and \Cref{lemma: B is taut}. 
\end{proof}

We are now ready to prove \Cref{thm: Borromean}.

\begin{proof}[Proof of \Cref{thm: Borromean}]

The statement regarding $L$-space surgeries follows from \cite[Proposition~9.8]{KMOS07}\footnote{Here the authors work with monopole $L$-spaces, but the monopole Floer/Heegaard Floer correspondence, established in the works of Kutluhan–Lee–Taubes \cite{KLT1, KLT2, KLT3, KLT4, KLT5}, implies that a manifold is a monopole $L$-space if and only if it is a (Heegaard Floer) $L$-space.} or from \cite[Lemma~2.6]{S}, by observing that the Borromean rings is a Brunnian link and that $(1,1,1)$-surgery on it is the Poincaré homology sphere, and therefore an $L$-space. These results, together with the amphichirality of $\mathcal{B}$, imply that if $(r_1,r_2,r_3)\in [1, \infty)^3 \cup (-\infty, -1]^3$, then the $(r_1,r_2,r_3)$-surgery on $\mathcal{B}$ is an $L$-space. 

Since manifolds supporting coorientable taut foliations are not $L$-spaces, to prove the theorem it is enough to construct taut foliations on all the other surgeries. To this end, we use the branched surface $B$ together with \Cref{thm: sec 3 general main thm borr}. In fact, by \Cref{prop: sec3 from B to B'}, the boundary train tracks of $B$ on $\partial_1 M$ and $\partial_2 M$ realise all rational multislopes in $(-\infty,1)\times(0,\infty)$. Therefore, by applying  \Cref{thm: general main thm}, we deduce that all surgeries on $\mathcal{B}$ with coefficients in $(-\infty,\infty)\times\big((-\infty,1)\setminus\{0\}\big)\times(0,\infty)$ support coorientable taut foliations. Since $\mathcal{B}$ is amphichiral and any permutation of its components can be realised by an isotopy of the link, one concludes the proof in the case where at most one surgery coefficient is zero.

When at least two surgery coefficients are zero, we may assume by symmetry that we are studying surgeries of type $(r,0,0)$, for some rational $r$. In this case, the branched surface obtained by splitting $B$ in a neighbourhood of $\partial_1 M$ and $\partial_2 M$, and gluing meridional discs as in the proof of \Cref{thm: general main thm}, is laminar in the resulting manifold. This implies that the manifold is irreducible, and since it has $b_1>0$, we conclude by \cite{Gabai} that it admits a coorientable taut foliation.
\end{proof}

We also point out the following immediate corollary of the proof of \Cref{thm: Borromean}, what will be useful in \Cref{subsec: Bing doubles}.

\begin{cor}\label{lemma: fol Bing}
Let $l\in \Z$ and let $U$ be a neighbourhood of $\frac{1}{l}$ in $\overline{\Q}$. Then, for every pair of rational numbers $r_1 \neq 0$ and $r_2 \neq 0$, there exist $r \in U$ and a coorientable taut foliation in the exterior of $\mathcal{B}$ that intersects the boundary tori in circles of slopes $r$, $r_1$, and $r_2$, respectively.\qed
\end{cor}

\subsection{The link $L8n5$.}\label{subsec: the link L8n5}
We now focus on the study of the three-component link $L8n5$. In doing this, we will present a variation of the ideas of the previous section that will appear also in the proof of the main theorem of the paper, in \Cref{sec: main theorem}. Also, in \Cref{sec: main theorem} we will need some of the resuts proved here.

The link $L8n5$ is depicted in \Cref{fig: borr twists}. We denote its exterior by $M'$, with boundary component $\partial_i M'$ associated to the link component $K'_i$, for $i=0,1,2$.

\begin{figure}[h]
    \centering
    \includegraphics[width=0.4\textwidth]{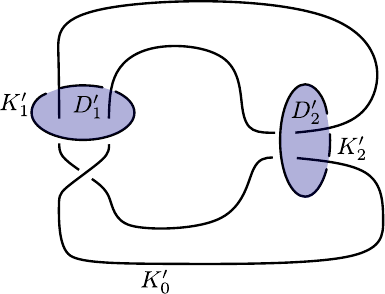}
    \caption{The link $L8n5$, together with the twice-punctured discs bounded by $K'_1$ and $K'_2$.}
    \label{fig: borr twists}
\end{figure}

The basic observation is that the exterior of $L8n5$ and the exterior of the Borromean rings are very much related. They are not diffeomorphic, but they are related by a \emph{mutation}, \ie, by the operation of cutting along a surface and regluing with a (possibly different) diffeomorphism. In fact, we obtain $M'$ by cutting $M$ along the twice punctured disc $D_1$ and gluing back with a $\pi$-rotation. In particular, we can define a $2$-complex $\Sigma'$ in the exterior of $L8n5$ by taking the image of the $2$-complex $\Sigma$ in the exterior $M$ of the Borromean rings after the following procedure:

\begin{itemize}
\item  Coorient $D_1\subset M$ as in \Cref{fig: coorientations of torus and discs} and use the coorientation to identify a neighbourhood of $D_1$ with $\mathbb{D}_*^2\times [-2,2]$. Here, $\mathbb{D}_*^2$ denotes the unit disc in $\mathbb{C}$ with two small discs of the same radius removed around the points $(-\frac{1}{2},0)$ and $(\frac{1}{2},0)$. We identify $D_1$ with $\mathbb{D}_*^2\times \{0\}$ and suppose that each slice $\mathbb{D}_*^2\times \{t\}$ is properly embedded in $M$. Moreover, we suppose that the intersection of the projection sphere with this tubular neighbourhood is contained in $\mathbb{R}\times [-2,2]\subset \mathbb{C}\times [-2,2]$.
\item Call $N=M\setminus{\mathbb{D}_*^2\times (-1,1)}$ and apply the self-diffeomorphism of $N$ that is a rotation of angle $(t+2)\pi$ on each slice $\mathbb{D}_*^2\times \{t\}$, for $t\in [-2,-1]$, and is the identity otherwise.
\item Glue back $\mathbb{D}_*^2\times \{-1\}$ to $\mathbb{D}_*^2\times \{1\}$ with the identity.
\end{itemize}

Roughly speaking, the $2$-complex $\Sigma'$ differs from $\Sigma$ only in a neighbourhood of $D_1$ and is obtained by $\Sigma$ by twisting it in a neighbourhood of $D_1$, so to follow the half-twist that makes $L8n5$ differ from the Borromean rings.

The $2$-complex $\Sigma'$ consists of: a torus $T'$, parallel to the component $K'_0$; the twice-punctured discs $D'_1$ and $D'_2$; the set $\mathcal{S}'$ containing four regions $A'_1,A'_2, B'_1,$ and $B'_2$, that are the images of the respective sectors in $\mathcal{S}$ of $\Sigma$. A partial picture of $\Sigma'$ in a neighbourhood of $D'_1$ is depicted in the top-left part of \Cref{fig: branched twist}.

We now have to make $\Sigma'$ into a branched surface. To do this, we first coorient the sectors of $\Sigma'$ with the coorientation induced by $\Sigma$. Then, we desingularise the intersections between $D'_2$ and the sectors in $S^2$ exactly as we did on $D_2$ in \Cref{fig: final example}, while on $D'_2$ we choose the desingularisation described in \Cref{fig: branched twist}. In this way we obtain a branched surface $\overline{B}'$, and, as we did with the case of the Borromean rings, we remove the sector $A'_1$, to get a branched surface $B'$. Notice that, by construction, the exterior of $B'$ in $M'$ is homeomorphic, as a sutured manifold, to the exterior of $B$ in $M$ and hence the proof of the following proposition is greatly simplified.

\begin{figure}[h]
    \centering
    \includegraphics[width=0.75\textwidth]{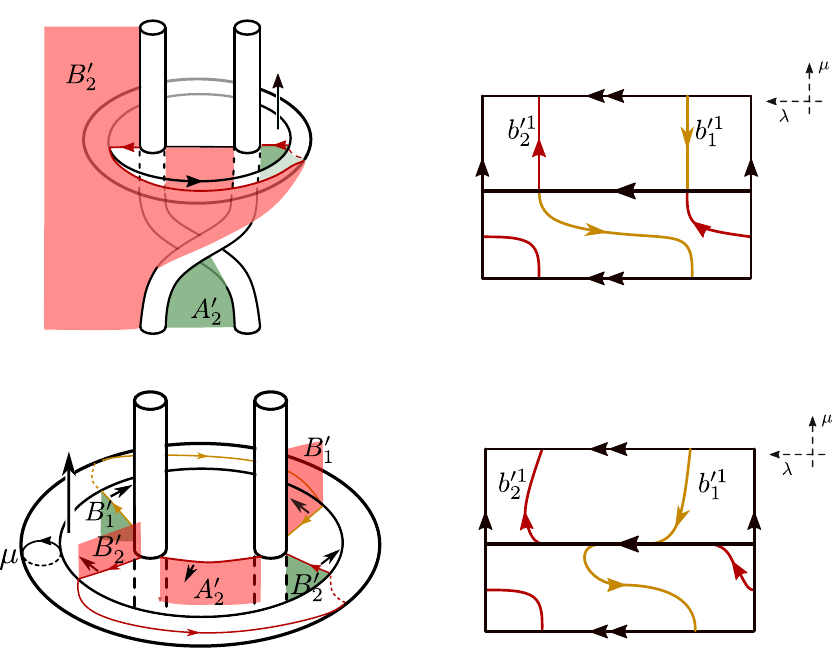}
    \caption{Top: a portion of the $2$-complex $\Sigma' \setminus A'_1$ in a neighbourhood of the disc $D'_1$, together with its intersection with $\partial_1 M'$. For simplicity, we do not draw the sector $B'_1$, but we show its intersection with the boundary in yellow. Bottom: the branched surface $B'$ and its boundary train track on $\partial_1 M'$.}
    \label{fig: branched twist}
\end{figure}

\begin{prop}\label{prop: sec3 from B to B' twist} 
The branched surface $B'$ satisfies the following properties:

\begin{itemize}

\item $B'$ has two cusps on $T$.
\item $B'$ contains no sink discs or half sink discs.
\item $B' \cap \partial_1 M'$ realises every rational slope in  $(-\infty, 1)$. 
\item $B' \cap \partial_2 M'$ realises every rational slope in  $(0, \infty)$. 

\end{itemize}
Moreover, the exterior $M'\setminus {\rm int}(N_{B'})$, as a sutured manifold, is homeomorphic to $X'\cup Y'$, where:

\begin{itemize}
\item $X$ is a product sutured ball. 
\item $(Y',\partial_v Y')=(\nu \partial_0 M', \mathcal{A}_1\cup \mathcal{A}_2)$, where $\mathcal{A}_1\cup \mathcal{A}_2\subset  \partial_0 M' \times \{1\}$ are two annuli whose slope is a meridian of $K'_0$.
\end{itemize}
\end{prop}

\begin{proof}
The presence of the two cusps on $T$ is a consequence on the choice of coorientations. Since branched surface $B'$, outside of a neighbourhood of $D'_1$, can be identified with $B$, we only need to check that no half sink discs in $D'_1$ are created when smoothing near $D'_1$. This smoothing is illustrated in the bottom-left part of \Cref{fig: branched twist}, and one can directly check that there are no half sink discs. This proves the first two items of the proposition. \Cref{fig: branched twist} also shows the boundary train track on $\partial_1 M'$, which is of type $a)+d)$, and hence realises every rational slopes in $(-\infty, 1)$. The boundary train track on $\partial_2 M'$ coincide with the boundary train track of $B$ on $\partial_2 M$ and we already showed in \Cref{prop: sec3 from B to B'}, that it realises every rational slope in $(0,\infty)$.

Finally, the statement about the exterior of $B'$ follows from the fact that it is homeomorphic as a sutured manifold to the exterior of $B$, for which they hold by virtue of \Cref{prop: sec3 from B to B'}.
\end{proof}

We can now state the analogous of \Cref{thm: sec 3 general main thm borr} for the manifold $M'$. We again refer the reader to \Cref{thm: general main thm} for its proof.

\begin{teo}\label{thm: sec 3 general main thm borr twist}
Let $s$ be any non-meridional slope on $\partial_0 M'$ and let $r_1$ and $r_2$ be rational slopes on $\partial_1 M'$ and $\partial_2 M'$ respectively. Suppose that the multislope $(r_1, r_2)$ is realised by $B'$, and that $r_i$ is not the longitude defined by the disc $D'_i$, for $i=1, 2$. Then, there exists a coorientable taut foliation in $M'$ intersecting the boundary transversely in a linear foliation of multislope $(s,r_1, r_2)$. In particular, when $s$ is rational, the manifold $M'(s, r_1, r_2)$ obtained by filling $M'$ along the multislope $(s, r_1, r_2)$ supports a coorientable taut foliation.
\end{teo}

We conclude this section by proving the following result.

\begin{prop}\label{prop: foliations on L8n5}
Let $(s,r_1,r_2)$ be a rational multislope on $L8n5$, where $s\in\Q$, $r_1\in (-\infty,1)\setminus\{0\}$, and $r_2\in (0,\infty)$. Then the $(s, r_1,r_2)$-surgery on $L8n5$ supports a coorientable taut foliation.
\end{prop}
\begin{proof}
It is a consequence of \Cref{thm: sec 3 general main thm borr twist} and the fact that, by \Cref{prop: sec3 from B to B' twist}, the boundary train track $\partial B'\cap \partial_1 M'$ realises every rational slope in $(-\infty,1)$, and $\partial B'\cap \partial_2 M'$ realises every rational slope in $(0,\infty)$.
\end{proof}

\subsection{Bing doubles.}\label{subsec: Bing doubles}
We conclude the section by presenting some applications of the previous results to Bing doubles of knots and links. We denote the solid torus $\matD^2\times S^1$ by $V$, and recall that a Bing double of a knot $K$ is a satellite link of $K$ whose companion is the link $P\subset V$ depicted in  \Cref{fig: Bing Pattern}. More precisely, fix a closed tubular neighbourhood $\nu K$ of $K$ and an orientation preserving diffeomorphism $\phi:\mathbb{D}^2\times S^1\to \nu K$. Then, $\phi(P)$ defines a two-component link with linking number zero in $S^3$, that by definition is a Bing double of $K$. Observe that changing the diffeomorphism $\phi$ can change the resulting link.

\begin{figure}[h]
    \centering
    \includegraphics[width=0.3\textwidth]{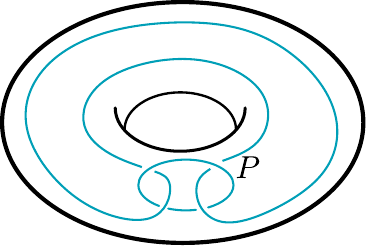}
    \caption{The pattern used to define Bing doubles. The exterior of $P$ in the solid torus coincide with the exterior of the Borromean rings in $S^3$.}
    \label{fig: Bing Pattern}
\end{figure}
We indicate by $\mu_V$ and $\lambda_V$ the oriented meridian and longitude of $V$. Observe that, if we call $K_0$ the core of the exterior of the standard solid torus $V$ in $S^3$, then $K_0\cup P$ is the Borromean rings $\mathcal{B}$. Moreover, 
$$
(\mu_V, \lambda_V)=(\lambda_{K_0}, \mu_{K_0}),
$$ where $\mu_{K_0}$ and $\lambda_{K_0}$ are the canonical meridian and longitude of $K_0$.

\begin{teo}\label{thm: Bing doubles}
Let $L$ be a fibered link with positive genus or any non-trivial knot and let $L'$ denote the link obtained by replacing each component of $L$ with one of its Bing doubles. Then every rational homology sphere obtained by doing surgery on each component of $\mathcal{L}'$ along a non-meridional slope supports a coorientable taut foliation.
\end{teo}

\begin{proof}
The proof is analogous to the proof of \cite[Theorem~1.10]{Stwobridge} where the same result was proved for Whitehead doubles and some more general satellite knots and links. 

We first analyse the case where $L$ is a non-trivial knot, and we denote it by $K$. Notice that the surgeries on the Bing double $\phi(P)$ of $K$ are parametrised by the rational multislopes $(r_1, r_2)$ on $P$, where we are fixing the canonical meridian-longitude bases of its components, seen as knots in $S^3$. Also observe that, since the $\phi(P)$ has linking number zero, the surgery associated to $(r_1, r_2)$ is a rational homology sphere if and only if $r_1\ne 0$ and $r_2\ne 0$. For this reason, we will assume that both $r_1$ and $r_2$ are non-zero.

We fix the canonical meridian-longitude basis $(\mu_K, \lambda_K)$ for the knot $K$ and we use it to identify rational slopes on $K$ with $\overline{\Q}=\matQ\cup \{\infty\}$. Keeping the notation introduced at the beginning of the section, the exterior of the Bing double of $K$ is obtained as 
$$
E_K\cup_{\varphi}M,
$$
where $E_K$ and $M$ are the exteriors of $K$ and $\mathcal{B}$ respectively, and $\varphi$ is the restriction of $\phi$ to $\partial V$.

In order to construct foliations on surgeries on the Bing double of $K$, our strategy is to find, for each pair of rationals $r_1\ne 0, r_2\ne 0$, a taut foliation $\mathcal{F}_K$ on $E_K$, and a taut foliation on $\mathcal{F}_{\mathcal
B}$ on $M$ such that:
\begin{itemize}
\item both foliations are transverse to the boundary, and the map $\varphi$ glues leaves of $\mathcal{F}_K$ to leaves of $\mathcal{F}_{\mathcal
B}$.
\item $\mathcal{F}_{\mathcal
B}$ intersects the boundary components associated to $P$ in a foliation by circles of slope $r_1$ and $r_2$ respectively.
\end{itemize}

The map $\varphi$
identifies the meridian $\mu_V$ of $V$ with $\mu_K$, and the longitude $\lambda_V$ of $V$ with some longitude of $K$, \ie,
\begin{align*}
    &\varphi(\lambda_{K_0})=\varphi(\mu_V)=\mu_K\\
    &\varphi(\mu_{K_0})=\varphi(\lambda_V)=l\mu_K+\lambda_K
\end{align*}
for some integer $l\in \matZ$.
Given two coprime integers $p,q$ the map $\varphi$ satisfies
$$
p\mu_K+q\lambda_K=\varphi\big{(}(p-ql)\lambda_{K_0}+q\mu_{K_0}\big{)}
$$
and therefore it acts on  slopes by identifying the slope $\frac{p}{q}$ on $K$ with the slope $(\frac{p}{q}-l)^{-1}$ on $K_0$.

By \cite[Theorem~1.1]{LR}, if $K$ is a non-trivial knot then there exists an interval $(-a, b)$, where $a,b>0$, such that for every rational slope $s\in (-a, b)$ there is a coorientable taut foliation on the exterior of $K$ intersecting the boundary torus transversely in a collection of circles of slope $s$. The slope $0$ on $K$ corresponds, via the identification given by $\varphi$, to the slope $-\frac{1}{l}$ on $K_0$. Hence, the interval $(-a,b)$ is identified with a neighbourhood $U$ of $-\frac{1}{l}\subset \overline{\matQ}$. 

It follows from \Cref{lemma: fol Bing} that for every integer $l \in \mathbb{Z}$, every neighbourhood $U$ of $-\frac{1}{l}$ in $\overline{\mathbb{Q}}$, and every rational multislope $(r_1, r_2)$ on $P$ with $r_1 \neq 0$ and $r_2 \neq 0$, there exist a rational number $r \in U$ and a coorientable taut foliation $\mathcal{F}_{\mathcal{B}}$ in $M$ that intersects the boundary tori in circles of slopes $r$, $r_1$, and $r_2$, respectively. Since $r$ lies in the neighbourhood $U$, by \cite{LR} we can find a taut foliation $\mathcal{F}_K$ in $E_K$ whose leaves can be glued to those of $\mathcal{F}_{\mathcal{B}}$ via $\varphi$, yielding the desired result.

When $L=K_1\sqcup \dots \sqcup K_d$ is a fibered link with multiple components and positive genus we can proceed in an analogous way. Let $S$ denote the fiber surface for $\mathcal{L}$. By intersecting $S$ with the boundaries of tubular neighbourhoods of the knots $K_1,\dots, K_d$ we obtain longitudes $\lambda^S_1,\dots \lambda^S_d$. We use them to define meridian-longitude bases for the components of $L$ and to identify rational slopes on $L$ with $\overline{\matQ}^d$.
By \cite[Theorem~1.1]{KR1} there exists, for every rational multislope $(r_1,\dots, r_d)$ in a neighbourhood of $0\in \overline{\matQ}^d$,  a coorientable taut foliation in the exterior of $L$ intersecting the boundary tori transversely in parallel curves of slopes $r_1,\dots, r_d$ respectively. The result now follows by applying to each component of $L$ the same reasoning as in the previous case, where we never made use of the fact that $\lambda_K$ was the canonical longitude of $K$.
\end{proof}

\section{Proof of the main theorem}\label{sec: main theorem}
This section is devoted to the proof of the main theorem of the paper. Recall from \Cref{sec: intro} that we described how to associate two weighted graphs, $\mathcal{G}_g$ and $\mathcal{G}_r$, to a knot diagram $D$.

\begin{teo}\label{thm: main theorem}
Let $K$ be a knot in $S^3$ with a diagram $D$, and suppose that $D$ has more than one twist region. Assume that:
\begin{itemize}
\item All weights of the graphs $\mathcal{G}_r$ and $\mathcal{G}_g$ are greater than one, and at least one weight is greater than two.
\item The graphs $\mathcal{G}_r$ and $\mathcal{G}_g$ are connected.
\end{itemize}
Then $K$ is persistently foliar. In particular, every non-trivial surgery on $K$ admits a coorientable taut foliation.
\end{teo}
\begin{proof}
When the graph $\mathcal{G}_r \cup  \mathcal{G}_g$ has at most four vertices, the result follows from \Cref{prop: at most four vertices}. When it has more than four vertices, from \Cref{thm: main thm parte 1}.
\end{proof}

The above proof is divided in two cases due to a technical reason; see \Cref{lemma: settore che non interseca} and the discussion preceding it.

We record here the following lemma, which provides an alternative way to verify the second condition of \Cref{thm: main theorem}.
\begin{lemma}\label{lemma: grafi connessi sse contrattili}
The following facts are equivalent:
\begin{enumerate}
\item The graphs $\mathcal{G}_g$ and  $\mathcal{G}_r$ are connected.
\item The graph $\mathcal{G}_g$ is a tree.\item The graph $\mathcal{G}_r$ is a tree.
\end{enumerate}
\end{lemma}

\begin{proof}
Replacing each blue arc in the graph $\Gamma$, defined in \Cref{sec: intro}, with a band in $S^2$ (as illustrated in \Cref{fig: Sigma}), yields a decomposition of $S^2$ into two subsurfaces, $\Sigma_g$ and $\Sigma_r$, which have the same homotopy type as $\mathcal{G}_g$ and $\mathcal{G}_r$, respectively. Since $\Sigma_g$ is a disc $\iff$ $\Sigma_{r}$ is a disc $\iff$ they are both connected, the claim follows.
\end{proof}

\begin{figure}[]
    \centering
    \includegraphics[width=0.4\textwidth]{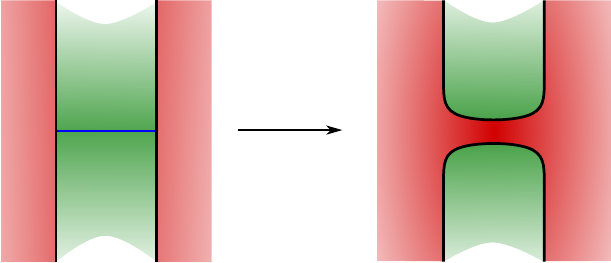}
    \caption{How to obtain $\Sigma_g$ and $\Sigma_r$ from $\Gamma$.}
    \label{fig: Sigma}
\end{figure}

We begin by fixing some notation.
For a link $L = K_0 \sqcup K_1 \sqcup \cdots \sqcup K_n$ in $S^3$, we denote by $S^3_{r_0, \dots, r_n}(L)$ the manifold obtained by performing Dehn surgery on each component of $L$, where $r_i$ is the surgery coefficient on the knot $K_i$.

The proof of \Cref{thm: main theorem} proceeds as follows. Starting from a diagram $D$ of a knot $K$, we define a link $L=K_0\sqcup \cdots \sqcup K_n$ in $S^3$ such that, for some $(r_1, \dots, r_n)\in \overline{\Q}^n$, the manifold $S^3_{\bullet,r_1,\dots, r_n}(L)$ is diffeomorphic to the exterior of $K$, where $\bullet$ indicates that an open tubular neighbourhood of $K_0$ has been removed. We then turn to the link $L$ itself and construct foliations on its surgeries.

The link $L$ is obtained by replacing each twist region in $D$ with a tangle in the following way. We remove all crossings in the twist region  if it contains an even number of them, and all but one otherwise. Then, we encircle the two strands with an unknotted circle, called a \emph{crossing circle}. The knot $K$ can be recovered from $L$ by performing $\tfrac{1}{k}$-surgery on each crossing circle, for appropriate integers $k$. A link $L$ obtained in this way is known in the literature as a fully augmented link \cite{Lackenby, Purcell}. See \Cref{fig: L from K} for an example.

\begin{figure}[h]
    \centering
    \includegraphics[width=0.6\textwidth]{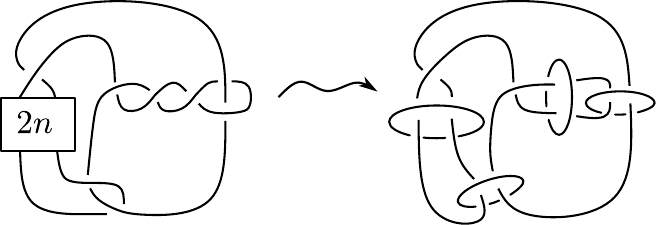}
    \caption{How the link $L$ is obtained from a diagram of $K$.}
    \label{fig: L from K}
\end{figure}

We denote the components of the link $L$ by $K_0, K_1,\dots, K_n$, where $K_1, \dots, K_n$ are the crossing circles. 
Each crossing circle $K_i$ bounds a twice-punctured disc $D_i$ in the exterior of $L$, which intersects the projection sphere transversely in three arcs; see \Cref{fig: horizontal disk}.

\begin{figure}[h]
    \centering
    \includegraphics[width=0.24\textwidth]{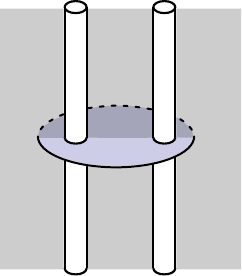}
    \caption{The twice-punctured disc bounded by the crossing circle. The projection sphere, shaded in gray, intersects it in three arcs.}
    \label{fig: horizontal disk}
\end{figure}
By construction, each crossing circle corresponds to a twist region of the diagram $D$. Moreover, if the twist region associated to the crossing circle $K_i$ contains $2k_i$ or $2k_i+1$ left-handed crossings, with $k_i\geq 0$, we label $K_i$ with the coefficient $\frac{1}{k_i}$. Analogously, if the twist region contains $-2k_i$ or $-(2k_i-1)$ right-handed crossings, with $k_i\leq 0$, we label $K_i$ with $\frac{1}{k_i}$. In this way, we have 

$$
S^3_{\bullet}(K)=S^3_{\bullet, \frac{1}{k_1},\dots, \frac{1}{k_n}}(L).
$$
Our goal now is to construct coorientable taut foliations on surgeries on $L$. More precisely, for every non-meridional slope $s$ on $K_0$, we aim to produce a coorientable taut foliation in the exterior of $L$ that intersects the boundary transversely in parallel curves of slope $s, \frac{1}{k_1}, \dots, \frac{1}{k_n}$, respectively. This will imply that $K$ is persistently foliar. The link $L$ enjoys all the properties that we used when studying the Borromean rings and the link $L8n5$ in the previous section, and we will proceed essentially in the same way.

We first define, in \Cref{subsec: the branched surface Bbar}, a branched surface $\overline{B}$ in the exterior of $L$ that has two \emph{meridional cusps} (as in \Cref{def: cusps}) on a torus parallel to $\partial \nu K_0$, and that intersects the boundary components corresponding to $K_1, \dots, K_n$ in a train track realising the rational multislope $(\frac{1}{k_1}, \dots, \frac{1}{k_n})$. Then, in  \Cref{subsec: more than 4 elements} we modify $\overline{B}$ to produce a branched surface $B$ suitable for constructing the desired foliations.

We first observe that we have already proven \Cref{thm: main theorem} in the case when there are at most $4$ vertices in $\mathcal{G}_g\cup \mathcal{G}_r$. Recall that, by definition, the vertices of $\mathcal{G}_g \cup \mathcal{G}_r$ coincide with the complementary regions of the graph $\Gamma$, constructed as in the left side of \Cref{fig: Gamma' intro}. The next lemma shows that if $D$ is a diagram such that the corresponding graph $\Gamma$ has at most $4$ complementary regions, then the fully augmented link $L$ associated to $D$ has at most $2$ crossing circles. 

Recall, from \Cref{sec: intro}, the graph $\Gamma'$, obtained by collapsing the blue arcs in $\Gamma$ as in the right side of \Cref{fig: Gamma' intro}, and notice that the number of vertices in $\Gamma'$ coincides with the number of crossing circles of $L$.

\begin{lemma}
Suppose that $\Gamma$ has at most $4$ complementary regions. Then the graph $\Gamma'$ has at most two vertices.
\end{lemma}
\begin{proof}
The proof is a simple computation with Euler characteristics. The graph $\Gamma'$ defines a cell decomposition of $S^2$ and therefore we have:
$$
\#\text{vertices of }\Gamma'-\#\text{edges of }\Gamma'+\#\text{complementary regions of }\Gamma'=2.
$$
Since $\Gamma'$ is a $4$-valent graph, the number of edges of $\Gamma'$ is twice the number of its vertices and therefore we have
$$
\#\text{complementary regions of }\Gamma'=2+\#\text{vertices of }\Gamma'.
$$
Since the complementary regions of $\Gamma'$ are in bijection with those of $\Gamma$, we deduce that $\Gamma'$ has at most $2$ vertices.
\end{proof}

If $D$ satisfies the hypotheses of \Cref{thm: main theorem}, then $D$ has more than one twist region, and $\Gamma'$ must have exactly two vertices. The only possible $4$-valent graphs in $S^2$ are, up to isotopy, the ones in \Cref{fig: graphs 2 vert}.

\begin{figure}[h]
    \centering
    \includegraphics[width=0.5\textwidth]{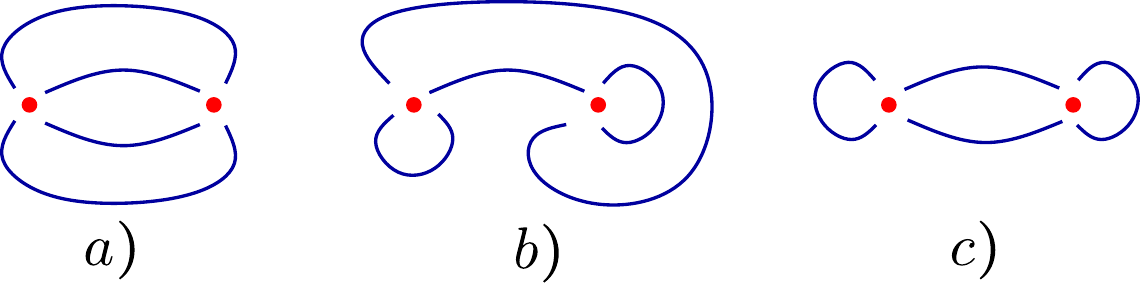}
    \caption{The possible $4$-valent graphs with two vertices in $S^2$ up to isotopy.}
    \label{fig: graphs 2 vert}
\end{figure}
The reader can easily check that if a knot diagram $D$ satisfies the hypotheses of \Cref{thm: main theorem} and yields a graph $\Gamma'$ of type $b)$ or $c)$, then it must be the diagram of a knot that is a connected sum of two non-trivial torus knots. Such knots are persistently foliar by virtue of \cite{DRpersistently}, and so we focus our attention on knots whose diagram has graph $\Gamma'$ of type $a)$, in which case we can suppose that $D$ is as in \Cref{fig: Borr}, where $m\in \{-1,0,1\}$, and $k$ and $n$ are non-zero integers. 
We remark that when $m=0$ the associated fully augmented link is the Borromean rings, when $m=1$ is $L8n5$, and when $m=-1$ is the mirror of $L8n5$.

\begin{figure}[h]
    \centering
    \includegraphics[width=0.6\textwidth]{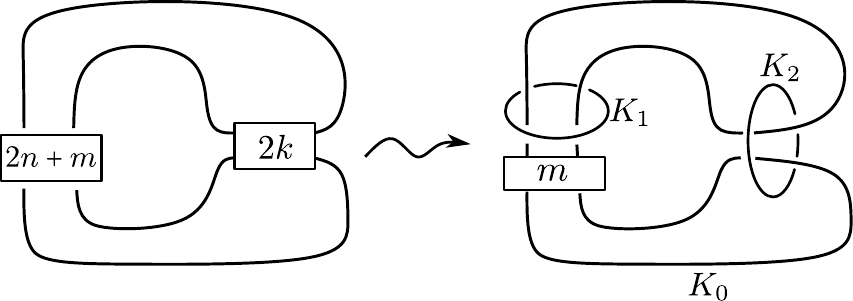}
    \caption{The possible knot diagrams whose associated graph $\Gamma'$ is of type $a)$.}
    \label{fig: Borr}
\end{figure}

\begin{prop}\label{prop: at most four vertices}
Suppose that $K\subset S^3$ has a diagram $D$ satisfying the hypotheses of \Cref{thm: main theorem}, and such that the graph $\mathcal{G}_g\cup \mathcal{G}_r$ has at most $4$ vertices. Then $K$ is persistently foliar. 
\end{prop}

\begin{proof}
We know from the previous discussion that $D$ is as in  \Cref{fig: Borr}, where $m\in\{-1,0,1\}$.

We first consider the case $m=0$, so that the fully augmented link $L$ associated to $D$ is the Borromean rings. For $i=1,2$, denote with $\frac{1}{k_i}$ the surgery coefficient on $K_i$ that allows us to recover $K$ from $L$. Note that our hypothesis on the weights of the graphs $\mathcal{G}_g$ and $\mathcal{G}_r$ implies that at least one among $k_1$ and $k_2$ has absolute value greater than or equal to two. By using the symmetries on the Borromean rings if necessary, we can suppose that this happens for $k_1$; hence $\frac{1}{k_1}\in (-1,1)$. To obtain the claim, it suffices to prove that for any non-meridional slope $s$ on $K_0$ we can find a coorientable taut foliation in the exterior of $L$ intersecting the boundary components corresponding to $K_0, K_1, K_2$ transversely in a foliation of slope $s, \frac{1}{k_1}, \frac{1}{k_2}$, respectively. 

If $\frac{1}{k_2}>0$, this follows from \Cref{thm: sec 3 general main thm borr} and the fact that, by \Cref{prop: sec3 from B to B'}, the boundary train tracks of the branched surface $B$ constructed in \Cref{subsub example: B} realise every rational slope in $(-\infty,1)$ on $K_1$ and every rational slope in $(0,\infty)$  on $K_2$. If $\frac{1}{k_2}<0$, since a knot is persistently foliar if and only if its mirror is, we can consider the mirror of $K$, to which now we can apply the argument of the previous sentence.

Suppose now $m=\pm 1$ and, by taking the mirror of $K$ if necessary, that $\frac{1}{k_2}>0$. The hypotheses of \Cref{thm: main theorem} imply that $2n+m\geq 3$ or $2n+m \leq -3$. In both cases, $K$ is surgery on the link $L8n5$, with surgery coefficient $\frac{1}{k_2}$ on $K_2$ and $\frac{1}{k_1}$ on $K_1$, where
\begin{enumerate}
\item $k_1=-\frac{2n+m-1}{2}$ if $2n+m\geq 3$, in which case $\frac{1}{k_1}< 0$.
\item $k_1=-\frac{2n+m-1}{2}$ if $2n+m\leq -3$, in which case $0<\frac{1}{k_1}\leq \frac{1}{2}$.
\end{enumerate}
In any case, we have that $\frac{1}{k_1}\in (\infty,1)\setminus\{0\}$ and $\frac{1}{k_2}\in (0,\infty)$, and the conclusion follows from \Cref{prop: sec3 from B to B' twist} and \Cref{thm: sec 3 general main thm borr twist}.
\end{proof}

\subsection{The branched surface $\overline{B}$}\label{subsec: the branched surface Bbar}
Having proved \Cref{thm: main theorem} when the graph  $\mathcal{G}_g\cup \mathcal{G}_r$ has at most  $4$ vertices, we now work with the following assumpion.

\begin{namedassumption}
The knot $K\subset S^3$ has a diagram $D$ that satisfies the hypotheses of \Cref{thm: main theorem}, and such that the graph $\mathcal{G}_g\cup \mathcal{G}_r$ has more than $4$ vertices.
\end{namedassumption}

We define a branched surface in the exterior of the fully augmented link $L$ associated to the diagram $D$. The constructions follows very closely the one presented for the Borromean rings and the link $L8n5$ in \Cref{sec: Borromean rings and Bing doubles}.

We denote by $M$ the exterior of $L$, and by $\partial_i M$ the boundary component corresponding to $K_i$, for $i=0, \dots n$. We begin the construction of our branched surface in $M$ by fixing a torus $T$ parallel to $\partial_0 M$. More precisely, we choose a collar neighbourhood $ \nu\partial_0 M=\partial_0 M \times [0,1]$ of $\partial_0 M$, with $\partial_0 M= \partial_0 M\times \{0\}$, and set $T=\partial_0 M\times \{1\}$. When all twist regions in $D$ contain an even number of crossings, the component $K_0$ of $L$ is embedded in the projection sphere. Otherwise, it is convenient to introduce an auxiliary link $L'$, whose exterior $M'$ is related to $M$ by a mutation: we cut $M$ along the discs $D_i$ that are adjacent to some crossing, and reglue with a rotation that removes the crossing. This new link $L'$ has the same crossing circles and discs $D_i$ of $L$, but has some new components $K'_0,\dots, K'_k$, whose union is now embedded in the projection sphere. We denote by $\partial_j M'$ the boundary component of the exterior of $L'$ associated to $K'_j$, for $j=0,\dots,k$, and $\nu \partial_j M'$ a collar of it. 

We define the $2$-complex $\Sigma$ in $M$ as the union of the following surfaces:

\begin{itemize}
\item The torus $T$.
\item The twice-punctured discs bounded in $M\setminus \big{(}{\partial_0 M \times [0,1)}\big{)}$ by the crossing circles $K_1,\dots, K_n$, that we still denote by $D_1,\dots, D_n$ respectively (see \Cref{fig: horizontal disk}).
\item The closures of the connected components of $S^2\setminus (D_1\cup \cdots D_n \cup \nu\partial_0 M)$ when $D$ has only twist regions with an even number of crossings.
Otherwise, we take the closures of the connected components of $S^2\setminus (D_1\cup \cdots D_n \cup \nu\partial_0 M'\cup \cdots\cup \nu\partial_k M')$ in $M'$, and take their image in $M$ after the mutation\footnote{This procedure is detailed at the beginning of \Cref{subsec: the link L8n5}, and showed in the top-left part of \Cref{fig: branched twist} in the case of a right-handed half twist. The case of a left-handed half twist is completely analogous.}.
\end{itemize}

We denote the last family of surfaces by $\mathcal{S}$.  

\begin{lemma}\label{lemma sec4: exterior of Sigma}
The complement of a regular neighbourhood of $\Sigma$ in $M$ is homeomorphic to the disjoint union of two balls, $X_1$ and $X_2$, and a product $T\times [0,1]$.
\end{lemma}
\begin{proof}
When all twist regions in $D$ have an even number of crossings, the proof is exactly as in \Cref{lemma sec3: exterior of Sigma}. Otherwise, we first cut $M$ along the torus $T \subset \Sigma$, obtaining a component homeomorphic to $T \times [0,1]$ and another component that can be identified with $M$ itself.  Cutting this latter component along the surfaces in $\mathcal{S}$ and the discs $D_i$ yields a manifold homeomorphic to the one obtained by cutting $M'$ along its intersection with the projection sphere and the discs $D_i$. By the same argument as in the proof of \Cref{lemma sec3: exterior of Sigma}, this manifold is the disjoint union of two balls.
\end{proof}

We want to obtain a branched surface by smoothing $\Sigma$. To do so, we assign coorientations to the surfaces composing $\Sigma$ and smooth it accordingly. More precisely, we proceed as follows:
\begin{enumerate}
    \item Coorient all the surfaces in $\Sigma$ so that, after smoothing according to these coorientations, the resulting branched surface has two cusps on $T$.
    
    \item Fix an element $S$ in $\mathcal{S}$, which will later be removed from $\Sigma$.
    
    \item For each disc $D_i$, choose a desingularisation of $\Sigma$ in a neighbourhood of $D_i$ so that the train track $\tau_i = B \cap \partial_i M$ realises the slope $\frac{1}{k_i}$, where $B$ denotes the branched surface obtained by smoothing $\Sigma \setminus S$ according to the chosen coorientations.
\end{enumerate}

\subsubsection{Coorientations of the elements in $\mathcal{S}$.}

By construction there is a canonical bijection between $\mathcal{S}$ and the set of connected components of $S^2\setminus \Gamma$, where $\Gamma$ is the graph associated to $D$ in \Cref{sec: intro}. We fix a colouring of these components, obtained as described in \Cref{sec: intro}, and use it to assign coorientations to the elements in $\mathcal{S}$. More precisely, recall that $X_1$ and $X_2$ denote the two balls in the exterior of $\Sigma$, and for consistency, we always draw our figures as if the reader is located inside the ball $X_1$. We require that the coorientation points into $X_1$ (resp. into $X_2$) for sectors corresponding to green (resp. red) complementary regions of $\Gamma$. 
We will also adhere to the convention that the colour green (resp. red) indicates the positive (resp. negative) side of a sector of $\Sigma$. When the diagram $D$ has some twists regions with an odd number of crossings, we first assign coorientations in the exterior of the auxiliary link $L'$ according to this convention, and then introduce the twist. See \Cref{fig: coorientations in S^2} for an example.

\begin{figure}[]
    \centering
    \includegraphics[width=0.5\textwidth]{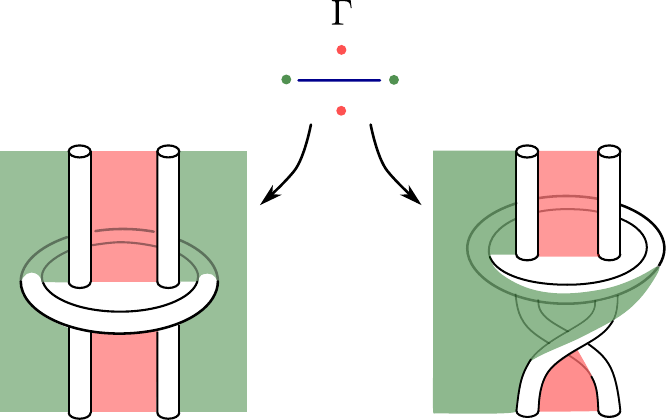}
    \caption{Coorientation of the elements in $\mathcal{S}$ induced by a colouring of the complementary regions of $\Gamma$. To simplify the figure, we have not shown the fourth element in $\mathcal{S}$ in the two-complex on the right.}
    \label{fig: coorientations in S^2}
\end{figure}

The connection between the graphs $\mathcal{G}_r$ and $\mathcal{G}_g$ and the link $L$ follows from the next observation, which ties these graphs -- without the data of the weights on the edges -- with the combinatorial data of the diagram of the link $L$ together with the choice of coorientations of the sectors in $\mathcal{S}$.

We define $\mathcal{S}_+$ (resp. $\mathcal{S}_-$)  as the set of sectors in $\mathcal{S}$ whose positive (resp. negative) side is contained in $X_1$. By our convention on the coorientations, $\mathcal{S}_+$ can be identified with the set of vertices of $\mathcal{G}_g$, and $\mathcal{S}_-$ with that of $\mathcal{G}_r$. The data of the edges can be recovered we defining $\mathcal{G}_+$ and $\mathcal{G}_-$ as follows.

\begin{itemize}
\item The vertices of $\mathcal{G}_{\pm}$ are the elements in $\mathcal{S}_{\pm}$.
\item For every intersection of two elements in $\mathcal{S}_{\pm}$ with the same connected component of $\partial_1 M\sqcup \cdots \sqcup \partial_n M$, we add an edge between the corresponding vertices of $\mathcal{G}_{\pm}$.
\end{itemize}
It is straightforward to show that the graph $\mathcal{G}_+$ is isomorphic to the underlying unweighted graph of $\mathcal{G}_g$, and that  $\mathcal{G}_-$ is isomorphic to that of $\mathcal{G}_r$. In fact, each twist region of $D$ gives rise to a blue arc of $\Gamma$ and to a crossing circle of $L$, and the two elements of $\mathcal{S}$ intersecting it  correspond to the two complementary regions of $\Gamma$ intersecting the blue arc only in its endpoint.

This discussion, together with \Cref{lemma: grafi connessi sse contrattili}, implies that under the standing assumptions of this section we have the following.

\begin{lemma}\label{lemma: grafi delle regioni}
The graphs $\mathcal{G}_+$ and $\mathcal{G}_-$ are both trees. \qed
\end{lemma}

\subsubsection{Coorientations of the discs $D_i$ and the torus $T$.}

Recall that to each component $K_1,\dots, K_n$ of $L$ is associated a non-zero integer number $k_i$ with the property that 
$$
S^3_{\bullet, \frac{1}{k_1},\dots, \frac{1}{k_n}}(L)=S^3_\bullet(K).
$$
The hypotheses of \Cref{thm: main theorem} imply that at least one $k_i$ has either absolute value greater than two, or it is associated to a twist region with a number of crossings that is odd and greater than or equal to $3$. We suppose $i=1$ without loss of generality. We also consider another index $i_0\ne 1$, say $i_0=2$, and fix a sector $S\in\mathcal{S}$ disjoint from $\partial_1 M$ and $\partial_{2} M$. Such a sector exists by the lemma below. We remark that this is the only part of the section where we use the standing assumption on the number of vertices of $\mathcal{G}_r\cup \mathcal{G}_g$, which is why we had to use an \emph{ad hoc} argument in the other cases.

\begin{lemma}\label{lemma: settore che non interseca}
There exists a sector $S\in \mathcal{S}$ that does not intersect $\partial_1 M$ and $\partial_{2}M$.  
\end{lemma}

\begin{proof}
The boundary components $\partial_1 M$ and $\partial_{2} M$ intersect at most two sectors\footnote{Indeed, exactly two under our assumptions.} in $\mathcal{S}$ each. Our standing assumptions imply that the set $\mathcal{S}$ has more than four elements, being it in bijection with set the vertices of the graph $\mathcal{G}_g\cup\mathcal{G}_r$, and so there is at least one sector $S\in \mathcal{S}$ with the desired property.
\end{proof}



We now declare the coorientations on the discs $D_i$ and on the torus $T$, starting with $D_2$, that we coorient arbitrarily. For our purposes\footnote{Recall that a knot is persistently foliar if and only if its mirror is.}, by taking the mirror of $K$ if necessary, we can suppose that the surgery coefficient $\frac{1}{k_2}$ is negative. Moreover, up to reversing all the coorientations of the sectors in $S^2$, we can suppose that the two sectors in $\mathcal{S}$ that intersect $\partial_2 M$ have coorientations pointing in $X_1$.  The intersection $\Sigma \cap \partial_2 M$ is union of $\lambda_2=\partial D_2$ and two arcs, coming from two sectors in $\mathcal{S}$, that close up to two meridional curves on $\partial_2 M$. Here, every sector involved is oriented by using the fixed coorientation and the ambient orientation, and their boundaries get the induced orientations. By our convention on the coorientations of the sectors in $\mathcal{S}$, we have that the two meridional curves of intersection on $\partial_2 M$ have opposite orientations. We fix a curve component $c_2$ of $D_2\cap T$ in the following way:
\begin{enumerate}
\item If the disc $D_2$ is adjacent to a crossing, we choose $c_2$ arbitrarily. 
\item Otherwise, we choose $c_2$ to be the circle component of $T\cap D_2$ that intersects the sector of $\mathcal{S}$ defining the meridional curve meeting $D_2$ positively; see \Cref{fig: cusp 2}.
\end{enumerate}

\begin{figure}[]
    \centering
    \includegraphics[width=0.23\textwidth]{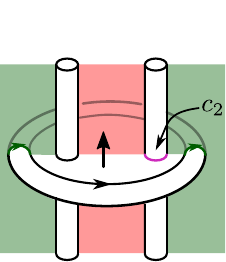}
    \caption{The figure shows how to choose the curve $c_2\in D_2\cap T$.}
    \label{fig: cusp 2}
\end{figure}

Fix a curve $c_1\in D_1\cap T$ arbitrarily, and consider the two annuli obtained by cutting $T$ along $c_1$ and $c_2$. We assign coorientations to $D_1$ and the annuli in the unique way that produces two cusps on $T$, as defined in \Cref{def: cusps}, along $c_1$ and $c_2$. More precisely, the coorientation of $D_2$ and the constraint of creating a cusp on $c_2$ force the coorientation on the annuli. Next, the coorientation on the annuli and the constraint of creating a cusp along $c_1$ impose a unique coorientation on $D_1$.

Finally, we coorient the discs $D_i$, for $i\ne 1,2$.
Let $\partial_{i_1} M,\dots, \partial_{i_m} M$ be the boundary components of $M$ that intersect the sector $S$, and denote by $\sigma_{i_j}$ the oriented arc $S\cap \partial_{i_j} M$. This arc intersects $D_{i_j}$ in two points, which coincide if there is no crossing adjacent to the disc, and are antipodal otherwise. We coorient $D_{i_j}$ so that the sign of intersection between $D_{i_j}$ and $\sigma_{i_j}$ is opposite to the sign of $\frac{1}{k_{i_j}}$. \Cref{fig: coori discs} shows an example.
We fix the coorientations of the discs that do not intersect $S$ arbitrarily.

\begin{figure}[]
    \centering
    \includegraphics[width=0.6\textwidth]{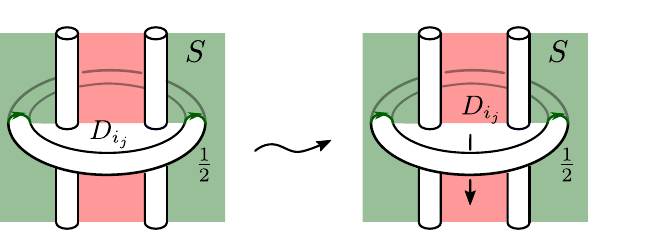}
    \caption{Since the surgery coefficient $\frac{1}{2}$ is positive, we coorient the disc $D_{i_j}$ so that the intersection with $\sigma_{i_j}$ is negative. If the surgery coefficient was negative, or $S$ was the other sector intersecting $\partial_{i_j}M$, we would have assigned the opposite coorientation on the disc.}
    \label{fig: coori discs}
\end{figure}
\subsubsection{Desingularisation}
We have cooriented every surface composing $\Sigma$, and we now discuss the possible desingularisations of the intersections between the discs $D_i$ and the sectors in $\mathcal{S}$.

For each $i=1,\dots, n$, we orient the component $K_i$ of $L$ as the boundary of $D_i$, and we consider the canonical meridian and longitude basis on $\partial_i M$.  Each disc intersects the sectors in $\mathcal{S}$ along three arcs, and as already done in \Cref{subsubsec: desingularisation} and depicted in \Cref{fig: possible slidings}, for each arc we have two possible ways of slide the sectors along the discs, thereby creating two arcs of double points. Therefore, near each disc $D_i$ we have eight possible choices of desingularisations, and each choice produces five sectors in $D_i$. 
When choosing how to desingularise, we have to pay attention in avoiding half sink disc sectors in $D_i$, and we want that the boundary train track $\tau_i=B\cap \partial_i M$, where $B$ is obtained by smoothing $\Sigma\setminus\{S\}$, realises the slope $\frac{1}{k_i}$, for $i=1,\dots,n $.
\newline

We first suppose that the disc $D_i$ is not adjacent to a crossing. Among the five sectors that are created in $D_i$, there are two of them that intersect only one component of $D_i\cap T$, and these are created when sliding on $D_i$ the two sectors in $\mathcal{S}$ the intersect $\partial_i M$. We denote by $\Delta_i^+$ the one created by sliding the sector in $\mathcal{S}$ whose intersection with $\partial_i M$ is a positive meridian of $K_i$, and by $\Delta_i^-$ the other. Each of these sectors has in its boundary two arcs of branch locus coming from the intersection between $D_i$ and the sectors in $\mathcal{S}$, and depending on the choice of desingularisation, the cusp directions along them either both point in the sector, or both point out of it. 
We say that a desingularisation of $\Sigma$ in a neighbourhood of a disc $D_i$ is of:
\begin{itemize}
\item type $\Delta^{\pm}-in$ if these cusp directions point in $\Delta_i^+$ and $\Delta_i^-$.
\item type $\Delta^{\pm}-out$ if these cusp direction point out of $\Delta_i^+$ and $\Delta_i^-$.
\item type $\Delta^+-out, \Delta^--in$ if these cusp directions point out of $\Delta_i^+$ and in to $\Delta_i^-$.
\end{itemize}
Clearly there is also a desingularisation of type  $\Delta^+-in, \Delta^--out$, but we will not need it.
\begin{rem}
The type of desingularisation only depends on how we slide the sectors that intersect $\partial_i M$, and we still have two choices on how to desingularise the arc of intersection in $D_i$ that is disjoint from $\partial_i M$. We will not need this additional degree of freedom, and any choice of desingularisation of that arc will work.  
\end{rem}

\begin{lemma}\label{lemma: desing no crossing}
Suppose that $D_i$ is not adjacent to a crossing, choose a desingularisation of $\Sigma$ in a neighbourhood of it, and smooth $\Sigma$ near $D_i$ according to the coorientations. Then:
\begin{itemize}
\item If the desingularisation is of type $\Delta^{\pm}-in$, the boundary train track on $\partial_i M$ realises all rational slopes in $(-\infty, \infty)$.
\item If the desingularisation is of type $\Delta^{\pm}-out$, the boundary train track on $\partial_i M$ realises all rational slopes in $(-1, 1)$.
\item If the desingularisation is of type $\Delta^+-out, \Delta^--in$, the boundary train track on $\partial_i M$ realises all rational slopes in $(-\infty, 1)$.
\end{itemize}
\end{lemma}
\begin{proof}
    Before the desingularisation, the intersection $\Sigma \cap \partial_i M$ consists of the longitude $\lambda_i = \partial D_i$ and two oppositely oriented meridians. With the notation introduced in \Cref{subsec: train tracks and rational slopes}, after smoothing, this $1$-complex becomes a train track of type $t_1 + t_2$, where $t_1$ is either $a)$ or $b)$, and $t_2$ is either $c)$ or $d)$. 
    More precisely, two arcs of branch locus in the boundary of $\Delta_i^+$ come from the intersection between $D_i$ and the sectors in $\mathcal{S}$, and $t_1$ is $a)$ if the cusp directions along them  point out of $\Delta_i^+$; it is $b)$ otherwise. Analogously, $t_2$ is $c)$ when the cusp directions in the boundary of $\Delta_i^-$ coming from the intersection between $D_i$ and the sectors in $\mathcal{S}$, point out of it, and $d)$ otherwise. We conclude by using \Cref{lemma: unlinked}. 
\end{proof}

\begin{figure}[]
    \centering
    \includegraphics[width=0.7\textwidth]{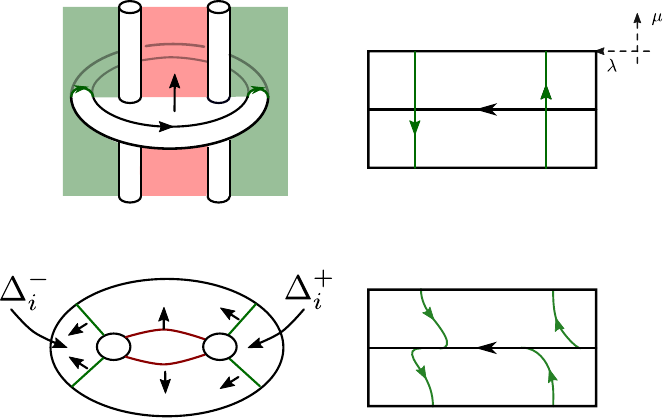}
    \caption{Top: the complex $\Sigma$, and its intersection with $\partial_i M$, before the desingularisation,. Bottom: the sectors in the disc $D_i$ after a desingularisation of type $\Delta^+-out, \Delta^--in$, and the boundary train track on $\partial_i M$.}
    \label{fig: desing}
\end{figure}
\Cref{fig: desing} shows an example of desingularisation of type $\Delta^+-out, \Delta^--in$, and the induced train track on the boundary.

In conclusion, if $D_i$ is not adjacent to a crossing, we choose the following desingularisation:
\begin{itemize}
    \item type $\Delta^{\pm}-out$, if $i=1$
    \item type $\Delta^+-out, \Delta^--in$, if $i=2$.
\item type $\Delta^{\pm}-in$, otherwise.
\end{itemize}

If $D_i$ is adjacent to a crossing of $L$, we perturb $\Sigma$ near $D_i$ as described in \Cref{fig: desing_twist right} for a right handed-crossing, and in \Cref{fig: desing_twist left} for a left-handed crossing. The pictures also also show the cusp directions created on the disc, and the boundary train tracks obtained after smoothing.

\begin{figure}[]
    \centering
    \includegraphics[width=0.8\textwidth]{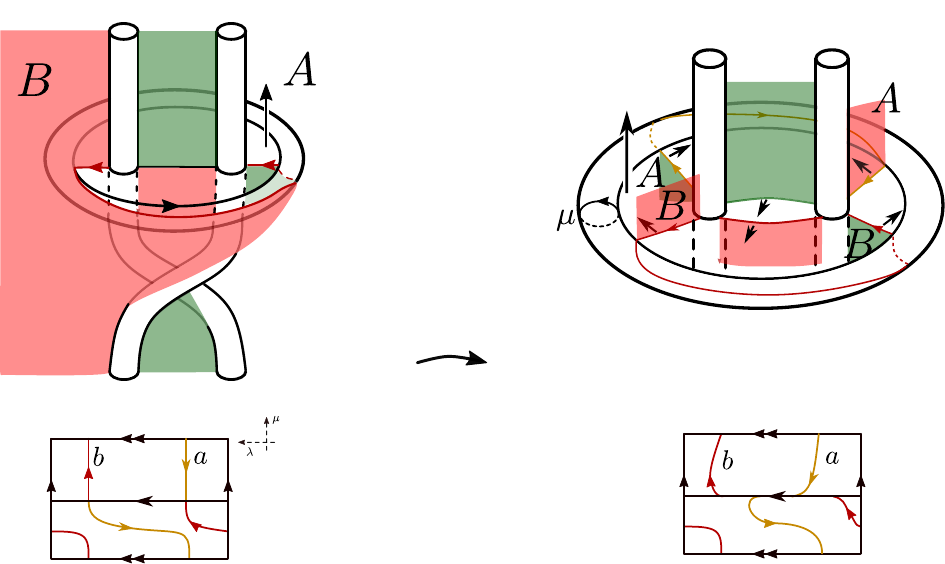}
    \caption{How to desingularise $\Sigma$ in a neighbourhood of a disc $D_i$ adjacent to a right-handed crossing. We denote by $b$ and $a$ the intersection of $\partial_i M$ with the sector $B$ and $A$, respectively. On the left, the sector $A$ is not drawn to make the picture simpler to understand. If the coorientation of the sectors in $\mathcal{S}$, or the one of $D_i$, is opposite to the one of the figure, we slide the sectors in $\mathcal{S}$ exactly in the same way, obtaining the same train track on $\partial_i M$, but all the cusp directions along their intersections with $D_i$ are reversed.}
    \label{fig: desing_twist right}
\end{figure}

\begin{figure}[]
    \centering
    \includegraphics[width=0.8\textwidth]{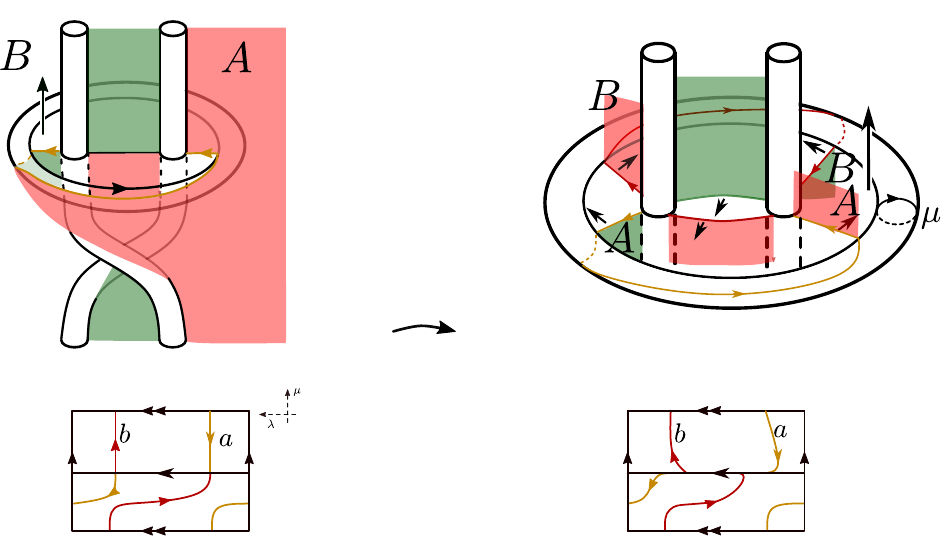}
    \caption{How to desingularise $\Sigma$ in a neighbourhood of a disc $D_i$ adjacent to a left-handed crossing. On the left, the sector $B$ is not drawn. If the coorientation of the sectors in $\mathcal{S}$, or the one of $D_i$, is opposite to the one of the figure, we slide the sectors in $\mathcal{S}$ exactly in the same way, obtaining the same train track on $\partial_i M$, but all the cusp directions along their intersections with $D_i$ are reversed.}
    \label{fig: desing_twist left}
\end{figure}

\begin{prop}\label{prop: from sigma to B}
There exists a smoothing of $\Sigma$ to a coorientable branched surface $\overline{B}$ satisfying the following properties:
\begin{itemize}
\item $\overline{B}$ has two cusps on $T$.
\item $\overline{B}$  contains no sink discs or half sink discs.
\item $\partial \overline{B} \cap \partial_i M$ realises the slope $\frac{1}{k_i}$ for each $i=1,\dots, n$;

\end{itemize}
\end{prop}
\begin{proof}
Let $\overline{B}$ be the branched surface obtained by smoothing $\Sigma$, perturbed as described above near the discs $D_i$, according to our choices of coorientations.
The proof proceeds exactly as the proof of \Cref{prop: sec 3 from sigma to B}, and it is enough to check that no sectors in the discs $D_i$ are half sink discs, and that the boundary train track realise the desired slopes. 

First of all, suppose that $D_i$ is adjacent to a crossing. In this case, by smoothing as in \Cref{fig: desing_twist right} or \Cref{fig: desing_twist left}, we see that every sector in $D_i$ has some cusp direction pointing out of it, and so there are no half sink discs.
If $D_i$ is adjacent to right-handed crossing, then by definition of the link $L$, the crossing circle $K_i$ is associated to a twist region of the diagram $D$ with an odd number $m\geq 3$ of right-handed crossings; hence $\frac{1}{k_i}$ is negative. The train track that we obtain on $\partial_i M$ is of type $a)+d)$, and hence realises all rational slopes in $(-\infty, 1)$ and, in particular, $\frac{1}{k_i}$. The case when $D_i$ is adjacent to a left-handed crossing is analogous.

Suppose now that $D_i$ is not adjacent to a crossing.
Recall that the cusp directions along $c_1$ and $c_2$ point into $D_1$ and $D_2$ respectively, and that on the other circles of intersection between $T$ and the discs $D_i$ the cusp directions point into $T$, see \Cref{fig: cusp}.
In particular, if $i\ne 1,2$, all cusp directions along $D_i\cap T$ point out of $T$, and since each sector in $D_i$ intersects at least one component of $D_i\cap T$, we have no half sink discs in it, and the boundary train track on $\partial_i M$ realises $\frac{1}{k_i}$ by \Cref{lemma: desing no crossing}.

If $i=1,2$, then the only possible half sink disc in $D_i$ is the sector intersecting $D_i\cap T$ only in $c_i$, \ie, the sector $\Delta_i^+$ or $\Delta_i^-$. When $i=1$, we have fixed a desingularisation of type $\Delta^{\pm}-out$, and so both these sectors have cusp directions pointing out of them. Moreover, \Cref{lemma: desing no crossing} implies that $\frac{1}{k_1}$ is realised by the boundary train track, since by our assumptions the slope $\frac{1}{k_1}\in (-\frac{1}{2}, \frac{1}{2})$ when $D_1$ is not adjacent to a crossing.  If $i=2$, by our choice of $c_2$, we know that the only possible half sink disc is the sector $\Delta_2^+$, and, since we have chosen a desingularisation of type $\Delta^+-out, \Delta^--in$, we know that there are cusp directions pointing out of it.
The boundary train track of $\overline{B}$ on $\partial_2 M$ realises all rational slopes in $(-\infty,1)$ by \Cref{lemma: desing no crossing} and, since $\frac{1}{k_2}$ is negative, we conclude.
\end{proof}

\subsection{The branched surface $B$}\label{subsec: more than 4 elements}
We are ready to define the branched surface that we will use to define the desired taut foliations. We define it as
$$
B=\overline{B}\setminus S,
$$ 
where $S$ is the sector in $\mathcal{S}$ given by \Cref{lemma: settore che non interseca}. Notice that since every cusp direction along the boundary of $S$ points outside of it, $B$ is a well-defined branched surface, \ie, $B$ satisfies the local models of \Cref{branched surface}. We first observe the following.

\begin{lemma}\label{lemma: conn implica conn}
Let $\mathcal{G}_1$ and $\mathcal{G}_2$ be two graphs, and let $\mathcal{G}'_1$ be the graph obtained from $\mathcal{G}_1$ by removing a vertex $v$ together with the interiors of all edges incident to it.
Denote by $v_1,\dots, v_m$ be the vertices of $\mathcal{G}'_1$ that were adjacent to $v$ in $\mathcal{G}_1$, and let $\mathcal{G}$ be the graph obtained from $\mathcal{G}'_1\sqcup \mathcal{G}_2$ by adding, for $i=1,\dots, m$, an edge connecting $v_i$ to some vertex $w_i$ of $\mathcal{G}_2$.
If $\mathcal{G}_1$ and $\mathcal{G}_2$ are connected, then $\mathcal{G}$ is connected.
\end{lemma}
\begin{proof}

Fix a vertex $z$ of $\mathcal{G}$. Since $\mathcal{G}_2$ is connected, it suffices to show that there exists a path in $\mathcal{G}$  from $z$ to one of the vertices $w_i$, for some $i=1,\dots, m$. 
If $z$ is a vertex of $\mathcal{G}_2$, the existence of such a path is clear. 
Otherwise, by the connectedness of $\mathcal{G}_1$ there is a path in $\mathcal{G}_1$ from $z$ to $v$, that decomposes into shorter paths after removing $v$ and the interiors of all edges incident to it. By construction, the endpoint of one of these paths is $z$, and the other\footnote{This path may also be degenerate, consisting only of the vertex $z$.} is one of the vertices $v_i$. Hence, $z$ is connected to some $v_i$, and therefore to the corresponding vertex $w_i$. 
\end{proof}

Recall the definition of the graphs $\mathcal{G}_+$ and $\mathcal{G}_-$ given before the statement of \Cref{lemma: grafi delle regioni}. For concreteness, we suppose that $S$ corresponds to a vertex in $\mathcal{G}_-$, but the proof of the next proposition works in exactly the same way if $S$ corresponds to a vertex in $\mathcal{G}_+$.

\begin{prop}\label{prop: from B to B'} 
The branched surface $B$ satisfies the properties of \Cref{prop: from sigma to B}. Moreover, the complement $M\setminus {\rm int}(N_{B})$ decomposes as a sutured manifold $X\cup Y$, where:

\begin{itemize}
\item $X$ is a product sutured ball. 
\item $(Y,\partial_v Y)=(\nu \partial_0 M, \mathcal{A}_1\cup \mathcal{A}_2)$, where $\mathcal{A}_1\cup \mathcal{A}_2\subset  \partial_0 M \times \{1\}$ are two annuli whose slope is a meridian of $K_0$.
\end{itemize}
\end{prop}

\begin{proof}  
The two cusps on $T$ are a consequence of the smoothing, exactly as it was for $\overline{B}$, and removing the sector $S$ has the effect of gluing some sectors of $\overline{B}$ along arcs in their boundary. Since $\overline{B}$ has no half sink discs, the same holds for $B$.

Removing $S$ has the effect of removing the branch $\sigma_{i_j}$ from the train track $\partial \overline{B}\cap \partial_{i_j} M$, for $i_j=i_1,\dots i_m$, and of leaving the other boundary train tracks unmodified. From the proof of \Cref{prop: from sigma to B}, we have that boundary train track of $\overline{B}$ on $\partial_{i_j} M$ realises every rational slopes in $(-x,y)$, for some $x,y\in \{1, \infty\}$ depending on the boundary component. Our choice of coorientation on the disc $D_{i_j}$ was made so that the branch $\sigma_{i_j}$ is oriented as the meridian $\mu_{i_j}$ when $\frac{1}{k_{i_j}}$ is negative, and in the opposite direction when when $\frac{1}{k_{i_j}}$ is positive. Hence, the train track obtained by removing $\sigma_{i_j}$ realises every rational slopes in $(-x,0)$ when $\frac{1}{k_{i_j}}$ is negative, and $(0,y)$ when $\frac{1}{k_{i_j}}$ is positive. Since we know from \Cref{prop: from sigma to B} that, for each $i=1,\dots, n$, the slope $\frac{1}{k_i}$ is realised by the boundary train track of $\overline{B}$, we deduce that the same holds for $B$.

We are left to prove the statement regarding the structure of $M\setminus {\rm int}(N_{B})$ as a sutured manifold.  When we remove $S$ from $\overline{B}$, the two balls $X_1$ and $X_2$ in the exterior of $\Sigma$ are glued along a subdisc of their boundary. Therefore we have that $M\setminus {\rm int}(N_{B})= X\cup Y$, where $X$ is topologically a ball, and $(Y, \partial_v Y)$ is homeomorphic to $(\nu \partial_0 M, \mathcal{A}_1\cup \mathcal{A}_2)$, with $\mathcal{A}_1\cup \mathcal{A}_2\subset \partial_0 M \times \{1\}$ two annuli produced by the two meridional cusps on $T$.
Regarding the sutured manifold $(X,\gamma_X)$, we denote 
$$
R(\gamma_{X})=\partial_h N_{B}\cap X=R_{+}\sqcup R_{-},$$
where $R_+$ is the subsurface of $\partial_h N_{B}\cap X$ along which the coorientation points into $X$, and $R_-$ is the one along which the coorientation points out of $X$.
We want to show that these two sets are connected, and use \Cref{rem: product ball} to conclude. We show it for $R_+$, the discussion for $R_-$ being analogous. If we denote by $\mathcal{R}_+$ the set of sectors of $B$ whose positive side is contained in $X$, then $R_+$ is homeomorphic to the union of all the positive sides of the sectors in $\mathcal{R}_+$. Observe that $\mathcal{S}'=\mathcal{S}\setminus\{S\}$ is contained in $\mathcal{R}_+$.
Our choice of coorientations implies that the positive side of every sector $F\in \mathcal{R}_+$ can be connected, by a path in $R_+$, to that of a sector in $\mathcal{S}'$. In fact:
\begin{itemize}
\item  This is immediate if $F\in \mathcal{S}'$.
\item If $F\subset T$, then either $F$ intersects both sides of some sector in $\mathcal{S}'$, or it intersects exactly two sectors in $\mathcal{S}'$, as showed in \Cref{fig: curve sector torus}. In both cases, we can find the desired path from the positive side of $F$ to that of a sector in $\mathcal{S}'$.
\item If $F$ is contained in some disc $D_i$, then, as a consequence of our coorientations of the sectors in $\mathcal{S}$, its positive side can be connected, by a path lying on the positive side of $D_i$, to the positive side of at least one of the sectors in $\mathcal{S}'$ that intersect $D_i$. 
\end{itemize}
This observation implies that, in order to prove that $R_+$ is connected, it is sufficient to prove that the positive sides of any two elements in $\mathcal{S}'$ can be connected by a path in $R_+$. 
In order words, it is enough to prove that the graph $G$, defined as follows, is connected. The vertices of $G$ are elements in $\mathcal{S}'$, and two vertices are connected by an edge if the positive sides of the corresponding sectors in $\mathcal{S}'$ can be joined by a path in $R_+$.

Call $v$ the vertex in $\mathcal{G}_-$ corresponding to $S$, and let $v_1,\dots v_m$ be the vertices in $\mathcal{G}_-$ adjacent to it. Denote with $\mathcal{G}'_-$ the graph obtained by removing the vertex $v$ and the interior of all the edges incident to it from $\mathcal{G}_-$, and observe that the vertices of $G$, \ie, the sectors in $\mathcal{S}'$, are in bijection with the vertices of $\mathcal{G}'_-\sqcup \mathcal{G}_+$. Moreover:

\begin{enumerate}
    \item If two sectors in $\mathcal{S}'$ intersect the same boundary component $\partial_i M$, for $i\ne 0$, then their positive sides can be connected by a path through the positive side of the disc $D_i$, as showed for example in \Cref{fig: S_+ is connected}.
    \item If a sector in $\mathcal{S}'$ intersects one of the boundary components $\partial_i M$ meeting $S$, \ie, if it corresponds to a vertex $v_j$ of $\mathcal{G}_-$ adjacent to $v$, then its positive side can be connected to that of one of the two sectors in $\mathcal{S}'$ that intersect the interior of $D_i$, as in \Cref{fig: S is connected}\footnote{The pictures in \Cref{fig: S_+ is connected} and \Cref{fig: S is connected} show the case when there is no crossing adjacent to $D_i$, but the same properties hold when there is a crossing.}. We denote by $w_j$ the vertex associated to such sector in $\mathcal{G}_+$.
\end{enumerate}

Consider now  the graph $\mathcal{G}$ obtained by $\mathcal{G}'_-\sqcup \mathcal{G}_+$ by adding, for $j=1,\dots, m$, an edge connecting $v_j$ to $w_j$. The vertices of this graph are identified with the vertices of $G$, and the two properties above imply that if two vertices of are connected by a path in $\mathcal{G}$, then they are connected by a path in $G$. The graph $\mathcal{G}$ is connected by virtue of \Cref{lemma: grafi delle regioni} and \Cref{lemma: conn implica conn}.  Therefore, we conclude that $G$ is connected, and so is $R_+$. In the same way one proves that $R_-$ is connected, and hence that $X$ is a product sutured ball. 
\end{proof}
We now study a little bit further the branched surface $B$ and our goal is to prove that $B$ is laminar. We start with the following lemma.

\begin{lemma}\label{lemma: no closed surfaces and annuli}
Any surface carried by $B$ intersects $\partial M$. In particular $B$ does not carry any closed surface. Moreover $B$ does not carry annuli properly embedded in $M$.
\end{lemma}
\begin{proof}
Suppose that $F$ is a properly embedded surface in $M$ that is carried by $B$, and observe that transversality forces the following property: if $F$ passes through a sector, \ie, intersects some fiber over its interior, then it intersects \emph{all} the fibers over the sector. Similarly to the case of train tracks, we can use $F$ to define \emph{weight system} on $B$: we assign to each sector in $B$ a weight given by the number of intersections between $F$ and any interval fiber  over that sector. By definition, these weights satisfy the condition represented in \Cref{fig: branched surface with weights}.
Recall that the sectors of $B$ are of three types:
\begin{enumerate}[$a)$]
\item sectors contained in the discs $D_1, \dots, D_n$;
\item sectors in $\mathcal{S}\setminus S$;
\item sectors contained in the torus $T$.
\end{enumerate}
\begin{figure}[]
    \centering
    \includegraphics[width=0.23\textwidth]{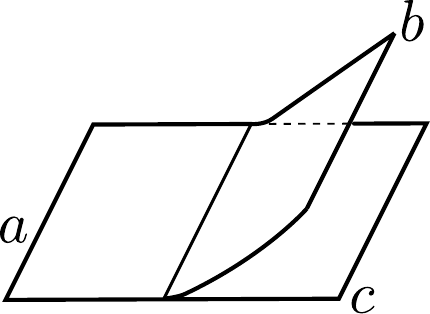}
    \caption{The weights in figure must satisfy $a=b+c$.}
    \label{fig: branched surface with weights}
\end{figure}
The sectors in the first two families intersect $\partial M$. So if $F$ is closed surface carried by $B$, then $F$ can intersect only the fibers over the sectors in $T$. This is not possible, since the two cusps along $T$ force $F$ to pass through some sectors of type $a)$, and $F$ must intersect $\partial M$, contradicting our hypothesis. 

Suppose now that $A$ is an annulus carried by $B$. 
Since all sectors of type $b)$ intersect the torus $T$, if $A$ passes through one of them, then by studying the weight system induced by $A$ near the intersection of $T$ and this sector, we have one of the two possible situations described in \Cref{fig: weight system near T}, with $a\ne 0$. In both cases, the compatibility conditions on the weights would force $a$ to be infinite, which is absurd. Therefore, $A$ does not pass through the sectors of type $b)$.
This implies that $A$ must (after a small isotopy) contain one of the two-punctured discs $D_i$. A computation with the Euler characteristic implies that $A$ is then obtained by gluing to $D_i$ an annulus and a disc along the two curves of intersection $T\cap D_i$. Neither of these curves, however, bounds a disc in $M$, since they are isotopic to meridians of $K_0$. This implies that there are no properly embedded annuli carried by $B$, and concludes the proof.
\end{proof}
\begin{figure}[]
    \centering
    \includegraphics[width=0.4\textwidth]{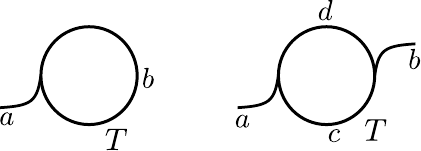}
    \caption{The two possible train tracks obtained by considering a section of $B$ in a neighbourhood of a sector in $T$ and the weight systems induced by $A$.}
    \label{fig: weight system near T}
\end{figure}

\begin{prop}\label{prop: B' is laminar}
The branched surface $B$ is laminar in $M$.
\end{prop}

\begin{proof}
We have to check that all the properties listed in \Cref{def: laminar} are satisfied. Recall from the previous proposition that $M\setminus {\rm int}(N_{B})= X\cup Y$ where $X$ is a product sutured ball and $Y$ are two meridional annular sutures in $\partial_0 M \times \{1\}$.
\begin{enumerate}[start=0, align=left, leftmargin=*]
\item \emph{$B$ has no trivial bubbles:} the only $\matD^2\times [0,1]$ region in the exterior of $B$ is $X$ and the fiber-collapsing map $\pi:{\rm int}(N_{B})\rightarrow B$ is not injective on $\matD^2 \times \{0,1\}$, since the union of the interiors of $R_+$ and $R_-$ intersects both the positive and negative sides of each sector in $\mathcal{S}$; 
\item \emph{$\partial_h N_{B}$ is incompressible and $\partial$-incompressible in $M\setminus {\rm int}(N_{B})$, and no component of $\partial_h N_{B}$ is a sphere or a properly embedded disc in $M$:} the horizontal boundary of $N_{B}$ is given by the two discs $\matD^2\times\{0\}$ and $\matD^2\times \{1\}$, in the boundary of $X$, and by two annuli in the boundary of $Y$, whose cores are essential simple closed curves in $\partial_0 M \times \{1\}$. This implies that $\partial_h N_B$  is incompressible in $M\setminus {\rm int}(N_{B})$. 

The horizontal boundary is $\partial$-incompressible since $X$ is a product sutured ball and the horizontal boundary of $N_{B}$ in $Y$ does not intersect $\partial M$.
Finally, no component of $\partial_h N_{B}$ is a sphere, and the two disc components in it are not properly embedded in $M$.
\item \emph{There is no monogon in $M\setminus {\rm int}(N_B)$:} this is a consequence of $B$ being cooriented;
\item \emph{$M\setminus {\rm int}(N_{B})$ is irreducible and $\partial M\setminus {\rm int}(N_{B})$ is incompressible in $M\setminus {\rm int}(N_B)$:} this is a direct consequence of $M\setminus {\rm int}(N_B)=X\cup Y$, where $X$ is a ball and $Y$ is the product of a torus and an interval;
\item \emph{$B$ contains no Reeb branched surfaces:} if this happens then $B$ carries a torus or a properly embedded annulus, which contradicts \Cref{lemma: no closed surfaces and annuli};
\item \emph{$B$ has no sink discs or half sink discs:} this follows from \Cref{prop: from B to B'}.
\end{enumerate}
This concludes the proof.
\end{proof}

\begin{lemma}\label{lemma: B is taut}
For every sector $F$ of $B$ and for every point $p\in F$ there exists an oriented simple closed curve containing $p$ that is positively transverse to $B$.
\end{lemma}

\begin{proof}
Let $p$ be a point in a sector $F$  of $B$, and consider the fiber $I_p$ of $N_B$ over $p$. The endpoints of $I_p$ lie (for a suitable chioice of $N_B$) in the horizontal boundary of $M\setminus {\rm int}(N_B)$. Recall that $M\setminus {\rm int}(N_B)$ is union of a product sutured ball $\matD^2 \times [0,1]$ and $\partial_0 M\times [0,1]$. If $S$ is not contained in the torus $T$, then one endpoint of $I_p$ lies in $\matD^2 \times \{0\}$ and the other in $\matD^2 \times \{1\}$, so we can connect them with a properly embedded simple arc in the exterior of $B$, and find a simple closed curve that intersect $B$ only in $F$. If $F$ is contained in the torus $T$, then one endpoint of $I_p$ is contained in $\matD^2 \times \{0,1\}$, and the other in $\partial_0 M\times \{1\}$. Orient $I_p=[0,1]_p$ according to the coorientation of $B$ and, without loss of generality suppose that $1\in I_p$ is contained in $\matD^2 \times \{0,1\}$. Recall that the coorientations on $T$ are assigned so that the two annuli obtained by cutting $T$ along the curves $c_1$ and $c_2$ -- the curves where the cusps were created -- are cooriented in opposite ways. This implies that there exists another sector $F'$ contained in $T$ with coorientation opposite to $F$, \ie, such that if we consider an oriented fiber $I_{p'}=[0,1]_{p'}$ over $p' \in F'$, then $0\in I_{p'}$ is contained in $\matD^2 \times \{0,1\}$. We can then find a properly embedded simple arc $\alpha \subset \matD^2\times [0,1]$ and a properly embedded simple arc $\beta\subset \partial_0 M\times [0,1]$ so that $I_p\cup \alpha\cup I_{p'}\cup \beta$ defines the desired curve. See \Cref{fig: curve sector torus} for a schematic picture.
\end{proof}
\begin{figure}[]
    \centering
    \includegraphics[width=0.35\textwidth]{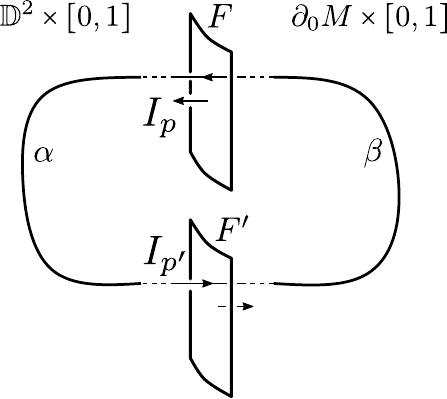}
    \caption{Schematic picture of transversal through a sector in $T$.}
    \label{fig: curve sector torus}
\end{figure}
The following result represents the core of the proof of the main theorem. We remark that the standing assumption that the graph $\mathcal{G}_g\cup \mathcal{G}_r$ has more than $4$ vertices plays no role in the proof below. The arguments presented here can be carried over unchanged to prove \Cref{thm: sec 3 general main thm borr} and \Cref{thm: sec 3 general main thm borr twist}.
\begin{teo}\label{thm: general main thm}
Let $s$ be any non-meridional slope on $\partial_0 M$ and let $(r_1, \dots, r_n)$ be a rational multislope on $\partial_1 M\sqcup \cdots \sqcup \partial_n M$. Suppose that the multislope $(r_1,\dots, r_n)$ is realised by the boundary train track $\partial B$, and that $r_i$ is not the longitude defined by the disc $D_i$, for $i=1,\dots, n$. Then, there exists a coorientable taut foliation in $M$ intersecting the boundary transversely in a linear foliation of multislope $(s,r_1,\dots, r_n)$. In particular, when $s$ is rational, the manifold $M(s, r_1, \dots, r_n)$ obtained by filling $M$ along the multislope $(s, r_1, \dots, r_n)$ supports a coorientable taut foliation.
\end{teo}
\begin{proof} 
As in the proof of \Cref{boundary train tracks} by Li (see \cite[Theorem 2.5]{L2}), we first split $B$ in a neighbourhood of $\partial_1 M\sqcup \cdots \sqcup \partial_n M$, obtaining a branched surface $B'$ that intersects each of these boundary components in simple closed curves of slopes $r_1, \dots, r_n$, respectively. We then consider the manifold $M(\cdot,r_1,\dots, r_n)$, obtained by filling the boundary component $\partial_i M$ along the slope $r_i$, for $i=1,\dots, n$. Let $B(r_1, \dots, r_n)$ denote the branched surface in $M(\cdot,r_1,\dots, r_n)$ obtained by capping $B'$ with meridional discs of the solid tori. 

Li proves that $B(r_1, \dots, r_n)$ is laminar, and therefore, by \cite[Theorem 1]{L}, fully carries an essential lamination $\mathcal{L}$ and the ambient manifold is irreducible. We can also suppose that the horizontal boundary of a regular neighbourhood of $B(r_1, \dots, r_n)$ is contained in $\mathcal{L}$. Our goal is now to apply \Cref{cusps implies persistently foliar}. To this end, we show that:
\begin{itemize}[leftmargin=*]
\item \emph{$B(r_1, \dots, r_n)$ is taut:} we must check that the horizontal boundary of the regular neighbourhood of $B(r_1, \dots, r_n)$ is Thurston norm minimising, and that every sector of $B(r_1, \dots, r_n)$ intersects an oriented simple closed curve positively transverse to the branched surface. 

We begin with the branched surface $B'$. Recall that $M\setminus {\rm int}(N_B) = X \cup Y$, where $X$ is a product sutured ball and
$(Y, \partial_v Y) = (\nu\partial_0M, \mathcal{A}_1 \cup \mathcal{A}_2)$, with $\mathcal{A}_1 \cup \mathcal{A}_2 \subset \partial_0M \times \{1\}$ two annuli whose cores are meridians of $K_0$. Since $B'$ is obtained by splitting $B$ in a neighbourhood of the boundary components $\partial_1 M, \dots, \partial_n M$, its exterior can be written as
$$ Y \cup X', \text{ with } X'=X\cup J,
$$
where $J$ is a $[0,1]$-bundle with $\partial_v J\cap \partial X \subset \partial_v X$ and such that $\partial J$ meets $\partial X$ so that the fibers agree.
Hence $X'$ is again a connected product sutured manifold and has some ``external'' annular sutures, \ie, those contained in $\partial_1 M \sqcup \cdots \sqcup \partial_n M$. 

Each sector of $B'$ intersects an oriented simple closed curve that is positively transverse to $B'$. In fact, we know that this is true for $B$, and we can suppose that such curves intersect $N_B$ in a union of interval fibers. By the definition of splitting, $B'$ is contained in $\rm{int} (N_B)$ in such a way that the interval fibers of $N_B$  are positively transverse to $B'$. Therefore, the set of positive transversals to $B$ obtained in \Cref{lemma: B is taut} also provides a set of positive transversals through all the sectors of $B'$.

The branched surface $B(r_1, \dots, r_n)$ is obtained by adding to $B'$ some meridional discs of the solid tori glued to $M$. It is hence straightforward that every sector of this branched surface intersects a positive transversal. Moreover, its exterior is obtained by gluing product sutured balls to $X'$ along the external annular sutures. Therefore, it is again union of $Y$ and a product sutured manifold, and so its horizontal boundary is Thurston norm minimising.

\item  \emph{up to splitting leaves, $\mathcal{A}_1, \mathcal{A}_2$ satisfy the noncompact extension property relative to ($B(r_1, \dots, r_n), \mathcal{L}$):}
The horizontal boundary of $Y$ consists of two annuli, one with boundary in $\mathcal{A}_1$ and the other with boundary in $\mathcal{A}_2$. Let $L_0$ and $L_1$ be the leaves containing $\partial_h Y$. By splitting \cite[Operation~2.1.2]{Gabaisus} if necessary, we can suppose that $L_0\ne L_1$.

Any leaf of $\mathcal{L}$ that passes through some sector in $D_i$ must also pass through a sector in $\mathcal{S}'$ intersecting $\partial_i M$. Otherwise, the intersection of this leaf with $\partial_i M$ would be a union of simple closed curves carried by the train track consisting of $\partial D_i$ only, and hence a union of longitudes of $K_i$, contradicting our assumptions. 
In particular the leaf $L_0$ must pass through a sector in $\mathcal{S}'$ that intersects $\partial_1 M$, and, as in the proof of \Cref{lemma: no closed surfaces and annuli}, it must spiral near the torus $T$; the same holds near $\partial_2 M$. This argument shows that each connected component of $L_0\setminus { \rm int}(\partial_h Y)$ is noncompact. Analogously, each connected component of $L_1\setminus { \rm int}(\partial_h Y)$ is noncompact.

Now consider the complementary regions\footnote{That is, the abstract closures of the connected components of the complement of $\mathcal{L}$ in a regular neighbourhood of $B(r_1, \dots, r_n)$.} of $\mathcal{L}$ in a regular neighbourhood of $B(r_1, \dots, r_n)$. These consist of products whose horizontal boundary lies in leaves of $\mathcal{L}$ and whose vertical boundary, if any, is contained in the vertical boundary of $X'\cup Y$. Let $P_j$, for $j=1,2$, denote the complementary region containing $\mathcal{A}_j$, possibly with $P_1=P_2$. Then $\partial_h P_j\subset (L_0\cup L_1)\setminus { \rm int}(\partial_h Y)$. 

If $P_j$ has no vertical boundary on $X'$, it is homeomorphic to one component of $(L_0\setminus { \rm int}(\partial_h Y))\times[0,1]$, which is noncompact, so we can find the required properly embedded copy of $[0,1]\times [0, \infty)$. If instead $P_1$ has vertical boundary on $X'$, recall that $X'$ is a product sutured manifold, say $X'=Q\times [0,1]$, with $Q \times \{0,1\} \subset L_0 \cup L_1$. Without loss of generality, assume that $Q\times \{0\}\subset L_0$ and $Q\times \{1\}\subset L_1$. We then modify $\mathcal{L}$ by isotoping $L_0$ near $Q \times [0,1]$ so that it contains $Q \times \{1\}$, and adding a leaf parallel to $L_0$. This operation (illustrated in \Cref{fig: noncompact}) produces a new lamination $\mathcal{L}'$ still fully carried by $B(r_1, \dots, r_n)$, for which the complementary region $P'_1$ containing $\mathcal{A}_1$ is disjoint from $X'$ and hence noncompact.
\begin{figure}[h]
    \centering
    \includegraphics[width=0.9\textwidth]{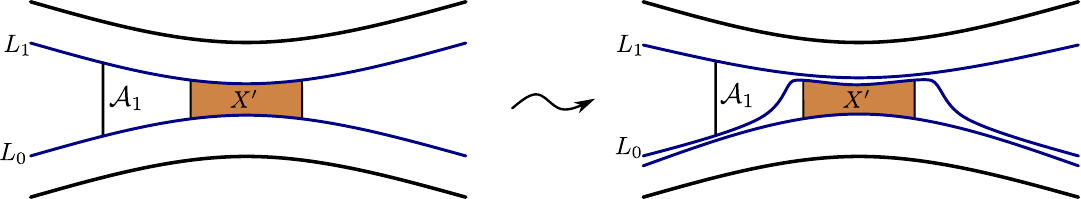}
    \caption{The picture shows how to modify the lamination in a neighbourhood of $X'$ so to achieve the noncompact extension property.}
    \label{fig: noncompact}
\end{figure}

The complementary region $P'_2$ containing $\mathcal{A}_2$ either coincides with $P_2$ (when $P_2$ does not intersect $X'$) or with $P'_1$ otherwise. In both cases, $P_2$ is noncompact, and thus both $\mathcal{A}_1$ and $\mathcal{A}_2$ satisfy the noncompact extension property relative to $(B(r_1, \dots, r_n), \mathcal{L}')$.
\end{itemize}
By \Cref{cusps implies persistently foliar}, it follows that for every non-meridional slope $s$, the lamination $\mathcal{L}$ extends to a taut foliation intersecting $\partial_0 M$ in curves of slope $s$. This completes the proof.
\end{proof}

We are finally ready to prove \Cref{thm: main theorem} when $\mathcal{G}_g\cup \mathcal{G}_r$ has more than $4$ vertices.
\begin{teo}\label{thm: main thm parte 1}
Let $K$ be a knot in $S^3$ with a diagram $D$. Suppose that:
\begin{itemize}
\item All weights of the graphs $\mathcal{G}_r$ and $\mathcal{G}_g$ are greater than one, and at least one weight is greater than two;
\item The graphs $\mathcal{G}_r$ and $\mathcal{G}_g$ are connected.
\end{itemize}
Suppose moreover the graph $\mathcal{G}_r\cup \mathcal{G}_g$ has more than $4$ vertices.
Then $K$ is persistently foliar.
\end{teo}
\begin{proof}
Consider the non-zero integers $k_i$, for $i=1, \dots, n$, such that $M(\cdot, \frac{1}{k_1}, \dots, \frac{1}{k_n})$ is the exterior of $K$. The boundary train track $\partial B$ realises the multislope $(\frac{1}{k_1}, \dots, \frac{1}{k_n})$ by virtue of \Cref{prop: from B to B'}. The result follows from \Cref{thm: general main thm}.
\end{proof}

\section{Relations with Tait graphs} \label{sec: Tait}
In this section we show that \Cref{thm: main theorem} can be rephrased in terms of Tait graphs.
These are a pair of graphs with $\pm 1$-labelled edges in $S^2$ that can be associated to a knot diagram $D$ through the following procedure. Take a checkerboard colouring of the connected components of the complement of the knot projection in the sphere. For convenience, we colour the regions green and red. The Tait graph $G_g(D)$ is defined by taking the green regions of the complement as vertices and assigning one edge for each crossing of $D$, according to the rule illustrated in \Cref{fig: Tait def}. Analogously, using the red regions as vertices yields the graph $G_r(D)$. 

\begin{figure}[h]
    \centering
    \includegraphics[width=0.6\textwidth]{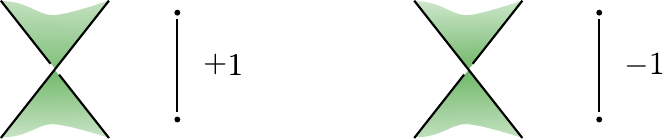}
    \caption{How to assign a weighted edge in the Tait graph to a crossing in the knot diagram.}
    \label{fig: Tait def}
\end{figure}

\Cref{thm: main theorem} will be restated in terms of graphs derived from the Tait graphs of $D$. To this end, we impose an additional assumption on the knot diagrams under consideration. Let $v,w$ two vertices in $G_g(D)$ (resp. $G_r(D)$) and suppose they are connected by multiple edges $e_1,\dots, e_h$, with $h\geq 2$. We say that $e_1,\dots, e_h$ are \emph{parallel} if there exists a disc $D\subset S^2$ that contains all the edges $e_1,\dots, e_h$ and does not intersect the interior of any other edge in $G_g(D)$ (resp. $G_r(D)$).

\begin{assum}\label{assum 2}
The diagram $D$ is such that if $e_1,\dots, e_h$ are the edges connecting a pair of vertices in $G_g$ or in $G_r$, then they are parallel.
\end{assum}

\Cref{assum 2} is not restrictive, and every knot has a diagram satisfying it. In fact, a diagram $D$ does not satisfy it if and only if it contains two distinct twist regions arranged as in the left side of \Cref{fig: twist}. If this happens, we can modify $D$ by twisting, as in the right side of \Cref{fig: twist}, to obtain a diagram with fewer twist regions. By iterating this procedure finitely many times, we obtain a diagram satisfying \Cref{assum 2}. From now on, we will suppose that the knot diagram $D$ satisfies \Cref{assum 1} and \Cref{assum 2}.

\begin{figure}[h]
    \centering
    \includegraphics[width=0.45\textwidth]{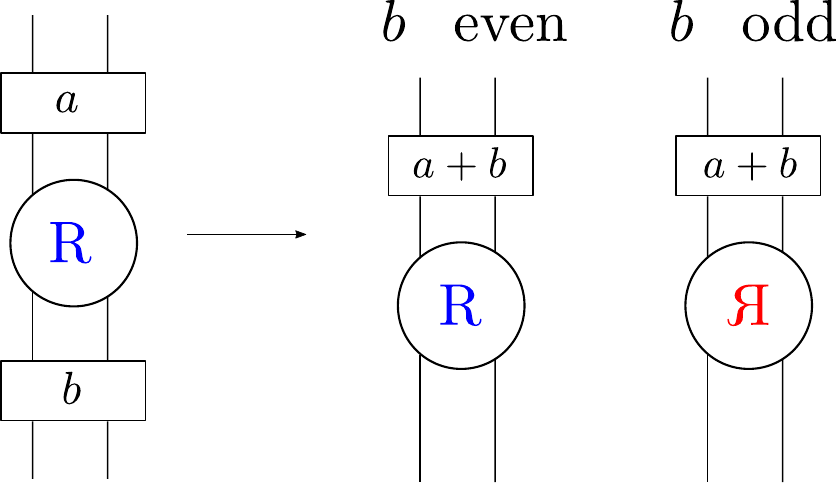}
    \caption{How to modify the diagram if it does not satisfy \Cref{assum 2}. 
    We coloured the front and the back of the letter blue and red, respectively. 
    When $b$ is odd, the back is shown, since the diagram has changed by a $\pi$–rotation.}
    \label{fig: twist}
\end{figure}

Fix one of the two Tait graphs of $D$ and denote it by $G$. 
Consider the maximal subgraph of $G$ whose vertices are exactly the bivalent vertices of $G$, and let 
$C_1,\dots, C_k$ denote its connected components. 
For each $i$, we define the positive integer
\[
c_i=\text{ number of vertices in } C_i + 1.
\]
Let $\mathcal{C}=\{c_1, \dots, c_k\}$. 
Notice that a connected component $C_i$ corresponds to a twist region in $D$ with $c_i$ crossings, except in the special case where $C_i$ is a cycle and $D$ has only one twist region; in that case, it corresponds to a twist region with $c_i-1$ crossings.

Starting from $G$, we define a graph $\mathscr{G}$ with weighted edges according to the following rule:
\begin{enumerate}
\item Consider the maximal subgraph $G'$ of $G$ that does not contain any bivalent vertices; in other words, $G'$ is obtained by removing all bivalent vertices of $G$ and the interior of all edges incident to them.
\item If two vertices $v$ and $w$ of $G'$ are connected by multiple edges $\{e_1, \dots, e_h\}$, collapse them to a single edge $e$ whose weight $w(e)$ is the absolute value of the sum of the labels of $e_1, \dots, e_h$.
\end{enumerate}

\Cref{fig: new graph rule} gives a pictorial description of this two–step construction.  

\begin{figure}[]
    \centering
    \includegraphics[width=0.75\textwidth]{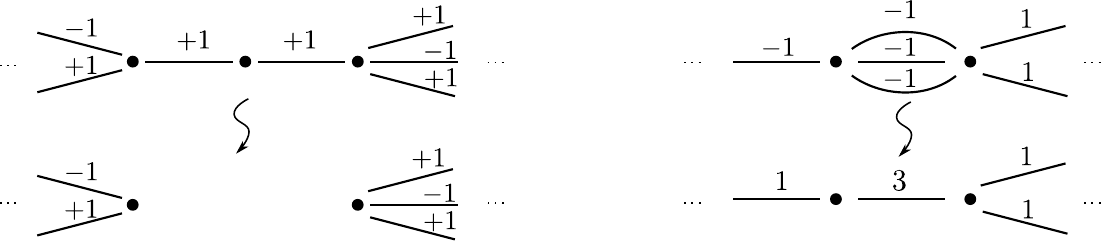}
    \caption{A pictorial description of the two-step costruction of the graph $\mathscr{G}$.}
    \label{fig: new graph rule}
\end{figure}

We denote by $\mathcal{W}$ the set of weights of the graph $\mathscr{G}$, and we are restate \Cref{thm: main theorem} as follows. 

\begin{teo}\label{thm: Tait graphs}
Let $K$ be a knot in $S^3$ with a diagram $D$, and suppose that $D$ has more than one twist region. Assume that:
\begin{itemize}
\item All elements of $\mathcal{C}\cup \mathcal{W}$ are greater than one, and at least one element is greater than two;
\item The graph $\mathscr{G}$ is contractible.
\end{itemize}
Then $K$ is persistently foliar.
\end{teo}

\begin{proof}
For concreteness, suppose that the graph $\mathscr{G}$ is constructed starting from the Tait graph $G_r(D)$, and denote it by $\mathscr{G}_r$. Recall from \Cref{sec: intro} that the graph $\Gamma$ is obtained by replacing each twist region of $D$ with a blue edge. As in \Cref{fig: Tait equiv color}, the checkerboard colouring of the complementary regions of the knot projection, used to define $G_g(D)$ and $G_r(D)$, induces a colouring of the complementary regions of $\Gamma$.

\begin{figure}[]
    \centering
    \includegraphics[width=0.3\textwidth]{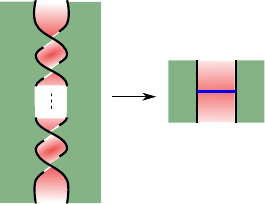}
    \caption{From a checkerboard colouring of the knot projection to a colouring of the complementary regions of $\Gamma$.}
    \label{fig: Tait equiv color}
\end{figure}

Each twist region of $D$ falls into one of two cases: either the bigons in it correspond to vertices in $G_r(D)$, or they do not.
Twist regions of the first type are in bijection with the connected components $C_1,\dots, C_k$ of the maximal subgraph of $G_r(D)$ containing exactly the bivalent vertices. 
Those of the second type, since $D$ satisfies \Cref{assum 2}, are in bijection with pairs of vertices in $G_r(D)$ connected by multiple edges $e_1,\dots, e_h$ with $h \geq 2$.
The graph $\mathcal{G}_r$ can be recovered from $G_r(D)$ by first removing the components $C_1,\dots, C_k$ together with all edges incident to them, and then collapsing all multiple edges between two vertices to a single edge. Since these are precisely the two steps in the definition of $\mathscr{G}_r$, we can identify $\mathscr{G}_r$ with $\mathcal{G}_r$, as illustrated in \Cref{fig: Tait equiv}.

The result follows from the observation that the elements of $\mathcal{C} \cup \mathcal{W}$ coincide with the weights of the graphs $\mathcal{G}_r$ and $\mathcal{G}_g$, together with the fact that $\mathcal{G}_r$ is contractible if and only if both $\mathcal{G}_g$ and $\mathcal{G}_r$ are connected, by \Cref{lemma: grafi connessi sse contrattili}.
\end{proof}

\begin{rem}
We stress that in order to apply \Cref{thm: Tait graphs} it is sufficient that the hypotheses on $\mathcal{C}, \mathcal{W}$ and $\mathcal{G}$ are satisfied for the choice of one of the two Tait graphs.
\end{rem}

\begin{figure}[h]
    \centering
    \includegraphics[width=0.75\textwidth]{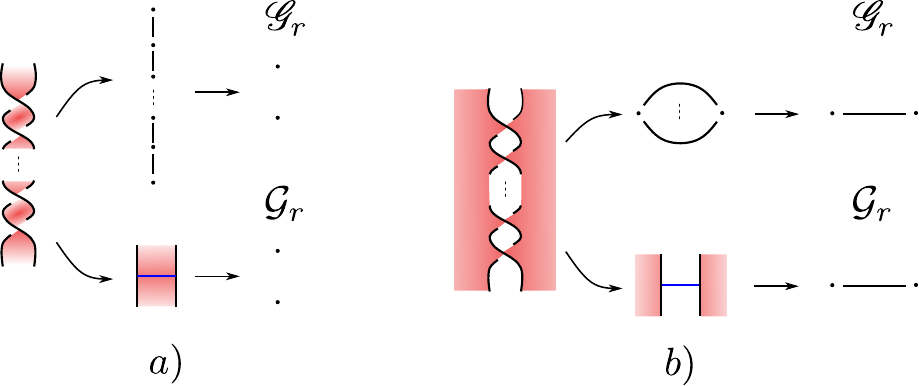}
    \caption{How the graphs $\mathscr{G}_r$ and $\mathcal{G}_r$ are obtained by $G_r(D)$.}
    \label{fig: Tait equiv}
\end{figure}

\section{Applications: arborescent knots and some braid closures}\label{sec: applications}
We conclude the paper with some applications of \Cref{thm: main theorem} to arborescent knots and braid closures.

\subsection{Arborescent knots}
\begin{figure}[]
    \centering
    \includegraphics[width=0.95\textwidth]{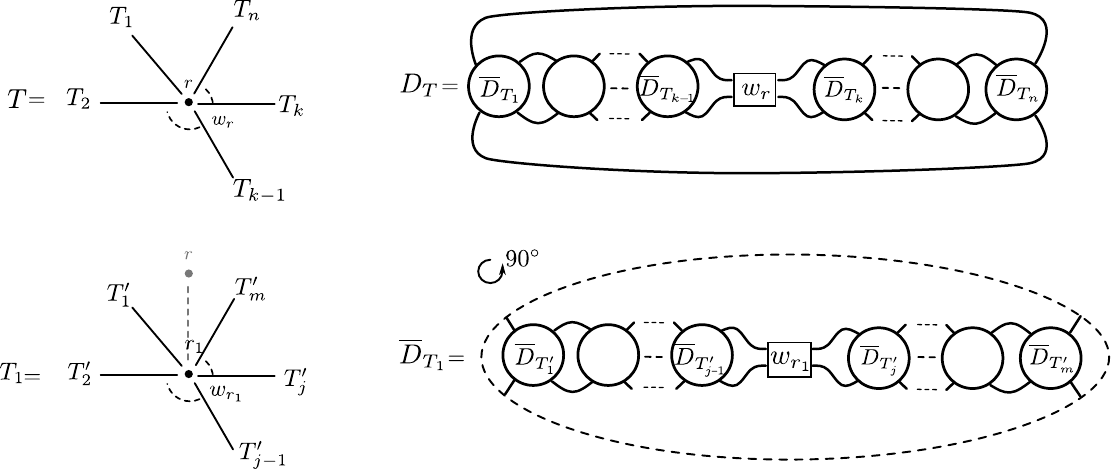}
    \caption{Top: how to associate a diagram $D_T$ to a weighted planar tree endowed with a choice of a root $r$ and an auxiliary angular sector at $r$, indicated in the picture as the one containing the letter $r$.\newline
    Bottom: the tree $T_1$ has a canonical root $r_1$, namely the vertex connected to the vertex $r$ in $T$. It also has a preferred angular sector, namely the one containing the edge connecting $r_1$ to $r$. We associate to $T_1$ the tangle diagram $\overline{D}_{T_1}$ obtained by a $90^{\circ}$ counterclockwise rotation of the diagram shown in the figure. The box with twists is placed in the leftmost (resp. rightmost) part of the diagram when the dotted edge is the right (resp. left) boundary of the sector with the weight $w_{r_1}$.}
    \label{fig: Arborescent recur}
\end{figure}
The family of arborescent tangles is defined as the minimal family of tangles that contains all rational tangles and is closed under both horizontal and vertical tangle composition. Arborescent links are those obtained as closures of arborescent tangles. These links can be described combinatorially using planar weighted trees\footnote{Here, by a planar weighted tree we mean a tree with $\mathbb{Z}$–weighted vertices embedded in $\mathbb{R}^2$, together with the choice of an angular sector at each vertex.}, where each vertex represents a twisted band and the weight gives the number of half-twists (positive for right-handed twists and negative for left-handed ones). Whenever two vertices are joined by an edge, the corresponding bands are plumbed together. In this way, one obtains a possibly non-orientable surface, and the associated link is its boundary. In our drawings, the chosen angular sector at a vertex $v$ will be indicated simply by writing the weight $w_v$ inside it.

Let $T$ be a weighted planar tree endowed with the auxiliary data of a root $r$ --- that is, a distinguished vertex --- and a chosen angular sector at $r$. A diagram in $S^2$ of the arborescent link associated to $T$ can be obtained through the following recursive procedure:
\begin{enumerate}
\item Denote by $T_1,\dots, T_n$ the connected components of the graph obtained from $T$ by removing $r$ and all edges incident to it. We order these components according to the cyclic order induced by the planar embedding of $T$, starting counterclockwise from the auxiliary angular sector. Each $T_i$ is itself a planar weighted tree with a root $r_i$, the unique vertex of $T_i$ that was adjacent to $r$ in $T$, and a preferred angular sector, the one containing the edge joining $r_i$ to $r$. We associate to each $T_i$ the tangle diagram $\overline{D}_{T_i}$ as described in the bottom part of \Cref{fig: Arborescent recur}.
\item We then associate to $T$ the diagram $D_T$ obtained by plugging the diagrams $\overline{D}_{T_i}$ into the diagram shown in the upper part of \Cref{fig: Arborescent recur}.
\end{enumerate}
A concrete example, together with the surface obtained by plumbing twisted bands, is shown in \Cref{fig: Arborescent}. For further details, we refer the reader to \cite{BonSieb}.

Arborescent links generalise two-bridge links, pretzel links and Montesinos links. These are obtained, respectively, by considering trees $a) ,b)$ and $c)$ in \Cref{fig: tree two bridge}.

It follows by \cite{Wu} that non-torus arborescent knots have no reducible surgeries. Moreover, it is a consequence of the results of \cite{LM, BM, LMZ} that the only arborescent knots with non-trivial $L$-space surgeries are, up to mirroring, the pretzel knots $P(-2, 3, q)$ and the torus knots $T(2, q)$, with $q\geq 1$ odd.  Therefore, the $L$-space conjecture predicts that all non-trivial surgeries on any of the remaining arborescent knots contain coorientable taut foliations.
We use \Cref{thm: main theorem} to prove that many arborescent knots are indeed persistently foliar.

\begin{figure}[]
    \centering
    \includegraphics[width=0.75\textwidth]{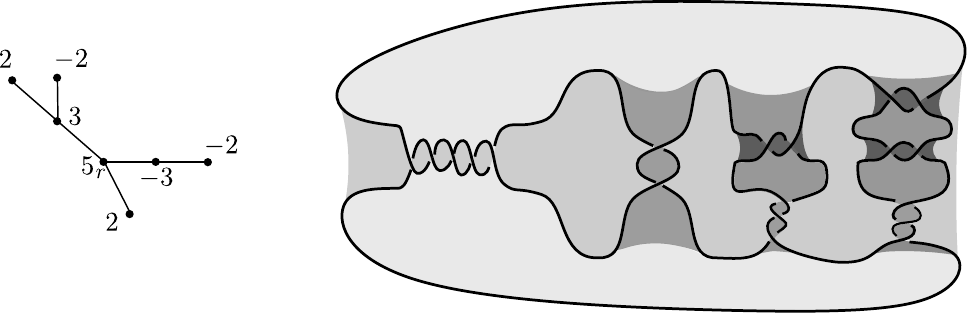}
    \caption{A weighted planar tree, together with a chosen root $r$ and angular sector at $r$, and the associated diagram. We have also shaded the surface obtained by plumbing twisted bands according to the tree.}
    \label{fig: Arborescent}
\end{figure}
\begin{figure}[]
    \centering
    \includegraphics[width=0.75\textwidth]{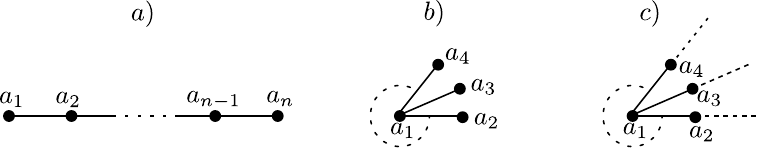}
    \caption{Weighted planar trees associated to two-bridge, pretzel and Montesinos links.}
    \label{fig: tree two bridge}
\end{figure}

\begin{lemma}\label{lemma: arborescent graph}
Let $T$ be a weighted planar tree whose weights all have absolute value greater than one and let $D_T$ be the diagram of the link associated to $T$ (with respect to any root and choice of auxiliary angular sector at it). Then the graphs $\mathcal{G}_g$ and $\mathcal{G}_r$ are contractible.
\end{lemma}

\begin{proof}
We prove the lemma by induction on the number of vertices of $T$. If $T$ has only one vertex, then the diagram $D$ is the standard diagram of the $T(2,q)$-torus link, viewed as the closure of a $2$-braid. In this case, one of the graphs consists of a single vertex, while the other has two vertices connected by an edge.

For the inductive step, suppose that $D_T$ is associated to the weighted planar tree $T$, endowed with a root $r$ and a chosen angular sector at $r$, and consider the graphs $\mathcal{G}_g$ and $\mathcal{G}_r$ associated to $D_T$. Notice that, since the weights of $T$ have absolute value greater than one, there are no twist regions in $D_T$ made up by only one crossing, and so the graphs $\mathcal{G}_g$ and $\mathcal{G}_r$ are uniquely defined. We fix the convention that the graph $\mathcal{G}_r$ is the one containing the infinite region in the plane of the diagram. 
Removing the root $r$ and all edges incident to it yields a graph with $n$ connected components $T_1, \dots, T_n$, each of which is again a weighted planar tree. As in \Cref{fig: Arborescent recur}, each $T_i$ has a canonical root $r_i$ and angular sector at $r_i$. Let $D_{T_i}$ be the corresponding diagram, and denote by $\mathcal{G}^i_r$ and $\mathcal{G}^i_g$ the associated graphs. By our convention, $\mathcal{G}^i_r$ is always chosen to contain the infinite region of the diagram. By the inductive hypothesis, all these graphs are contractible. 

The graph $\mathcal{G}_g$ is contractible since it is obtained in the following way:

\begin{itemize}
\item For each $i$, a vertex $v^i$ in $\mathcal{G}^i_r$ is replaced by two vertices, $v^i_1$ and $v^i_2$, and each edge that was incident to it is reassigned either to $v^i_1$ or $v^i_2$. In this way, we obtain two contractible graphs $\mathcal{G}^i_1$ and $\mathcal{G}^i_2$, containing  $v^i_1$ and $v^i_2$ respectively.  
\item The disjoint union $\mathcal{G}^1_1\sqcup\cdots \sqcup \mathcal{G}^n_1$ is then quotiented by identifying all the vertices $v^i_1$ to a single vertex $v_1$, producing a contractible graph $\mathcal{G}_1$; An analogous operation on the $\mathcal{G}^i_2$ yields another contractible graph $\mathcal{G}_2$ with vertex $v_2$.
\item Finally, the graphs $\mathcal{G}_1$ and $\mathcal{G}_2$ are connected by adding an edge joining $v_1$ to $v_2$.
\end{itemize}
An example of this procedure is shown in \Cref{fig: arborescent graph}.  By \Cref{lemma: grafi connessi sse contrattili}, the graph $\mathcal{G}_r$ is contractible if and only if $\mathcal{G}_g$ is, and this concludes the proof.
\end{proof}
\begin{figure}[h]
    \centering
    \includegraphics[width=0.85\textwidth]{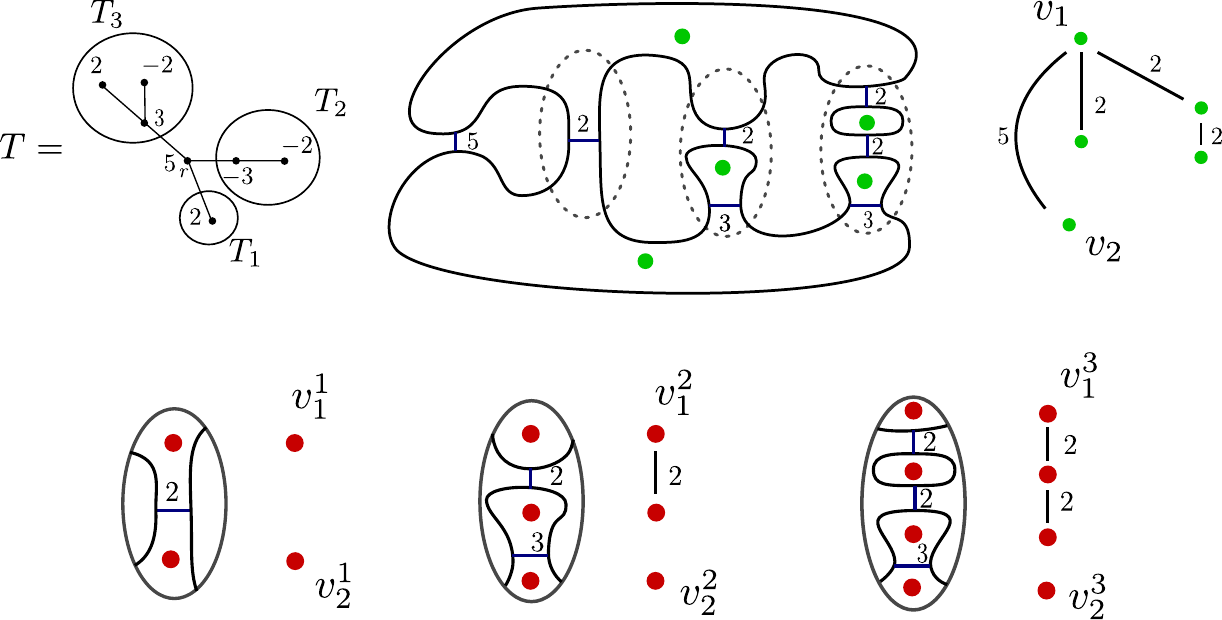}
    \caption{Top: The graphs $\Gamma$ and $\mathcal{G}_g$ associated with the diagram $D_T$ of the tree $T$.
Bottom: The graphs $\mathcal{G}^i_1$ and $\mathcal{G}^i_2$ associated with the tree $T_i$, for $i = 1, 2, 3$. The graph $\mathcal{G}_g$ is obtained by identifying all the vertices $v^i_1$ as a single vertex $v_1$, and all the vertices $v^i_2$ as a single vertex $v_2$, and then connecting these two vertices with an edge.}
    \label{fig: arborescent graph}
\end{figure}

\begin{teo}\label{thm: arborescent}
Let $K$ be an arborescent knot associated to a weighted planar tree $T$ with more than one vertex. If all vertex weights have absolute value greater than one, and at least one weight has absolute value greater than two, then $K$ is persistently foliar.
\end{teo}

\begin{proof}
Consider the diagram $D_T$ associated to $T$, with respect to any root and the choice of any auxiliary angular sector at it. Since $T$ has more than one vertex, the diagram $D_T$ has more than one twist region. Since the weights of the tree $T$ coincide with the weights of the graphs $\mathcal{G}_g$ and $\mathcal{G}_r$ and both graphs are contractible by \Cref{lemma: arborescent graph}, we can conclude by applying \Cref{thm: main theorem}.
\end{proof}

\subsection{Some braid closures}
We conclude the section with some applications to braid closures. Let $B_n$ denote the braid group on $n$ strands with $\sigma_1,\dots, \sigma_{n-1}$ the Artin generators of $B_n$. We denote the closure of a braid $\beta$ by $\hat{\beta}$.

\begin{defn}
A word $w=\sigma_{i_1}^{a_1}\sigma_{i_2}^{a_2}\cdots \sigma_{i_k}^{a_k}$ in the Artin generators is \emph{reduced} when $\sigma_{i_{j+1}}\ne \sigma_{i_j}$ for $j=1, \dots, k-1$ and $\sigma_{i_k}\ne \sigma_{i_1}.$
\end{defn}

\begin{defn}
Let $w=\sigma_{i_1}^{a_1}\sigma_{i_2}^{a_2}\cdots \sigma_{i_k}^{a_k}$ be a reduced word in $B_n$. For a fixed index $h\in \{1,\dots, n-1\}$, we say that any $\sigma_{i_j}^{a_j}$ with ${i_j}=h$ is an \emph{occurrence} of $\sigma_{h}$ in $w$. Denote by $I_h$ the set of indices $j_1< \dots < j_m$ such that $\sigma_{i_j}=\sigma_h$. We order the elements of $I_h$ and extend this to a cyclic order, so that $j_1$ is considered consecutive to $j_m$. We say that two occurrences $\sigma_{i_j}^{a_j}$ and $\sigma_{i_{j'}}^{a_{j'}}$ of $\sigma_h$ in $w$ are \emph{consecutive} if $j$ and $j'$ are consecutive as elements in $I_h$.
\end{defn}

The following theorem states that if a knot $K$ can be written as the closure of a braid of a particular type, then $K$ is persistently foliar. In particular, $K$ is not an $L$-space knot and  has no reducible surgeries.

\begin{teo}\label{thm: braids}Let $n\geq 3$ be an odd positive integer, and let $K=\hat{\beta}$, where 
\[
\beta = \sigma_{i_1}^{a_1} \sigma_{i_2}^{a_2} \cdots \sigma_{i_k}^{a_k} \in B_n
\] 
is a reduced word, each exponent satisfies $|a_j| \ge 2$, and for at least one index $j_0$, we have $|a_{j_0}| > 2$. Suppose further that the following conditions hold: 
\begin{itemize}
\item For every odd index $h\geq 1$, between any two consecutive occurrences of $\sigma_h$, there is at least one occurrence of $\sigma_{h+1}$. 
\item For every even index $h\geq 2$, between any two consecutive occurrences of $\sigma_h$, there is at least one occurrence of $\sigma_{h-1}$. 
\end{itemize}
Then $K$ is persistently foliar.
\end{teo}

\begin{proof}
Consider the diagram of $\beta$ in $D^2\subset S^2$ given by the word $\sigma_{i_1}^{a_1}\sigma_{i_2}^{a_2}\cdots \sigma_{i_k}^{a_k}$, and take the diagram $D$ of $K$ obtained by identifying the endpoints of this diagram in $S^2$ with $n$ embedded arcs, as usual.
The absolute values $|a_i|$ of the exponents coincide with the weights of the graphs $\mathcal{G}_r$ and $\mathcal{G}_g$ associated to $D$, and by hypothesis these are all greater than one, with at least one greater two. Therefore, to prove the desired result, it suffices to show that both $\mathcal{G}_r$ and $\mathcal{G}_g$ are connected. To this end, consider the graph $\Gamma\subset S^2$ constructed as described in \Cref{sec: intro}. Removing the blue arcs from $\Gamma$ yields $n$ parallel circles that decompose $S^2$ into two discs $D_0, D_n$ and $n-1$ annuli $A_1, \dots, A_{n-1}$, where $D_0$ is adjacent to $A_1$, and $D_n$ is
adjacent to $A_{n-1}$.
Each annulus $A_i$ is further divided by the blue arcs into $m_i$ discs, where $m_i$ is the number of occurrences of $\sigma_i$ in the word defining $\beta$. Without loss of generality, assume that $\mathcal{G}_g$ is the graph containing $D_n$ as a vertex. Then the vertices of $\mathcal{G}_g$ correspond to $D_n$ and all the complementary regions of $\Gamma$ contained in $A_1,\dots, A_h,\dots, A_{n-2}$, where the indices $h$ range over the odd integers between $1$ and $n-1$. See \Cref{fig: braid} for an example.
We have that:
\begin{itemize}
\item \emph{Every vertex in $A_{n-2}$ is connected to the vertex $D_n$:} this follows from the hypothesis that between any two occurrences of $\sigma_{n-2}$ there is one occurrence of $\sigma_{n-1}$;
\item \emph{Every vertex in $A_{h}$ is connected to some vertex in $A_{h+2}$, for odd $h$ with $1\leq h< n-2$:} this follows from the assumption that between any two occurrences of $\sigma_{h}$ there is at least one occurrence of $\sigma_{h+1}$, for all odd $h\geq 1$. 
\end{itemize}
This implies, by induction, that every vertex in $\mathcal{G}_g$ is connected by a path to the vertex $D_n$ and hence the graph is connected. In the same way, by using that between any two consecutive occurrences of $\sigma_h$ there is at least one occurrence of $\sigma_{h-1}$, for all even $h\geq 2$, one proves that $\mathcal{G}_r$ is connected. By \Cref{thm: main theorem}, we obtain the desired result.
\end{proof}

\begin{figure}[h]
    \centering
    \includegraphics[width=0.75\textwidth]{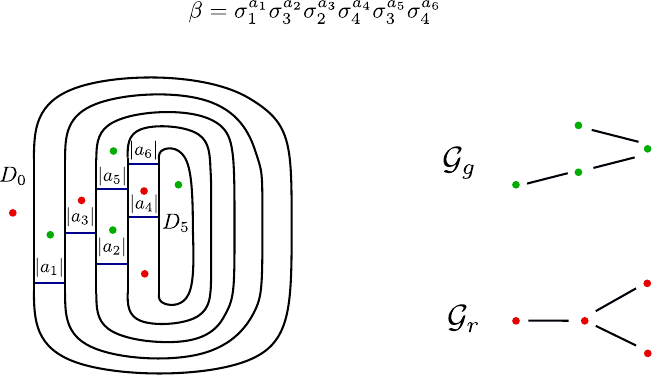}
    \caption{An illustration showing the graphs $\Gamma$, $\mathcal{G}_r$ and $\mathcal{G}_g$ associated to the diagram of a knot given as the closure of a braid.}
    \label{fig: braid}
\end{figure}
In the case of $3$-braid closures, the statement of \Cref{thm: braids} specialises to the following.

\begin{cor}\label{cor: 3-braids}
    
Let $K$ be the closure of a braid 
\[
\beta = \sigma_1^{a_1}\sigma_2^{a_2}\sigma_1^{a_3}\cdots \sigma_2^{a_k} \in B_3,
\] 
with $|a_i|\ge 2$ for all $i$, and $|a_{i_0}|>2$ for some index $i_0$. Then $K$ is persistently foliar.
\end{cor}

\bibliographystyle{alpha}
\bibliography{biblio}
\end{document}